%% file: Accelerated_RPDC_final.tex
\documentclass[11pt]{article}

\usepackage{epsfig,epsf,fancybox}
\usepackage{amsmath}
\usepackage{mathrsfs}
\usepackage{amssymb}
\usepackage{graphicx}
\usepackage{color}
\usepackage[linktocpage,pagebackref,colorlinks,linkcolor=blue,anchorcolor=blue,citecolor=blue,urlcolor=blue,hypertexnames=false]{hyperref}
\usepackage{boxedminipage}
\usepackage{stmaryrd}
\usepackage{multirow}
\usepackage{booktabs}
\usepackage{accents}
\usepackage{cite}
\usepackage{float}
\usepackage[ruled,vlined,linesnumbered]{algorithm2e}

\newtheorem{theorem}{Theorem}[section]

\newtheorem{lemma}[theorem]{Lemma}
\newtheorem{proposition}[theorem]{Proposition}

\,
\,

\newcommand{\be}{\begin{equation}}
\newcommand{\ee}{\end{equation}}
\newcommand{\ba}{\begin{array}}
\newcommand{\ea}{\end{array}}
\newcommand{\bpm}{\begin{pmatrix}}
\newcommand{\epm}{\end{pmatrix}}

\newcommand{\bea}{\begin{eqnarray}}
\newcommand{\eea}{\end{eqnarray}}
\newcommand{\beaa}{\begin{eqnarray*}}
\newcommand{\eeaa}{\end{eqnarray*}}

\newcommand{\bal}{\begin{align}}
\newcommand{\eal}{\end{align}}
\newcommand{\baln}{\begin{align*}}
\newcommand{\ealn}{\end{align*}}

\input{macros.tex}
\begin{document}

%\title{Accelerated primal-dual proximal block coordinate update}

\title{Accelerated Primal-Dual Proximal Block Coordinate Updating Methods for Constrained Convex Optimization\thanks{This work is partly supported by NSF grant DMS-1719549 and CMMI-1462408.}}

\author{Yangyang Xu\thanks{xuy21@rpi.edu. Department of Mathematical Sciences, Rensselaer Polytechnic Institute} \and Shuzhong Zhang\thanks{zhangs@umn.edu. Department of Industrial \& Systems Engineering, University of Minnesota}
}

\date{}

\maketitle

\begin{abstract}
Block Coordinate Update (BCU) methods enjoy low per-update computational complexity because every time
only one or a few block variables would need to be updated among possibly a large
number of blocks.
They are also easily parallelized and thus have been particularly popular for solving problems involving large-scale dataset and/or variables.
In this paper, we propose a primal-dual BCU method for solving linearly constrained convex program in multi-block variables. The method is an accelerated version of a
primal-dual algorithm proposed by the authors, which applies randomization in selecting block variables to update and establishes an $O(1/t)$ convergence rate under convexity assumption. We show that the rate can be accelerated to $O(1/t^2)$ if the objective is strongly convex. In addition, if one block variable is independent of the others in the objective, we then show that the algorithm can be modified to
achieve a linear rate of convergence.
The numerical experiments
show that the accelerated method performs stably with a single set of parameters while the original
method needs to tune the parameters for different datasets in order to achieve a comparable level of performance.

\vspace{0.5cm}


\vspace{0.1cm}


\end{abstract}

\section{Introduction}
Motivated by the need to solve large-scale optimization problems and increasing capabilities in parallel computing, block coordinate update (BCU) methods have become particularly popular in recent years due to their low per-update computational complexity, low memory requirements, and their potentials in a distributive computing environment. %parallelized computing.
%For solving optimization problems,
In the context of optimization,
BCU first appeared in the form of block coordinate descent (BCD) type of algorithms  which can be applied to solve unconstrained smooth problems or those with separable nonsmooth terms in the objective (possibly with separable constraints). More recently, it has been developed for solving problems with nonseparable nonsmooth terms and/or constraint in a primal-dual framework.

In this paper, we consider the following linearly constrained multi-block structured optimization model: %problems
\begin{equation}\label{eq:mb-prob}
\min_x f(x)+\sum_{i=1}^M g_i(x_i), \st \sum_{i=1}^M A_i x_i =b,
\end{equation}
where $x$ is partitioned into disjoint blocks $(x_1,x_2,\ldots,x_M)$, $f$ is a smooth convex function with Lipschitz continuous gradient, %convex Lipschitz differentiable function,
and each $g_i$ is proper closed convex and possibly non-differentiable. Note that $g_i$ can include an indicator function of a convex set $\cX_i$, and thus \eqref{eq:mb-prob} can implicitly include certain separable block constraints in addition to the nonseparable linear constraint.

Many applications arising in statistical and machine learning, image processing, and finance can be formulated in the form of \eqref{eq:mb-prob} including the basis pursuit \cite{chen2001atomic}, constrained regression \cite{james2013pcreg},  support vector machine in its dual form \cite{cortes1995svm}, portfolio optimization \cite{markowitz1952portfolio}, just to name a few.

Towards finding a solution for \eqref{eq:mb-prob}, we will first present an accelerated proximal Jacobian alternating direction method of multipliers (Algorithm \ref{alg:ajadmm}), and then we generalize it to an accelerated {\it randomized}\/ primal-dual block coordinate update method (Algorithm \ref{alg:arpdc}). Assuming strong convexity on the objective function, we will establish $O(1/t^2)$ convergence rate results of the proposed algorithms by adaptively setting the parameters, where $t$ is the total number of iterations. In addition, if further assuming smoothness and the full-rankness we then obtain linear convergence of a modified method (Algorithm \ref{alg:rpdc-lin}).

\subsection{Related methods}
Our algorithms are closely related to randomized coordinate descent methods, primal-dual coordinate update methods, and accelerated primal-dual methods. In this subsection, let us briefly review the three classes of methods and discuss their relations to our algorithms.

\subsubsection*{Randomized coordinate descent methods}
In the absence of linear constraint, Algorithm \ref{alg:arpdc} specializes to randomized coordinate descent (RCD), which was first proposed in \cite{nesterov2012rcd} for smooth problems and later generalized in \cite{richtarik2014iteration, Lu_Xiao_rbcd_2015} to nonsmooth problems. It was shown that RCD converges sublinearly with rate $O(1/t)$, which can be accelerated to $O(1/t^2)$ for convex problems and achieves a linear rate for strongly convex problems. By choosing multiple block variables at each iteration, \cite{richtarik2012parallel} proposed to parallelize the RCD method and showed the same convergence results for parallelized RCD. This is similar to setting $m>1$ in Algorithm \ref{alg:arpdc}, allowing parallel updates on the selected $x$-blocks.

\subsubsection*{Primal-dual coordinate update methods}
%With the existence of nonseparable
In the presence of linear constraints, coordinate descent methods may fail to converge to a solution of the problem because fixing all but one block, the selected block variable may be uniquely determined by the linear constraint. To perform coordinate update to the linearly constrained problem \eqref{eq:mb-prob}, one effective approach is to update both primal and dual variables. Under this framework, %Among this kind of methods,
the alternating direction method of multipliers (ADMM) is one popular choice. Originally, ADMM \cite{Glowinski1975, gabay1976dual} was proposed for solving two-block structured problems with separable objective (by setting $f=0$ and $M=2$ in \eqref{eq:mb-prob}), for which its convergence and also convergence rate have been well-established (see e.g.\ \cite{boley2013local, deng2012global, HeY12-adm, monteiro2010iteration}). However, directly extending ADMM to the multi-block setting such as \eqref{eq:mb-prob}
may fail to converge; %, ADMM may diverge
see \cite{chen2016direct} for a divergence example of the ADMM even for solving a linear system of equations. Lots of efforts have been spent on establishing the convergence of multi-block ADMM under stronger assumptions (see e.g.~\cite{cai2014directstrong, chen2016direct, li2015convergent, lin2015global, gao2015first}) such as strong convexity or orthogonality conditions on the linear constraint. Without additional assumptions, modification is necessary for the ADMM applied to multi-block problems to be convergent; see \cite{deng2013parallel, He-Hou-Yuan, he2012alternating, xu2016hybrid} for example. Very recently, \cite{GXZ-RPDCU2016} proposed a randomized primal-dual coordinate (RPDC) update method, whose asynchronous parallel version was then studied in \cite{xu2017async-pdc}. Applied to \eqref{eq:mb-prob}, RPDC is a special case of Algorithm \ref{alg:arpdc} with fixed parameters. It was shown that RPDC converges with rate $O(1/t)$ under convexity assumption. More general than solving an optimization problem, primal-dual coordinate (PDC) update methods have also appeared in solving fixed-point or monotone inclusion problems \cite{pesquet2014class, peng2016arock, combettes2015stochastic, peng2016cf}. However, for these problems, the PDC methods are only shown to converge but no convergence rate estimates are known unless additional assumptions are made such as the strong monotonicity condition.

\subsubsection*{Accelerated primal-dual methods}
%On solving convex programs, acceleration techniques can be used to improve a method with $O(1/t)$ convergence rate to one with $O(1/t^2)$ rate.
It is possible to accelerate the rate of convergence from $O(1/t)$ to $O(1/t^2)$ for gradient type methods. The first acceleration result was shown by Nesterov~\cite{nesterov1983method} for solving smooth unconstrained problems.
%The standard gradient method converges at $O(1/t)$ rate while the accelerated one with $O(1/t^2)$ rate.
The technique has been generalized to accelerate gradient-type methods on possibly nonsmooth convex programs \cite{FISTA2009, nesterov2013gradient}. Primal-dual methods on solving linearly constrained problems can also be accelerated by similar techniques. Under convexity assumption, the augmented Lagrangian method (ALM) is accelerated in \cite{he2010aalm} from $O(1/t)$ convergence rate to $O(1/t^2)$ by using a similar technique as that in \cite{FISTA2009} to the multiplier update, and \cite{xu2017accelerated-alm} accelerates the linearized ALM using a technique similar to that in \cite{nesterov2013gradient}. Assuming strong convexity on the objective, \cite{goldstein2014fast} accelerates the ADMM method, and the assumption is weakened in \cite{xu2017accelerated-alm} to assuming the strong convexity
for one component of the objective %of one component
function. On solving bilinear saddle-point problems, various primal-dual methods can be accelerated if either primal or dual problem is strongly convex \cite{chambolle2011first, dang2014randomized, bredies2016accelerated}. Without strong convexity, partial acceleration is still possible in terms of the rate depending on some other quantities; see e.g.~\cite{ouyang2015accelerated, chen2014optimal}.

\subsection{Contributions of this paper}
We accelerate the proximal Jacobian ADMM \cite{deng2013parallel} and also generalize it to an accelerated primal-dual coordinate updating method for linearly constrained multi-block structured convex program, where in the objective there is a nonseparable smooth function. With parameters fixed during all iterations, the generalized method reduces to that in \cite{GXZ-RPDCU2016} and enjoys $O(1/t)$ convergence rate under mere convexity assumption. By adaptively setting the parameters at different iterations, we show that the accelerated method has $O(1/t^2)$ convergence rate if the objective is strongly convex. In addition, if there is one block variable that is independent of all others in the objective (but coupled in the linear constraint) and also the corresponding component function is smooth, we modify the algorithm by treating that independent variable in a different way and establish a linear convergence result. Numerically, we test the accelerated method on quadratic programming and compare it to the (nonaccelerated) RPDC method in \cite{GXZ-RPDCU2016}. The results demonstrate that the accelerated method performs efficiently and stably
%with a single set of parameters provided by our analysis
with the parameters automatically set in accordance of the analysis,
while the RPDC method needs to tune its parameters for different data in order to have a comparable performance.

\subsection{Nomenclature and basic facts} %preliminary results}

\textbf{Notations.} {For a positive integer $M$, we denote $[M]$ as $\{1,\ldots,M\}$.} We let $x_S$ denote the subvector of $x$ with blocks indexed by $S$. Namely, if $S=\{i_1, \ldots, i_m\}$, then $x_S=(x_{i_1},\ldots, x_{i_m})$. Similarly, $A_S$ denotes the submatrix of $A$ with columns indexed by $S$, and $g_S$ denotes the sum of component functions indicated by $S$. {We use $\nabla_i f(x)$ for the partial gradient of $f$ with respect to $x_i$ at $x$ and $\nabla_S f(x)$ with respect to $x_S$. For a nondifferentiable function $g$, $\tilde{\nabla} g(x)$ denotes a subgradient of $g$ at $x$.} We reserve $I$ for the identity matrix and use $\|\cdot\|$ for Euclidean norm. Given a symmetric positive semidefinite (PSD) matrix $W$, for any vector $v$ of appropriate size, we define $\|v\|_W^2=v^\top W v$, {and
\begin{equation}\label{eq:def-Delta}
\Delta_W(v^+,v^o,v)=\frac{1}{2}\big[\|v^+-v\|_W^2-\|v^o-v\|_W^2+\|v^+-v^o\|_W^2\big].
\end{equation}
If $W=I$, we simply use $\Delta(v^+,v^o,v)$.
}
Also, we denote
\begin{equation}\label{eq:nota-Phi}
{ g(x)=\sum_{i=1}^m g_i(x_i),}\quad F(x)=f(x)+g(x),\quad\Phi(\hat{x},x,\lambda)=F(\hat{x})-F(x)-\langle\lambda, A\hat{x}-b\rangle.
\end{equation}

\textbf{Preparations.} %Since there are only linear constraints in \eqref{eq:mb-prob}, a point $x^*$ is a solution to \eqref{eq:mb-prob} \emph{if and only if} the Karush-Kuhn-Tucker (KKT) conditions hold, i.e., there exists a $\lambda^*$ such that
{ A point $(x^*,\lambda^*)$ is called a Karush-Kuhn-Tucker (KKT) point of \eqref{eq:mb-prob} if
}
\begin{equation}\label{eq:kkt-conds}
%\begin{align}
0\in\partial F(x^*)-A^\top \lambda^*,\quad Ax^*-b=0.
%\end{align}
\end{equation}
{ For convex programs, the conditions in \eqref{eq:kkt-conds} are sufficient for $x^*$ to be an optimal solution of \eqref{eq:mb-prob}, and they are also necessary if a certain qualification condition holds (e.g., the Slater condition: there is $x$ in the interior of the domain of $F$ such that $Ax=b$).
}
Together with the convexity of $F$, \eqref{eq:kkt-conds} implies
\begin{equation}\label{eq:1stopt-cond}
\Phi(x,x^*,\lambda^*)\ge 0,\,\forall x.
\end{equation}
%where $\Phi$ is defined in \eqref{eq:nota-Phi}.

We will use the following lemmas as basic facts. { The first lemma is straightforward to verify from the definition of $\|\cdot\|_W$; the second one is similar to Lemma 3.3 in \cite{GXZ-RPDCU2016}; the third one is from Lemma 3.5 in \cite{GXZ-RPDCU2016}.} 

\begin{lemma} For any vectors $u, v$ and symmetric PSD matrix $W$ of appropriate sizes, it holds that
\begin{equation}\label{uv-cross}
u^\top W v = \frac{1}{2}\left[\|u\|_W^2-\|u-v\|_W^2+\|v\|_W^2\right].
\end{equation}
\end{lemma}

\begin{lemma}\label{lem:xy-rate}
Given a function $\phi$, for a given $x$ and a random vector $\hat{x}$, if for any $\vlam$ (that may depend on $\hat{x}$)
it holds
$\EE\Phi(\hat{x},x,\lambda)\le \EE \phi(\lambda),$
then for any $\gamma>0$, we have
$$\EE\big[F(\hat{x})-F(x)+\gamma\|A\hat{x}-b\|\big]\le \sup_{\|\vlam\|\le \gamma}\phi(\vlam).$$
\end{lemma}

\begin{proof}
Let $\hat{\lambda}=-\frac{\gamma(A\hat{x}-b)}{\|A\hat{x}-b\|}$ if $A\hat{x}-b\neq 0$, and $\hat{\lambda}=0$ otherwise. Then
$$\Phi(\hat{x},x,\hat{\lambda})=F(\hat{x})-F(x)+\gamma \|A\hat{x}-b\|.$$
In addition, since $\|\hat{\lambda}\|\le \gamma$, we have $\phi(\hat{\lambda})\le \sup_{\|\lambda\|\le \gamma}\phi(\lambda)$ and thus $\EE \phi(\hat{\lambda})\le \sup_{\|\lambda\|\le \gamma}\phi(\lambda)$. Hence, we have the desired result from $\EE\Phi(\hat{x},x,\hat{\lambda})\le \EE \phi(\hat{\lambda})$.
\end{proof}

\begin{lemma} \label{lem:equiv-rate}
Suppose $\EE\big[F(\hat{x})-F(x^*)+\gamma\|A\hat{x}-b\|\big] \le \epsilon.$
Then,
\[
\EE\|A\hat{x}-b\|\leq \frac{\epsilon}{\gamma-\|\lambda^*\|}, \mbox{ and }  -\frac{\epsilon\|\lambda^*\|}{\gamma-\|\lambda^*\|}\le\EE\big[F(\hat{x})-F(x^*)\big] \leq \epsilon,
\]
where $(x^*,\lambda^*)$ satisfies the optimality conditions in \eqref{eq:kkt-conds}, and we assume $\|\lambda^*\|< \gamma$.
\end{lemma}

\textbf{Outline.} The rest of the paper is organized as follows. Section \ref{sec:ajadmm} presents the accelerated proximal Jacobian ADMM and its convergence results. In section \ref{sec:arpdc}, we propose an accelerated primal-dual block coordinate update method with convergence analysis. %analyzes the convergence of Algorithm \ref{alg:arpdc} and shows its $O(1/t^2)$ convergence rate under strong convexity assumption.
Section \ref{sec:linear} assumes more structure on the problem \eqref{eq:mb-prob} and modifies the algorithm in section \ref{sec:arpdc} to have linear convergence. Numerical results are provided in section \ref{sec:numerical}. Finally, section \ref{sec:conclusion} concludes the paper.

\section{Accelerated proximal Jacobian ADMM}\label{sec:ajadmm}
In this section, we propose an accelerated proximal Jacobian ADMM for solving \eqref{eq:mb-prob}. At each iteration, the algorithm updates all $M$ block variables in parallel by minimizing a linearized proximal approximation of the augmented Lagrangian function, and then it renews the multiplier. Specifically, it iteratively performs the following updates:
\begin{subequations}\label{eq:ajadmm}
\begin{align}
&x_i^{k+1}=\argmin_{x_i} \left\langle \nabla_i f(x^k)-A_i^\top(\lambda^k-\beta_k r^k), x_i\right\rangle + g_i(x_i) + \frac{1}{2}\|x_i-x_i^k\|_{P_i^k},\, i=1,\ldots, M,\label{eq:ajadmm-x}\\
&\lambda^{k+1}=\lambda^k-\rho_k r^{k+1},\label{eq:ajadmm-lam}
\end{align}
\end{subequations}
where $\beta_k$ and $\rho_k$ are scalar parameters, $P^k$ is an $M\times M$ block diagonal matrix with $P_i^k$ as its $i$-th diagonal block for $i=1,\ldots,M$, and $r^k=Ax^k-b$ denotes the residual. Note that \eqref{eq:ajadmm-x} consists of $M$ independent subproblems, and they can be solved in parallel.

Algorithm \ref{alg:ajadmm} summarizes the proposed method. It reduces to the proximal Jacobian ADMM in \cite{deng2013parallel} if $\beta_k,\rho_k$ and $P^k$ are fixed for all $k$ and there is no nonseparable function $f$. We will show that adapting the parameters as the iteration progresses can accelerate the convergence of the algorithm.

\begin{algorithm}[H]\caption{Accelerated proximal Jacobian ADMM for \eqref{eq:mb-prob}}\label{alg:ajadmm}
\DontPrintSemicolon
\textbf{Initialization:} choose $x^1$, set $\lambda^1=0$, and let $r^1=Ax^1-b$\;
\For{$k=1,2,\ldots$}{
Choose parameters $\beta_k, \rho_k$ and a block diagonal matrix $P^k$\;
Let $x^{k+1}\gets$ \eqref{eq:ajadmm-x} and $\lambda^{k+1}\gets$ \eqref{eq:ajadmm-lam} with $r^{k+1}=Ax^{k+1}-b$.\;
\If{a certain stopping criterion satisfied}{
Return $(x^{k+1},\lambda^{k+1})$.
}
}
\end{algorithm}

\subsection{Technical assumptions}
Throughout the analysis in this section, we make the following assumptions.
\begin{assumption}\label{assump:saddle-pt}
There exists $(x^*,\lambda^*)$ satisfying the KKT conditions in \eqref{eq:kkt-conds}.
\end{assumption}

\begin{assumption}\label{assump:full-lip-F}
$\nabla f$ is Lipschitz continuous with modulus $L_f$.
\end{assumption}

\begin{assumption}\label{assump:str-cvx-F}
{The function $g$ is strongly convex with modulus $\mu>0$.
}
\end{assumption}

The first two assumptions are standard, and the third one is for showing convergence rate of $O(1/t^2)$, where $t$ is the number of iterations. { Note that if $f$ is strongly convex with modulus $\mu_f>0$, we can let $f\gets f-\frac{\mu_f}{2}\|\cdot\|^2$ and $g\gets g+\frac{\mu_f}{2}\|\cdot\|^2$. This way, we have a convex function $f$ and a strongly convex function $g$. Hence, Assumption \ref{assump:str-cvx-F} is without loss of generality.} With only convexity, Algorithm \ref{alg:ajadmm} can be shown to converge at the rate $O(1/t)$ with parameters fixed for all iterations, and the order $1/t$ is optimal as shown in the very recent work \cite{li2016optimal}.

\subsection{Convergence results}
In this subsection, we show the $O(1/t^2)$ convergence rate result of Algorithm \ref{alg:ajadmm}. First, we establish a result of running one iteration of Algorithm \ref{alg:ajadmm}.
\begin{lemma}[One-iteration analysis]\label{lem:1iter-ajadmm}
Under Assumptions \ref{assump:full-lip-F} and \ref{assump:str-cvx-F}, let $\{(x^k,\lambda^k)\}$ be the sequence generated from Algorithm \ref{alg:ajadmm}. Then for any $k$ and $(x,\lambda)$ such that $Ax=b$, it holds that
\begin{equation}\label{eq:1iter-ajadmm}
\begin{aligned}
&~\Phi(x^{k+1},x,\lambda) \\
\le &~ \frac{1}{2\rho_k}\left[\|\lambda-\lambda^k\|^2-\|\lambda-\lambda^{k+1}\|^2+\|\lambda^k-\lambda^{k+1}\|^2\right]-\beta_k\|r^{k+1}\|^2  \\
&~-\frac{1}{2}\left[\|x^{k+1}-x\|_{P^k-\beta_k A^\top A+\mu I}^2-\|x^k-x\|_{P^k-\beta_k A^\top A}^2+\|x^{k+1}-x^k\|_{P^k-\beta_k A^\top A - L_f I}^2\right]. 
\end{aligned}
\end{equation}
\end{lemma}
Using the above lemma, we are able to prove the following theorem.
\begin{theorem}\label{thm:ajadmm-g}
Under Assumptions \ref{assump:full-lip-F} and \ref{assump:str-cvx-F}, let $\{(x^k,\lambda^k)\}$ be the sequence generated by Algorithm~\ref{alg:ajadmm}. Suppose that the parameters are set to satisfy
\begin{equation}\label{eq:ajadmm-para-1}
0<\rho_k\le 2\beta_k,\quad P^k\succeq \beta_k A^\top A+ L_f I,\,\forall k\ge 1,
\end{equation}
and there exists a number $k_0$ such that for all $k\ge 2$,
\begin{eqnarray}
\frac{k+k_0+1}{\rho_k} &\le& \frac{k+k_0}{\rho_{k-1}}, \label{eq:ajadmm-para-2} \\
 (k+k_0+1)(P^k-\beta_k A^\top A) &\preceq& (k+k_0)(P^{k-1}-\beta_{k-1} A^\top A+\mu I). \label{eq:ajadmm-para-3}
\end{eqnarray}
Then, for any $(x,\lambda)$ satisfying $Ax=b$, we have
\begin{eqnarray}
 \sum_{k=1}^t(k+k_0+1)\Phi(x^{k+1},x,\lambda) +\sum_{k=1}^t\frac{k+k_0+1}{2}(2\beta_k-\rho_k)\|r^{k+1}\|^2& & \nonumber \\ %\cr
 +\frac{t+k_0+1}{2}\|x^{t+1}-x\|^2_{P^t-\beta_t A^\top A+\mu I}
&\le & \phi_1(x,\lambda), \label{eq:ajadmm-rate-g}
\end{eqnarray}
where 
\begin{equation}\label{eq:def-lit-phi}\phi_1(x,\lambda)=\frac{k_0+2}{2\rho_1}\|\lambda-\lambda^1\|^2 + \frac{k_0+2}{2}\|x^1-x\|^2_{P^1-\beta_1 A^\top A}.
\end{equation}
\end{theorem}
In the next theorem, we provide a set of parameters that satisfy the conditions in Theorem \ref{thm:ajadmm-g} and establish the $O(1/t^2)$ convergence rate result.
\begin{theorem}[Convergence rate of order $1/t^2$]\label{thm:ajadmm-spc}
Under Assumptions \ref{assump:saddle-pt} through \ref{assump:str-cvx-F}, let $\{(x^k,\lambda^k)\}$ be the sequence generated by Algorithm \ref{alg:ajadmm} with parameters set to:
{
\begin{equation}\label{eq:ajadmm-para-spc}
\beta_k=\rho_k=k\beta,\quad P^k=kP+L_f I,\,\forall k\ge 1,
\end{equation}
}
where $P$ is a block diagonal matrix satisfying $0\prec P-\beta A^\top A\preceq \frac{\mu}{2}I$.
Then,
\begin{equation}\label{eq:ajadmm-rate-spc}
\max\left\{\beta\|r^{t+1}\|^2,\ \|x^{t+1}-x^*\|^2_{P-\beta A^\top A}\right\}
\le \frac{2}{t(t+k_0+1)}\phi_1(x^*,\lambda^*),
\end{equation}
where $k_0=\frac{2L_f}{\mu}$, 
and $\phi_1$ is defined in \eqref{eq:def-lit-phi}.
In addition, letting $\gamma=\max\left\{ 2\|\lambda^*\|,1+\|\lambda^*\|\right\}$ and
$$T=\frac{t(t+2k_0+3)}{2},\quad \bar{x}^{t+1}=\frac{\sum_{k=1}^t (k+k_0+1)x^k}{T},$$ we have
\begin{subequations}\label{eq:ajadmm-erg-rate-spc}
\begin{align}
&|F(\bar{x}^{t+1})-F(x^*)|\le \frac{1}{T}\max_{|\|\lambda\|\le\gamma}\phi_1(x^*,\lambda),\\
&\|A\bar{x}^{t+1}-b\|\le \frac{1}{T\max\{1,\|\lambda^*\|\} }\max_{\|\lambda\|\le\gamma}\phi_1(x^*,\lambda).
\end{align}
\end{subequations}
\end{theorem}

%Some remarks are in order below, 
%which are in fact also valid for Theorem \ref{thm:rate1}.
%\begin{remark}
% If the coupling term $f$ exists, then $k_0$ can be regarded as a condition number of \eqref{eq:mb-prob}. The convergence rate results given in  \eqref{eq:ajadmm-rate-spc} and \eqref{eq:ajadmm-erg-rate-spc} are dependent on $k_0$. The $O(1/t^2)$ convergence result in \cite{chambolle2011first} for a primal-dual method implicitly maintains a factor of $\frac{L_f^2}{\mu^2}$ in their bound,
%while our above $O(1/t^2)$ bound requires only a factor of $\frac{L_f}{\mu}$.
%If $f$ does not exist at all, then it corresponds to $L_f=0$, and so $k_0=1$. In this case, $T=t(t+5)/2$ 
%is independent of $\mu$.
%Since in this case $\beta$ can be set to the order of $\mu$, our result generalizes
%the result in \cite{goldstein2014fast} which shows that the two-block ADMM for \emph{separable} strongly convex problems has a convergence rate of $O(\frac{1}{\mu t^2})$.
%\end{remark}

\section{Accelerating randomized primal-dual block coordinate updates}\label{sec:arpdc}
In this section, we generalize Algorithm \ref{alg:ajadmm} to a randomized setting where the user may choose to update a subset of blocks at each iteration. Instead of updating all $M$ block variables, we randomly choose a subset of them to renew at each iteration. Depending on the number of processors (nodes, or cores), we can choose a single or multiple block variables for each update.

\subsection{The algorithm}
Our algorithm is an accelerated version of the randomized primal-dual coordinate update method recently proposed in \cite{GXZ-RPDCU2016}, {for which we shall use RPDC as its acronym.\footnote{In fact, \cite{GXZ-RPDCU2016} presents a more general algorithmic framework. It assumes two groups of variables, and each has multi-block structure. Our method in Algorithm \ref{alg:arpdc} is an accelerated version of one special case of Algorithm 1 in \cite{GXZ-RPDCU2016}.}} At each iteration, it performs a block proximal gradient update to a subset of randomly selected primal variables while keeping the remaining ones fixed, followed by an update to the multipliers. Specifically,
at iteration $k$, it selects an index set $S_k\subset \{1,\ldots,M\}$ with cardinality $m$ and performs the following updates:
\begin{subequations}\label{eq:fw-arpdc}
\begin{align}
&x_i^{k+1}=\left\{
\begin{array}{ll}\argmin\limits_{x_i}\langle \nabla_i f(x^k)-A_i^\top(\lambda^k-\beta_k r^k), x_i\rangle+g_i(x_i)+\frac{\eta_k}{2}\|x_i-x_i^k\|^2, &\text{ if } i\in S_k,\\[0.1cm]
x_i^k,& \text{ if }i\not\in S_k\end{array}\right.\label{eq:fw-arpdc-x}\\
& r^{k+1}=r^k+\sum_{i\in S_k}A_i(x_i^{k+1}-x_i^k),\\
& \lambda^{k+1}=\lambda^k - \rho_k r^{k+1},\label{eq:fw-arpdc-lam}
\end{align}
\end{subequations}
where $\beta_k, \rho_k$ and $\eta_k$ are algorithm parameters, and their values will be determined later. Note that we use $\frac{\eta_k}{2}\|x_i-x_i^k\|^2$ in \eqref{eq:fw-arpdc-x} for simplicity. It can be replaced by a PSD matrix weighted norm square term as in \eqref{eq:ajadmm-x}, and our convergence results still hold.

Algorithm \ref{alg:arpdc} summarizes the above method. If the parameters $\beta_k, \rho_k$ and $\eta_k$ are fixed during all the iterations, i.e., constant parameters, the algorithm reduces to a special case of the RPDC method in \cite{GXZ-RPDCU2016}. Adapting these parameters to the iterations, we will show that Algorithm \ref{alg:arpdc} enjoys faster convergence rate than RPDC if the problem is strongly convex.

\begin{algorithm}[H]\caption{Accelerated randomized primal-dual block coordinate update method for \eqref{eq:mb-prob}}\label{alg:arpdc}
\DontPrintSemicolon
\textbf{Initialization:} choose $x^1$, set $\lambda^1=0$, let $r^1=Ax^1-b$, and choose parameter $m$\;
\For{$k=1,2,\ldots$}{
Select $S_k\subset \{1,2,\ldots,M\}$ uniformly at random with $|S_k|=m$.\;
Choose parameters $\beta_k, \rho_k$ and $\eta_k$.\;
Let $x^{k+1}\gets$ \eqref{eq:fw-arpdc-x} and $\lambda^{k+1}\gets$ \eqref{eq:fw-arpdc-lam}.\;
\If{a certain stopping criterion satisfied}{
Return $(x^{k+1},\lambda^{k+1})$.
}
}
\end{algorithm}

%In this section, we analyze the convergence of Algorithm \ref{alg:arpdc}. Assuming strong convexity on $F$, we show that Algorithm \ref{alg:arpdc} enjoys the $O(1/t^2)$ convergence rate. We first establish some preliminary results. Then we show that under certain conditions on the parameters, Algorithm \ref{alg:arpdc} converges with $O(1/t^2)$ rate. Finally, we provide a set of parameters that satisfy the required conditions.

\subsection{Convergence results}
In this subsection, we establish convergence results of Algorithm \ref{alg:arpdc} under Assumptions \ref{assump:saddle-pt} and \ref{assump:str-cvx-F}, and also the following partial gradient Lipschitz continuity assumption.

\begin{assumption}\label{assump:lip-F}
For any $S\subset\{1,\ldots,M\}$ with $|S|=m$, $\nabla_S f$ is Lipschitz continuous with a uniform constant $L_m$.
\end{assumption}

Note that if $\nabla f$ is Lipschitz continuous with constant $L_f$, then $L_m\le L_f$ and $L_M = L_f$. In addition, if $x^+$ and $x$ only differ on a set $S\subset [M]$ with cardinality $m$, then
\begin{equation}\label{eq:S-lip-ineq}
f(x^+) \le f(x) + \langle \nabla f(x), x^+-x\rangle + \frac{L_m}{2}\|x^+-x\|^2.
\end{equation} 

Similar to the analysis in section \ref{sec:ajadmm}, we first establish a result of running one iteration of Algorithm \ref{alg:arpdc}. Throughout this section, we denote $\theta=\frac{m}{M}$.
\begin{lemma}[One iteration analysis]\label{lem:1iter}
Under Assumptions \ref{assump:str-cvx-F} and \ref{assump:lip-F}, let $\{(x^k,\lambda^k)\}$ be the sequence generated from Algorithm \ref{alg:arpdc}.  Then for any $x$ such that $Ax=b$, it holds
\begin{align}
& ~\EE\left[\Phi(x^{k+1},x,\lambda^{k+1})+(\beta_k-\rho_k)\|r^{k+1}\|^2+ \frac{\mu}{2}\|x^{k+1}-x\|^2\right] \label{eq:sum-bd1} \\
\le & ~(1-\theta)\EE\left[\Phi(x^k,x,\lambda^k)+\beta_k\| r^k\|^2+\frac{\mu}{2}\|x^k-x\|^2\right] - \EE\left[\Delta_{\eta_k I-\beta_k A^\top A}(x^{k+1},x^k,x) - \frac{L_m}{2} \|x^{k+1}-x^k\|^2\right]. \nonumber
%& &+\frac{\beta_k}{2}\EE\left(\|A(x^{k+1}-x)\|^2-\|A(x^k-x)\|^2+\|A(x^{k+1}-x^k)\|^2\right)\nonumber \\
%& &-\frac{1}{2}\EE\left(\eta_k\|x^{k+1}-x\|^2-\eta_k\|x^k-x\|^2+ (\eta_k-L_m)\|x^{k+1}-x^k\|^2\right)
\end{align}
\end{lemma}
When $\mu=0$ (i.e., \eqref{eq:mb-prob} is convex), Algorithm \ref{alg:arpdc} has $O(1/t)$ convergence rate with fixed $\beta_k,\rho_k,\eta_k$. This can be shown from \eqref{eq:sum-bd1}, and {a similar result in slightly different form} has been established in \cite[Theorem 3.6]{GXZ-RPDCU2016}.
For completeness, we provide its proof in the appendix. %the readers' convenience,
%we provide the result below without proof.
\begin{theorem}[Un-accelerated convergence]
\label{thm:naccl}
Under Assumptions \ref{assump:saddle-pt} and \ref{assump:lip-F}, let $\{(x^k,\lambda^k)\}$ be the sequence generated from Algorithm \ref{alg:arpdc} with $\beta_k=\beta,\rho_k=\rho,\eta_k=\eta$ for all $k$, satisfying
$$0<\rho\le \theta\beta,\quad \eta\ge L_m+\beta\|A\|_2^2,$$
where $\|A\|_2$ denotes the spectral norm of $A$. Then
\begin{subequations}
\begin{align}
&\big|\EE [F(\bar{x}^t)-F(x^*)]\big|\le \frac{1}{1+\theta (t-1)}\max_{\|\lambda\|\le \gamma}\phi_2(x^*,\lambda),\\
&\EE\|A\bar{x}^t-b\|\le \frac{1}{(1+\theta (t-1))\max\{1, \|\lambda^*\|\} }\max_{\|\lambda\|\le \gamma}\phi_2(x^*,\lambda),
\end{align}
\end{subequations}
where $(x^*,\lambda^*)$ satisfies the KKT conditions in \eqref{eq:kkt-conds}, $\gamma=\max\{\|2\lambda^*\|, 1+\|\lambda^*\|\}$, and
$$\bar{x}^t=\frac{x^{t+1}+\theta\sum_{k=2}^t x^k}{1+\theta (t-1)},\quad \phi_2(x,\lambda)=(1-\theta)\left(F(x^1)-F(x)\right)+\frac{\eta}{2}\|x^1-x\|^2+\frac{\theta\|\lambda\|^2}{2\rho}.$$
\end{theorem}

When $F$ is strongly convex, the above $O(1/t)$ convergence rate can be accelerated to $O(1/t^2)$ by adaptively changing the parameters at each iteration.
The following theorem is our main result. It shows an $O(1/t^2)$ convergence result under certain conditions on the parameters. Based on this theorem, we will give a set of parameters that satisfy these conditions, thus providing a specific scheme to choose the paramenters.

\begin{theorem}\label{thm:rate0}
Under Assumptions \ref{assump:str-cvx-F} and \ref{assump:lip-F}, let $\{(x^k,\lambda^k)\}$ be the sequence generated from Algorithm \ref{alg:arpdc} with parameters satisfying the following conditions for a certain number $k_0$:
\begin{subequations}\label{eq:para-conds}
\begin{eqnarray}
\theta(k+k_0+1)&\ge& 1,\,\forall k\ge 2,\label{eq:cond1}\\
(\beta_{k-1}-\rho_{k-1})(k+k_0)&\ge & (1-\theta)(k+k_0+1)\beta_k,\,\forall 2\le k \le t,\label{eq:cond2}\\
\frac{\theta(k+k_0+1)-1}{\rho_{k-1}}&\ge&\frac{\theta(k+k_0+2)-1}{\rho_k},\,\forall\, 2\le k\le t-1,\label{eq:cond3}\\
\frac{\theta(t+k_0+1)-1}{\rho_{t-1}}&\ge&\frac{t+k_0+1}{\rho_t},\,\label{eq:cond4}\\
\beta_k(k+k_0+1)&\ge &\beta_{k-1}(k+k_0),\,\forall k\ge 2,\label{eq:cond5}\\
(k+k_0+1)(\eta_k-L_m)I&\succeq &\beta_k(k+k_0+1)A^\top A,\,\forall k\ge 1,\label{eq:cond6}\\
(k+k_0)\eta_{k-1}+\mu\big(\theta(k+k_0+1)-1\big)&\ge & (k+k_0+1)\eta_k,\,\forall k\ge 2.\label{eq:cond7}
\end{eqnarray}
\end{subequations}
Then for any $(x,\lambda)$ such that $Ax=b$, we have
\begin{eqnarray}
& & (t+k_0+1) \EE\Phi(x^{t+1},x,\lambda)+\sum_{k=2}^t\big(\theta(k+k_0+1)-1\big)\EE\Phi(x^k,x,\lambda)\nonumber\\
&\le & (1-\theta)(k_0+2)\EE\left[\Phi(x^1,x,\lambda^1)+\beta_1\| r^1\|^2+\frac{\mu}{2}\|x^1-x\|^2\right]+\frac{\eta_1(k_0+2)}{2}\EE\|x^1-x\|^2 \nonumber \\
& &+\frac{\theta(k_0+3)-1}{2\rho_1}\EE\|\lambda^1-\lambda\|^2-\frac{t+k_0+1}{2}\EE\|x^{t+1}-x\|_{(\mu+\eta_t) I-\beta_t A^\top A}^2 . \label{eq:sum-bd4}
\end{eqnarray}
\end{theorem}

Specifying the parameters that satisfy \eqref{eq:para-conds}, we show $O(1/t^2)$ convergence rate of Algorithm \ref{alg:arpdc}.

\begin{proposition}\label{prop:param-cond}
The following parameters satisfy all conditions in \eqref{eq:para-conds}:
\begin{subequations}\label{eq:paras}
\begin{align}
&\beta_k= \frac{\mu(\theta k+2+\theta)}{2\rho\|A\|_2^2},\,\forall k\ge 1,\label{eq:paras-beta}\\
&\rho_k=\left\{
\begin{array}{ll}
\frac{\theta \beta_k}{(6-5\theta)}, & \text{ for }1\le k\le t-1,\\[0.2cm]
\frac{(t+k_0+1)\rho_{t-1}}{\theta(t+k_0+1)-1}, &\text{ for }k=t
\end{array}
\right.\label{eq:paras-rho}\\
&\eta_k=\rho\beta_k\|A\|_2^2+L_m,\,\forall k\ge 1,\label{eq:paras-eta}
\end{align}
\end{subequations}
where $\rho\ge 1$ and
\begin{equation}
k_0=\frac{4}{\theta}+\frac{2L_m}{\theta \mu}.\label{eq:paras-k0}
\end{equation}
\end{proposition}

\begin{theorem}[Accelerated convergence]\label{thm:rate1}
Under Assumptions \ref{assump:saddle-pt}, \ref{assump:str-cvx-F} and \ref{assump:lip-F}, let $\{(x^k,\lambda^k)\}$ be the sequence generated from Algorithm \ref{alg:arpdc} with parameters taken as in \eqref{eq:paras}.
 Then
\begin{equation}\label{eq:rate-obj-feas}
\big|\EE[F(\bar{x}^{t+1})-F(x^*)]\big|\le \frac{1}{T}\max_{\|\lambda\|\le \gamma}\phi_3(x^*,\lambda),\quad \EE\|A\bar{x}^{t+1}-b\|\le \frac{1}{T\max\{1,\|\lambda^*\|\}}\max_{\|\lambda\|\le \gamma}\phi_3(x^*,\lambda),
\end{equation}
where $\gamma=\max\{2\|\lambda^*\|, 1+\|\lambda^*\|\}$,
\begin{eqnarray*}
\bar{x}^{t+1} &=& \frac{(t+k_0+1)x^{t+1}+\sum_{k=2}^t\big(\theta(k+k_0+1)-1\big)x^k}{T}, \\
\phi_3(x, \lambda) &=&(1-\theta)(k_0+2)\left[F(x^1)-F(x)+\beta_1\|r^1\|^2+\frac{\mu}{2}\|x^1-x\|^2\right]  \\
& & %\hspace{2cm}
+\frac{\eta_1(k_0+2)}{2}\|x^1-x\|^2+\frac{\theta(k_0+3)-1}{2\rho_1}\|\lambda\|^2
\end{eqnarray*}
and $$T= (t+k_0+1)+\sum_{k=2}^t\big(\theta(k+k_0+1)-1\big).$$
In addition,
$$\EE\|x^{t+1}-x^*\|^2\le \frac{2\phi_3(x^*,\lambda^*)}{(t+k_0+1)\left(\frac{(\rho-1)\mu}{2\rho}(\theta t + \theta + 2)+2\mu+L_m\right)}.$$
\end{theorem}

\section{Linearly convergent primal-dual method}
\label{sec:linear}

In this section, we assume some more structure on \eqref{eq:mb-prob} and show that a linear rate of convergence is possible.  %to have a linearly convergent block coordinate update method.
If there is no linear constraint, Algorithm \ref{alg:arpdc} reduces to the RCD method proposed in \cite{nesterov2012rcd}. It is well-known that RCD converges linearly if the objective is strongly convex. However, with the presence of linear constraints, mere strong convexity of the objective of the primal problem only ensures the smoothness %but cannot guarantee
%strong convexity
of its Lagrangian dual function, but not its strong concavity. Hence, in general, we do not expect linear convergence by only assuming strong convexity on the primal objective function. To ensure linear convergence on both the primal and dual variables, we need additional assumptions.

Throughout this section, we suppose that there is at least one block variable being absent in the nonseparable part of the objective, namely $f$. For convenience, we rename this block variable to be $y$, and the corresponding component function and constraint coefficient matrix as $h$ and $B$. Specifically, we consider the following problem
\begin{equation}\label{eq:mb-prob-y}
\min_{x,y} f(x_1,\ldots,x_M)+\sum_{i=1}^M g_i(x_i)+h(y), \st \sum_{i=1}^M A_ix_i+By=b.
\end{equation}

{One example of \eqref{eq:mb-prob-y} is the problem that appears while computing a point on the central path of a convex program. Suppose we are interested in solving
\begin{equation}\label{eq:genral-cp}
\min_x f(x_1,\ldots,x_M), \st \sum_{i=1}^M A_ix_i \le b, \, x_i\ge 0, i = 1,\ldots, M.
\end{equation}
Let $y= b - \sum_{i=1}^M A_ix_i$ and use the log-barrier function. We have the log-barrier approximation of \eqref{eq:genral-cp} as follows:
\begin{equation}\label{eq:genral-cp-logbar}
\min_{x, y} f(x_1,\ldots,x_M) - \mu \sum_{i=1}^M e^\top \log x_i - \mu e^\top \log y, \st \sum_{i=1}^M A_ix_i + y = b,
\end{equation}
where $e$ is the all-one vector. As $\mu$ decreases, the approximation becomes more accurate. 
}

Towards a solution to \eqref{eq:mb-prob-y}, we modify Algorithm \ref{alg:arpdc} by updating $y$-variable after the $x$-update. Since there is only a single $y$-block, to balance $x$ and $y$ updates, we do not renew $y$ in every iteration but instead update it in probability $\theta=\frac{m}{M}$. Hence, roughly speaking, $x$ and $y$ variables are updated in the same frequency. The method is summarized in Algorithm \ref{alg:rpdc-lin}.

\begin{algorithm}[h]\caption{Randomized primal-dual block coordinate update for \eqref{eq:mb-prob-y}}\label{alg:rpdc-lin}
\DontPrintSemicolon
\textbf{Initialization:} choose $(x^1,y^1)$, set $\lambda^1=0$, and choose parameters $\beta,\rho,\eta_x,\eta_y, m$. \;
Let $r^1=Ax^1+By^1-b$ and $\theta=\frac{m}{M}$. \;
\For{$k=1,2,\ldots$}{
Select index set $S_k\subset\{1,\ldots,M\}$ uniformly at random with $|S_k|=m$. \;
Keep $x_i^{k+1}=x_i^k,\,\forall i\not\in S_k$ and update
\begin{equation}\label{eq:update-x-lin}
x_i^{k+1}=\argmin\limits_{x_i}\left\langle \nabla_i f(x^k)-A_i^\top(\lambda^k-\beta r^k), x_i\right\rangle+g_i(x_i)+\frac{\eta_x}{2}\|x_i-x_i^k\|^2, \text{ if } i\in S_k.
\end{equation}
Let $r^{k+\frac{1}{2}}=r^k+\sum_{i\in S_k}A_i(x_i^{k+1}-x_i^k)$. \;
%Let $y^{k+1}=y^k$ or $y^{k+1}=\tilde{y}^{k+1}$ in probability $1-\theta$ and $\theta$
In probability $1-\theta$ keep $y^{k+1}=y^k$, and in probability $\theta$ let $y^{k+1}=\tilde{y}^{k+1}$, where
\begin{equation}\label{class-y-update}
 \tilde{y}^{k+1}=\argmin_y h(y)-\left\langle B^\top(\lambda^k-\beta r^{k+\frac{1}{2}}), y\right\rangle+\frac{\eta_y}{2}\|y-y^k\|^2.
\end{equation}
Let $r^{k+1}=r^{k+\frac{1}{2}}+B(y^{k+1}-y^k)$. \;
Update the multiplier by
\begin{equation}\label{eq:update-lam-y}
\lambda^{k+1}=\lambda^k-\rho r^{k+1} .
\end{equation}
\If{a certain stopping criterion is satisfied}{
Return $(x^{k+1},y^{k+1},\lambda^{k+1})$.
}
}
\end{algorithm}

\subsection{Technical assumptions}
In this section, we denote $z=(x,y,\lambda)$. Assume $h$ is differentiable. Similar to \eqref{eq:kkt-conds}, a point $z^*=( x^*, y^*,\vlam^*)$ is called a KKT point of \eqref{eq:mb-prob-y} if
\begin{subequations}\label{kkt}
\begin{align}
&0\in\partial F( x^*)- A^\top\vlam^*,\label{kkt1}\\
&\nabla h( y^*)- B^\top\vlam^*=0,\label{kkt2}\\
& A x^*+ B y^*- b=0\label{kkt3}.
\end{align}
\end{subequations}
Besides Assumptions \ref{assump:str-cvx-F} and \ref{assump:lip-F}, we make two additional assumptions as follows.
\begin{assumption}\label{assump:saddle-pt-lin}
There exists $z^*=(x^*,y^*,\lambda^*)$ satisfying the KKT conditions in \eqref{kkt}.
\end{assumption}

\begin{assumption}\label{assump:str-cvx-h}
The function $h$ is strongly convex with modulus $\nu$, and its gradient $\nabla h$ is Lipschitz continuous with constant $L_h$.
\end{assumption}

The strong convexity of $F$ and $h$ implies
\begin{subequations}\label{scvx-Fh}
\begin{eqnarray}
%\langle\vx^{k+1}-\vx^*, \tilde{\nabla} F(\vx^{k+1})-\tilde{\nabla} F(\vx^*)\rangle\ge & \mu\|\vx^{k+1}-\vx^*\|^2,
F(x^{k+1})-F(x^*)-\langle\tilde{\nabla}F(x^*),x^{k+1}-x^*\rangle &\ge &\, \frac{\mu}{2}\|\vx^{k+1}-\vx^*\|^2,\label{scvx-F}\\
\langle \vy^{k+1}-\vy^*, \nabla h(\vy^{k+1})-\nabla h(\vy^*)\rangle &\ge &\, \nu\|\vy^{k+1}-\vy^*\|^2.\label{scvx-h}
\end{eqnarray}
\end{subequations}
%where $\tilde{\nabla}F(x^*)$ is a subgradient of $F$ at $x^*$.

%In the following, we sketch a convergence rate analysis highlighting the key components and steps. The proofs, however, will be omitted for succinctness and be found in Appendix.

\subsection{Convergence analysis}
Similar to Lemma \ref{lem:1iter}, we first establish a result of running one iteration of Algorithm \ref{alg:rpdc-lin}. It can be proven by similar arguments to those showing Lemma \ref{lem:1iter}.
\begin{lemma}[One iteration analysis]\label{lem:linear-1step}
Under Assumptions \ref{assump:str-cvx-F}, \ref{assump:lip-F}, and \ref{assump:str-cvx-h}, let $\{( x^k, y^k,\vlam^k)\}$ be the sequence generated from Algorithm \ref{alg:rpdc-lin}. Then for any $k$ and $(x,y,\lambda)$ such that $Ax+By=b$, it holds
\begin{eqnarray}
& & \EE\varphi(z^{k+1},z)+( \beta -\rho)\EE \|r^{k+1}\|^2+\frac{1}{\rho}\EE\Delta(\lambda^{k+1},\lambda^k,\lambda)\nonumber\\
& &+\EE\left[\Delta_P(x^{k+1},x^k,x)-\frac{L_m}{2}\|x^{k+1}-x^k\|^2\right]+\EE\Delta_Q(y^{k+1},y^k,y)\nonumber\\
%& & + \frac{1}{2}\EE\left[\|x^{k+1}- x\|_P^2-\|x^k- x\|_P^2+\|x^{k+1}- x^k\|_{P-L_m I}^2\right] + \frac{1}{2}\EE\left[\|y^{k+1}- y\|_Q^2-\|y^k- y\|_Q^2+\|y^{k+1}- y^k\|_Q^2\right]\nonumber\\
&\le &
(1-\theta)\EE\varphi(z^k,z)+\beta(1 -\theta)\EE \|r^k\|^2+\frac{1-\theta}{\rho}\EE\Delta(\lambda^k,\lambda^{k-1},\lambda) \nonumber \\
& & + \beta \EE\langle A(x^{k+1}- x),  B( y^{k+1}- y^k)\rangle+\beta(1 -\theta)\EE\langle B(y^k- y),A(x^{k+1}-x^k)\rangle. \label{lin-ineq1-1Y}
\end{eqnarray}
where $P=\eta_x I-\beta A^\top A$, $Q=\eta_y I-\beta B^\top B,$ and
\begin{equation}\label{eq:def-varphi}
\varphi(z^k,z)=F(x^k)-F(x)+\frac{\mu}{2}\|x^k-x\|^2+\big\langle y^k- y,{\nabla} h(y^k)\big\rangle-\big\langle \lambda, Ax^k+By^k-b\big\rangle.
\end{equation}
\end{lemma}

%We have
%\begin{align}\label{eq:bd-F-term-lip}
%&\langle x^{k+1}-x,\nabla f(x^{k+1})\rangle-(1-\theta)\langle x^k-x, \nabla f(x^k)\rangle -\langle x^{k+1}-x,\nabla f(x^{k+1})-\nabla f(x^k)\rangle\cr
%=&\langle x^{k+1}-x^k,\nabla f(x^k)+\theta\langle x^k-x,\nabla f(x^k)\rangle\cr
%\ge& f(x^{k+1})-f(x^k)-\frac{L_m}{2}\|x^{k+1}-x^k\|^2+\theta(f(x^k)-f(x)+\frac{\mu_f}{2}\|x^k-x\|^2)\cr
%=&f(x^{k+1})-f(x)-(1-\theta)(f(x^k)-f(x))-\frac{L_m}{2}\|x^{k+1}-x^k\|^2+\frac{\theta\mu_f}{2}\|x^k-x\|^2.
%\end{align}
%Adding \eqref{eq:bd-F-term-lip} and \eqref{eq:term2} to \eqref{lin-ineq1-1Y} gives
%

In the following, we let
\begin{equation}\label{eq:def-Psi}
\Psi(z^k,z^*)=F(x^k)-F(x^*)-\langle\tilde{\nabla}F(x^*),x^k-x^*\rangle+\big\langle y^k- y^*,\nabla h(y^k)-\nabla h(y^*)\big\rangle,
\end{equation}
%where $\tilde{\nabla} F( x^{k})$ is a subgradient of $F$ at $ x^{k}$,
and
\begin{eqnarray}
& & \psi(z^k,z^*;P,Q,\beta,\rho,c,\tau) \nonumber \\
&=& (1-\theta)\EE\Psi(z^k,z^*)+\frac{\beta(1 -\theta)}{2}\EE \|r^k\|^2+\frac{1}{2}\EE\|x^{k}- x^*\|^2_{P+\mu(1-\theta)I}+\frac{1}{2}\EE\|y^{k}- y^*\|^2_{Q+\frac{\beta(1 -\theta)}{\tau}B^\top B}\nonumber\\
%& & +\frac{\mu(1-\theta)}{2}\EE\|x^k-x^*\|^2+\frac{1}{2}\EE\|y^{k}- y^*\|^2_{Q+\frac{\beta(1 -\theta)}{\tau}B^\top B}+\frac{\beta(1 -\theta)}{2\tau}\EE\|B(y^k- y^*)\|^2\nonumber \\
%&\hspace{-3cm}
& &
+\frac{1}{2\rho}\EE\left[\|\lambda^{k}-\lambda^*\|^2-(1-\theta)\|\lambda^{k-1}-\lambda^*\|^2+\frac{1}{\theta}\|\lambda^{k}-\lambda^{k-1}\|^2\right]. \label{eq:def-psi}
\end{eqnarray}

The following theorem is key % a major step
to establishing linear convergence of Algorithm \ref{alg:rpdc-lin}.
\begin{theorem}\label{thm-pre}
Under Assumptions \ref{assump:str-cvx-F} through \ref{assump:str-cvx-h}, let $\{( x^k, y^k,\vlam^k)\}$ be the sequence generated from Algorithm \ref{alg:rpdc-lin} with $\rho=\theta\beta$. Let $0<\alpha<\theta$ and $\gamma=\max\left\{\frac{8\|A\|_2^2}{\alpha\mu},\frac{8\|B\|_2^2}{\alpha\nu}\right\}$. Choose $\delta,\kappa\ge0$ such that
\begin{equation}\label{eq:cond-kappa-del}
2\left[\begin{array}{cc}1-(1-\theta)(1+\delta) & (1-\theta)(1+\delta)\\(1-\theta)(1+\delta)& \kappa-(1-\theta)(1+\delta)\end{array}\right]\succeq \left[\begin{array}{cc}\theta & 1-\theta\\ 1-\theta & \frac{1}{\theta}-(1-\theta)\end{array}\right],
\end{equation}
and positive numbers $\eta_x,\eta_y,c,\tau_1,\tau_2,\beta$ such that
\begin{subequations}\label{eq:paras-lin}
\begin{eqnarray}
%&\frac{\alpha\mu}{2}I\succ \frac{\beta}{2\tau_1}A^\top A\\
%&c\le \frac{\alpha \nu}{4L_h^2}\\
P &\succeq& \beta(1-\theta)\tau_2 A^\top A+L_m I\label{eq:choice-P-mat}\\
Q &\succeq& 8c Q^\top Q + 4c\rho^2(1-\theta)(1+\frac{1}{\delta})B^\top B B^\top B+\beta\tau_1 B^\top B.\label{eq:choice-Q-mat}%\\
%& \frac{\alpha\nu}{4} I\succ \frac{\beta(1-\theta)}{2\tau_2}B^\top B - c(1-\theta)\nu^2I\\
%&\left(\frac{\beta\theta^2}{2}+\frac{1}{\gamma}\right)I\succ c\rho^2\left(\kappa+2(1-\theta)(1+\frac{1}{\delta})\right) BB^\top .
\end{eqnarray}
\end{subequations}
Then it holds that
\begin{eqnarray}
& &(1-\alpha)\EE\Psi(z^{k+1},z^*)+\frac{1}{2}\EE \|x^{k+1}- x^*\|^2_{P+(\frac{\alpha\mu}{2}+\mu)I-\frac{\beta}{\tau_1}A^\top A}+\frac{1}{2}\EE\|y^{k+1}- y^*\|^2_{Q+(\frac{3\alpha\nu}{2}-8cL_h^2)I}\nonumber\\
& &+\big(\frac{\beta-\rho}{2}+\frac{1}{\gamma}\big)\EE \|r^{k+1}\|^2-\left(c\rho^2\big(\kappa+2(1-\theta)(1+\frac{1}{\delta})\big)+2c(\beta-\rho)^2\right) \EE \|B^\top r^{k+1}\|^2\nonumber \\
& & +\left(\frac{1}{2\rho}+\frac{c}{2}\sigma_{\min}(BB^\top)\right)\EE\left[\|\lambda^{k+1}-\lambda^*\|^2
-(1-\theta)\|\lambda^{k}-\lambda^*\|^2+\frac{1}{\theta}\|\lambda^{k+1}-\lambda^k\|^2\right]\nonumber \\
& \le & \psi(z^k,z^*;P,Q,\beta,\rho,c,\tau_2). \label{lin-ineq5-1Y}
\end{eqnarray}
\end{theorem}

Using Theorem \ref{thm-pre}, a linear convergence rate of Algorithm \ref{alg:rpdc-lin} follows.
\begin{theorem}\label{thm-linear}
Under Assumptions \ref{assump:str-cvx-F} through \ref{assump:str-cvx-h}, let $\{( x^k, y^k,\vlam^k)\}$ be the sequence generated from Algorithm \ref{alg:rpdc-lin} with $\rho=\theta\beta$. Let $0<\alpha<\theta$ and $\gamma=\max\left\{\frac{8\|A\|_2^2}{\alpha\mu},\frac{8\|B\|_2^2}{\alpha\nu}\right\}$. Assume that $B$ is full row-rank and $\max\{\|A\|_2,\|B\|_2\}\le 1$.  Choose $\delta,\kappa,\eta_x,\eta_y,c,\beta,\tau_1,\tau_2$ satisfying \eqref{eq:cond-kappa-del} and \eqref{eq:paras-lin}, and in addition,
\begin{subequations}\label{eq:paras-lin2}
\begin{eqnarray}
\frac{\alpha}{2}\mu +\theta\mu &>& \frac{\beta}{\tau_1}\label{eq:paras-lin2-tau1}\\
\frac{3\alpha\nu}{4} &>& 4cL_h^2+\frac{\beta(1-\theta)}{2\tau_2}\label{eq:paras-lin2-c1}\\
\frac{1}{\gamma} &>& c\rho^2\left(\kappa+2(1-\theta)(1+\frac{1}{\delta})\right)+2c(\beta-\rho)^2.\label{eq:paras-lin2-c2}
\end{eqnarray}
\end{subequations}
Then
\begin{equation}\label{eq:lin-cvg-ineq}
\psi(z^{k+1},z^*;P,Q,\beta,\rho,c,\tau_2)\le \frac{1}{\eta} \psi(z^k,z^*;P,Q,\beta,\rho,c,\tau_2),
\end{equation}
where
\begin{eqnarray*}
\eta &=&\min\left\{\frac{1-\alpha}{1-\theta},\, 1+\frac{\frac{\alpha}{2}\mu+\theta\mu-\frac{\beta}{\tau_1}}{\eta_x+\mu(1-\theta)},1+\frac{\frac{3\alpha\nu}{4}-4cL_h^2-\frac{\beta(1-\theta)}{2\tau_2}}{\frac{\eta_y}{2}+\frac{\beta(1-\theta)}{2\tau_2}},\,\right.\\
%&\hspace{5cm}
& & \hspace{1cm}
\left.1+\frac{\frac{2}{\gamma}-2c\rho^2\left(\kappa+2(1-\theta)(1+\frac{1}{\delta})\right)-4c(\beta-\rho)^2}{\beta(1-\theta)},\,1+c\rho\sigma_{\min}(BB^\top)\right\}>1.
\end{eqnarray*}
\end{theorem}

We finish this section by making a few remarks.

\begin{remark}
We can always rescale $A,B$ and $b$ without essentially altering the linear constraints. Hence, the assumption $\max\{\|A\|_2,\|B\|_2\}\le 1$ can be made without losing generality.
From \eqref{eq:lin-cvg-ineq}, it is easy to see that {when $P\succ 0$ and $Q\succ0$}, $(x^k,y^k)$ converges to $(x^*,y^*)$ R-linearly in expectation. In addition, note that
\begin{eqnarray*}
& & \|\lambda^{k+1}-\lambda^*\|^2-(1-\theta)\|\lambda^k-\lambda^*\|^2+\frac{1}{\theta}\|\lambda^{k+1}-\lambda^k\|^2\\
&=& \theta\|\lambda^{k+1}-\lambda^*\|^2+2(1-\theta)\langle\lambda^{k+1}-\lambda^*,\lambda^{k+1}-\lambda^k\rangle + (\frac{1}{\theta}-1+\theta)\|\lambda^{k+1}-\lambda^k\|^2\\
&\ge& \left(\theta-\frac{(1-\theta)^2}{\frac{1}{\theta}-1+\theta}\right)\|\lambda^{k+1}-\lambda^*\|^2 \\
&=& \frac{\theta}{\frac{1}{\theta}-1+\theta}\|\lambda^{k+1}-\lambda^*\|^2.
\end{eqnarray*}
Hence, \eqref{eq:lin-cvg-ineq} also implies an R-linear convergence of $\lambda^k$ to $\lambda^*$ in expectation.
\end{remark}

{
\begin{remark}
We give examples of parameters that satisfy the conditions required in Theorem \ref{thm-linear}. First consider the case of $\theta=1$, i.e., all blocks are updated at each iteration. In this case, we can choose $\delta=0,\kappa=\frac{1}{2}$ to satisfy \eqref{eq:cond-kappa-del} and $\eta_x=\beta\|A\|_2^2+L_f$ to satisfy \eqref{eq:choice-P-mat} and let $\alpha=\frac{1}{2}$ and $\tau_1=\frac{\beta}{\mu}$ to ensure that \eqref{eq:paras-lin2-tau1} holds. Finally, choose $\eta_y>\big(\beta+\frac{\beta^2}{\mu}\big)\|B\|_2^2$ and $c$ sufficiently small, and all other conditions in Theorem \ref{thm-linear} are satisfied.
Next consider the case of $\theta<1$. We can choose $\delta=\frac{\theta}{4(1-\theta)}$ and $\kappa=\frac{3}{\theta}+\frac{3\theta}{4}-2$ to satisfy \eqref{eq:cond-kappa-del}, and let $\alpha=\frac{\theta}{2}$, $\tau_1=\frac{\beta}{\theta\mu}$, $\tau_2=\frac{2\beta(1-\theta)}{\nu}$, $\eta_x=\beta(1+(1-\theta)\tau_2)\|A\|_2^2+L_m$, and $\eta_y>\beta(1+\tau_1)\|B\|_2^2$. With such choices, all other conditions required in Theorem \ref{thm-linear} hold when $c$ is sufficiently small. 
\end{remark}
}

\begin{remark}
If there is only one $x$-block and there is no $f$ function, then Algorithm \ref{alg:rpdc-lin} reduces to the so-called linearized ADMM. To show the linear convergence of the linearized ADMM, one scenario in \cite[Theorem 3.1]{deng2012global} assumes\footnote{Besides the scenario that $g$ and $h$ are strongly convex, $h$ is smooth, and $B$ is of full row-rank, \cite[Theorem 3.1]{deng2012global} also shows linear convergence of the linearized ADMM under three other different scenarios.} %But none of the four cases implies another one.}
the strong convexity of $g$ and $h$, the smoothness of $h$,
and the full row-rankness of $B$. %the full row-rankness of $B$.
In Theorem \ref{thm-linear}, we make the same assumptions, and so our result can be considered as a generalization.
% to have the linear convergence but our algorithm is more general.
\end{remark}

\section{Numerical experiments}\label{sec:numerical}
The aim of this section is to test the practical performance of the proposed algorithms. %acceleration assured by the theoretical bounds.
%In this section,
We test Algorithm \ref{alg:arpdc} on %nonnegative
quadratic programming
\begin{equation}\label{eq:qp}
\min_x F(x)=\frac{1}{2} x^\top Q x + c^\top x, \st Ax=b,\, x\ge0,
\end{equation}
{and Algorithm \ref{alg:rpdc-lin} on the log-barrier approximation of linear programming
\begin{equation}\label{eq:log-b-lp}
\min_{x,y} c^\top x - e^\top \log x - e^\top \log y, \st Ax+y = b,\, x_i\le u_i, \forall i.
\end{equation}
}

\textbf{Quadratic programming.} Two types of randomized implementations are considered: one with fixed parameters and the newly introduced one with adaptive parameters,
%Two settings are used, one with nonadaptive parameters and the other with adaptive parameters, and
which shall be called nonadaptive RPDC and adaptive RPDC respectively. Note that the former reduces to the method proposed in \cite{GXZ-RPDCU2016} when applied to \eqref{eq:qp}. The purpose of the experiment is to test the effect of acceleration for the latter approach.
%and we compare it to the RPDC method proposed in \cite{GXZ-RPDCU2016}. Note that when applied to \eqref{eq:qp}, RPDC can be regarded as a special case of Algorithm \ref{alg:arpdc} with nonadaptive parameters.

The data was generated randomly as follows. We let $Q=HDH^\top\in\RR^{n\times n}$, where $H$ is Gaussian randomly generated orthogonal matrix and $D$ is a diagonal matrix with $d_{ii}=1+(i-1)\frac{L-1}{n-1},\,i=1,\ldots,n$. Hence, the smallest and largest singular values of $Q$ are 1 and L respectively, and the objective of \eqref{eq:qp} is strongly convex with modulus $1$. The components of $c$ follow standard Gaussian distribution, and those of $b$ follow uniform distribution on $[0,1]$. We let $A=[B, I]\in\RR^{p\times n}$ to guarantee the existence of feasible solutions, where $B$ was generated according to standard Gaussian distribution. In addition, we normalized $A$ so that it has a unit spectral norm.

{In the test, we fixed $n=2000, p=200$ and varied $L$ among $\{10, 100, 1000\}$.} For both nonadaptive and adaptive RPDC, we evenly partitioned $x$ into $40$ blocks, i.e., each block consists of 50 coordinates, and we set $m=40$, i.e., all blocks are updated at each iteration. For the adaptive RPDC, we set the values of its parameters according to \eqref{eq:paras} with $\rho=1$, and those for the nonadaptive RPDC were set based on Theorem \ref{thm:naccl} with $\rho=\beta,\, \eta=100+\beta,\,\forall k$ where $\beta$ varied among $\{1,10,100,1000\}$. Figures \ref{fig:qp-p200-c10} through \ref{fig:qp-p200-c1000} plot the objective values and feasibility violations by Algorithm \ref{alg:arpdc} under these two different settings. From these results, we see that adaptive RPDC performed well for all three datasets with a single set of parameters while the performance of the nonadaptive one was severely affected by the penalty parameter.

\begin{figure}
\begin{center}
\begin{tabular}{cccc}
$\beta=1$ & $\beta=10$ & $\beta=100$ & $\beta=1000$ \\
\includegraphics[width=0.2\textwidth]{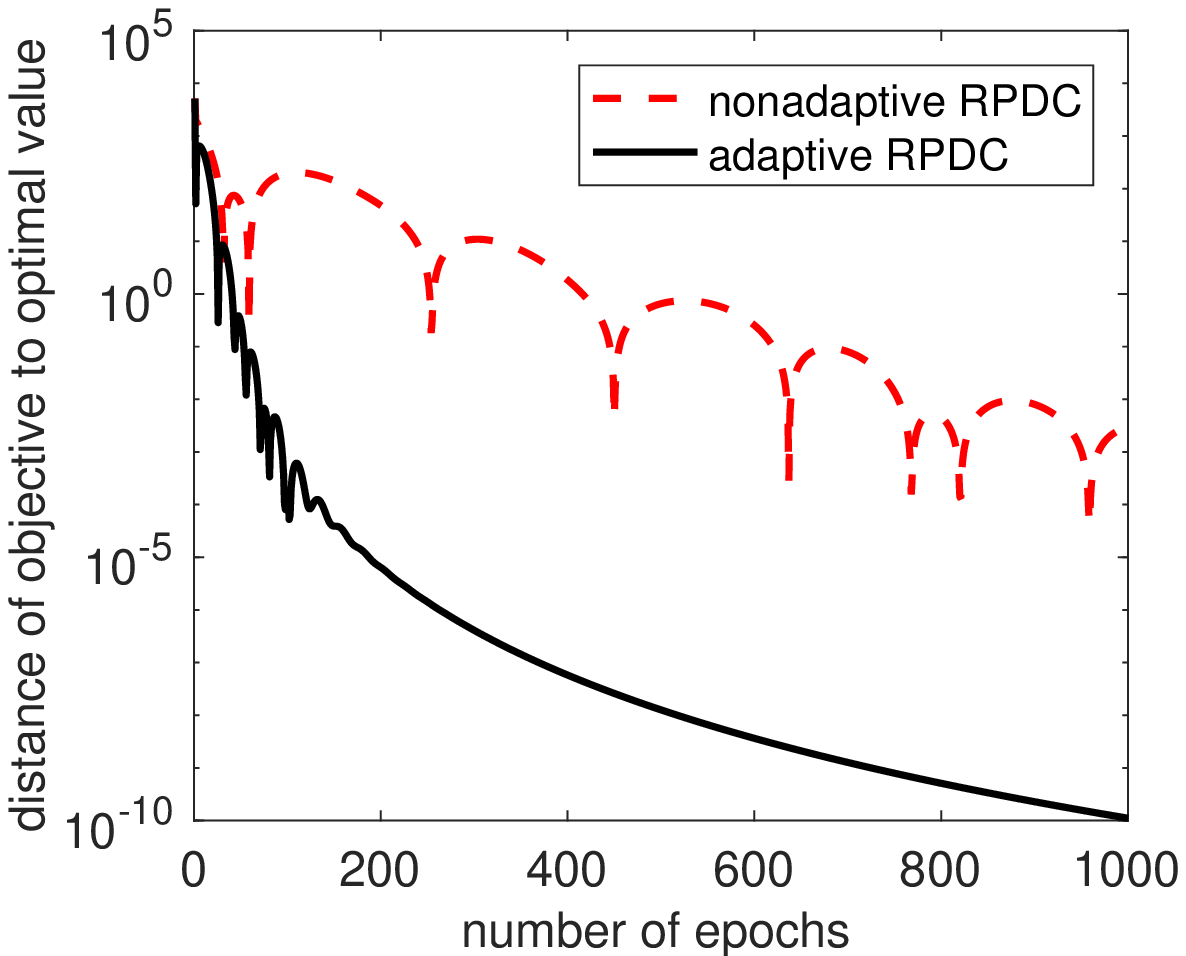}&
\includegraphics[width=0.2\textwidth]{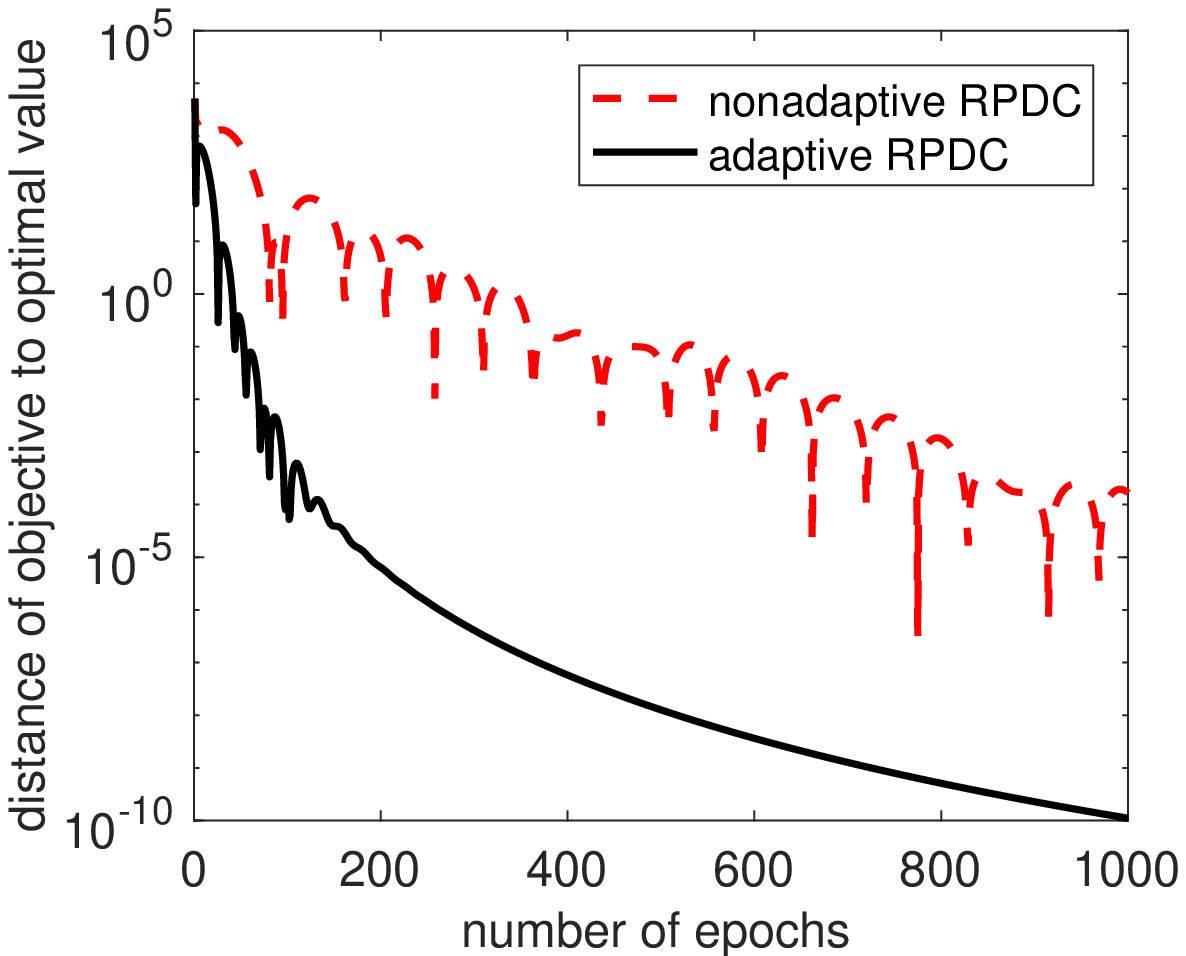}&
\includegraphics[width=0.2\textwidth]{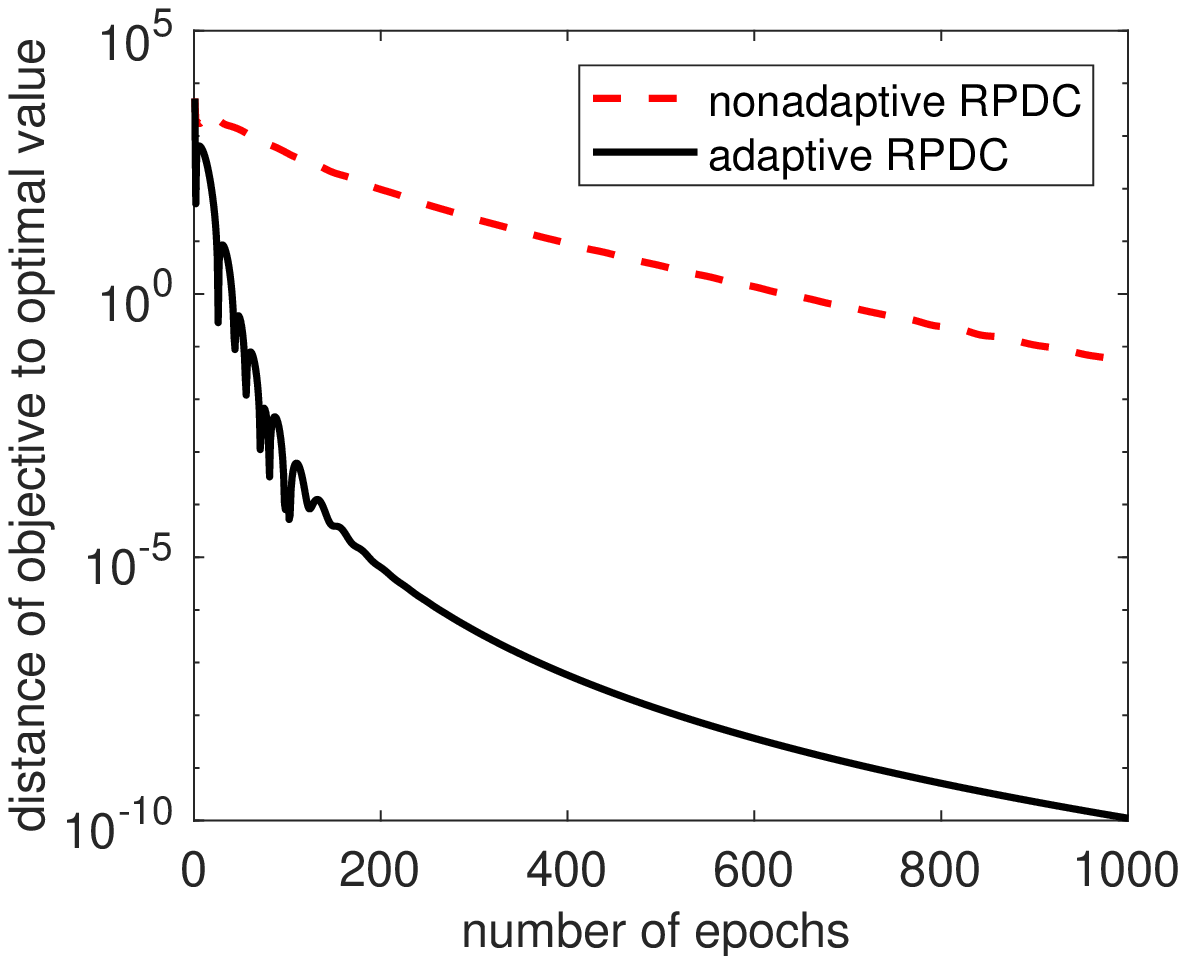}&
\includegraphics[width=0.2\textwidth]{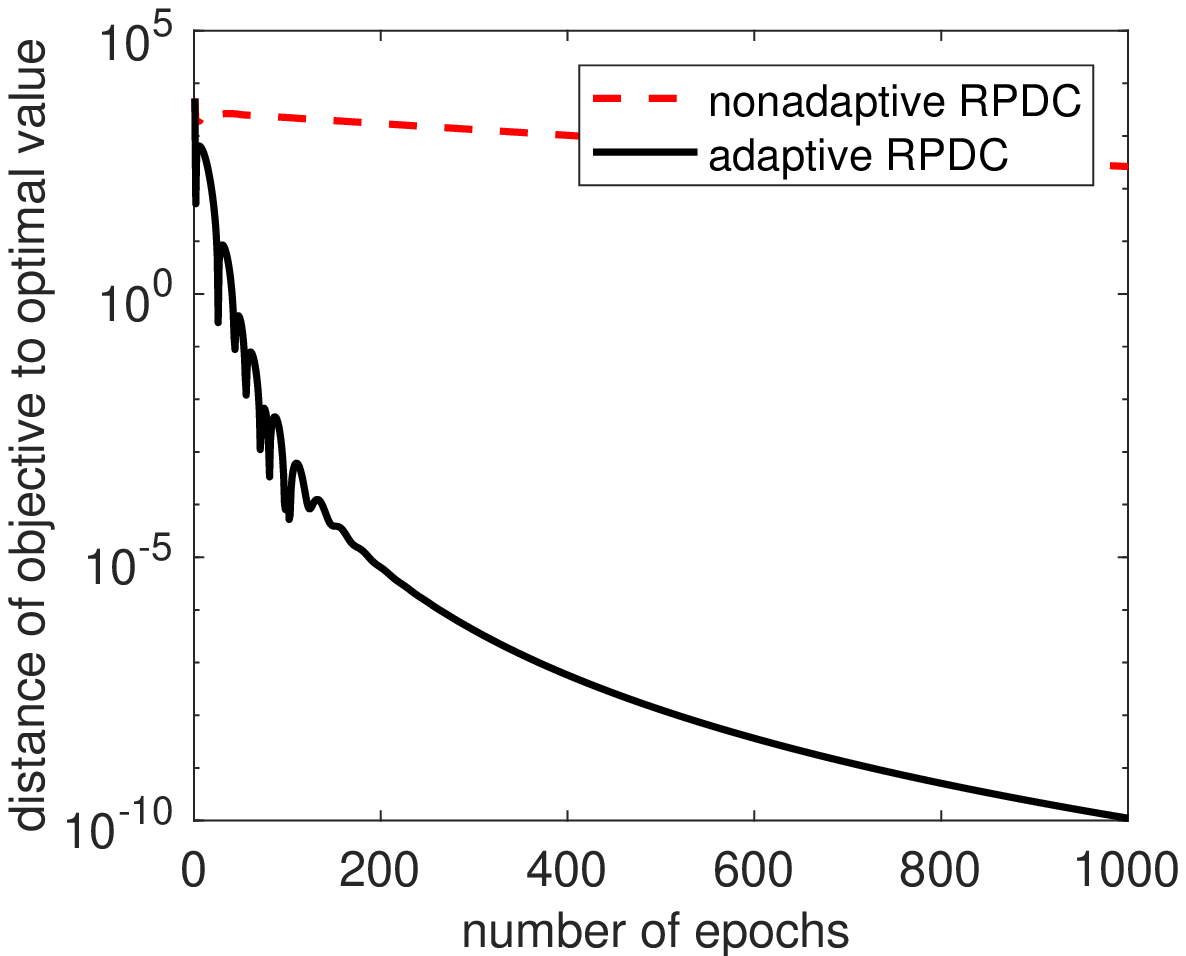}\\
\includegraphics[width=0.2\textwidth]{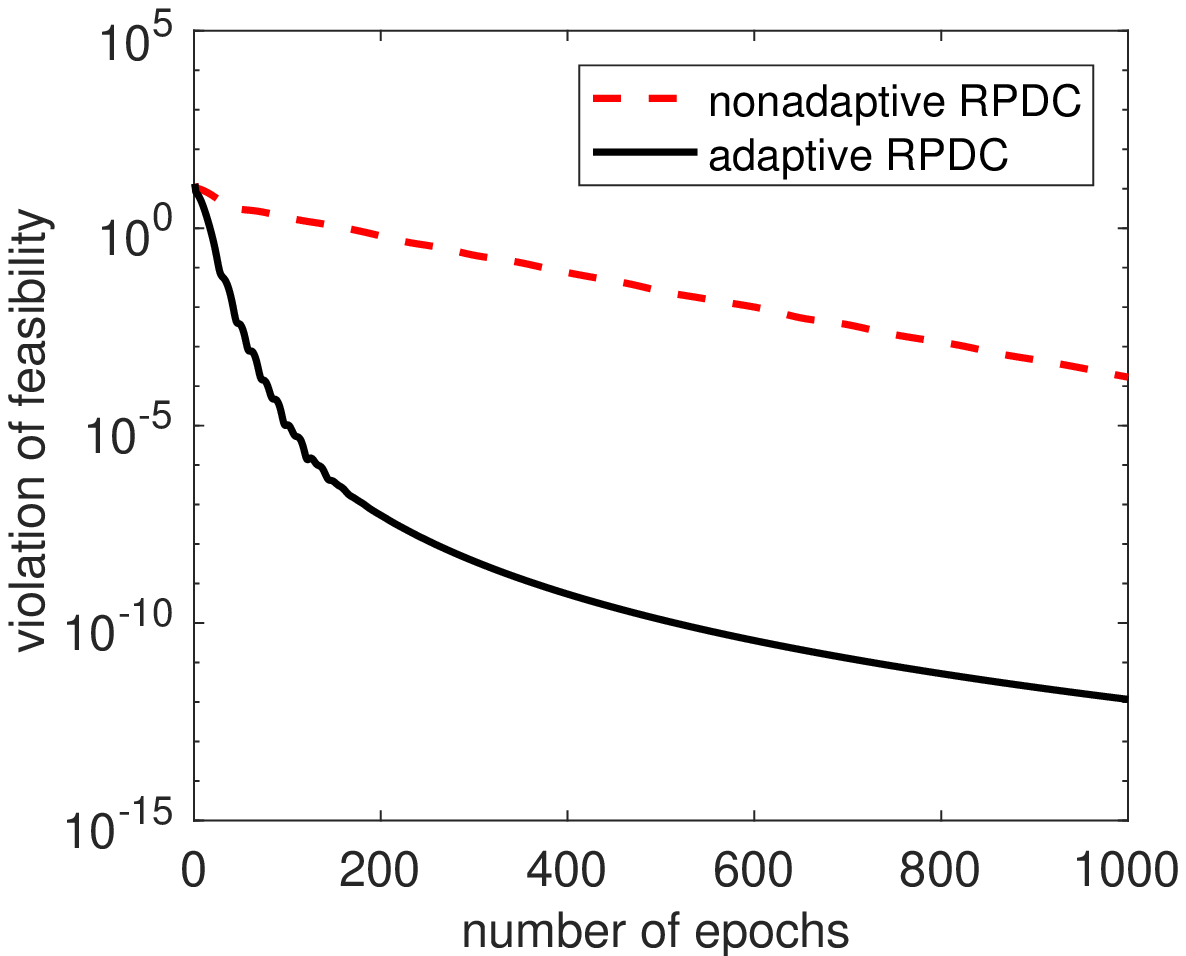}&
\includegraphics[width=0.2\textwidth]{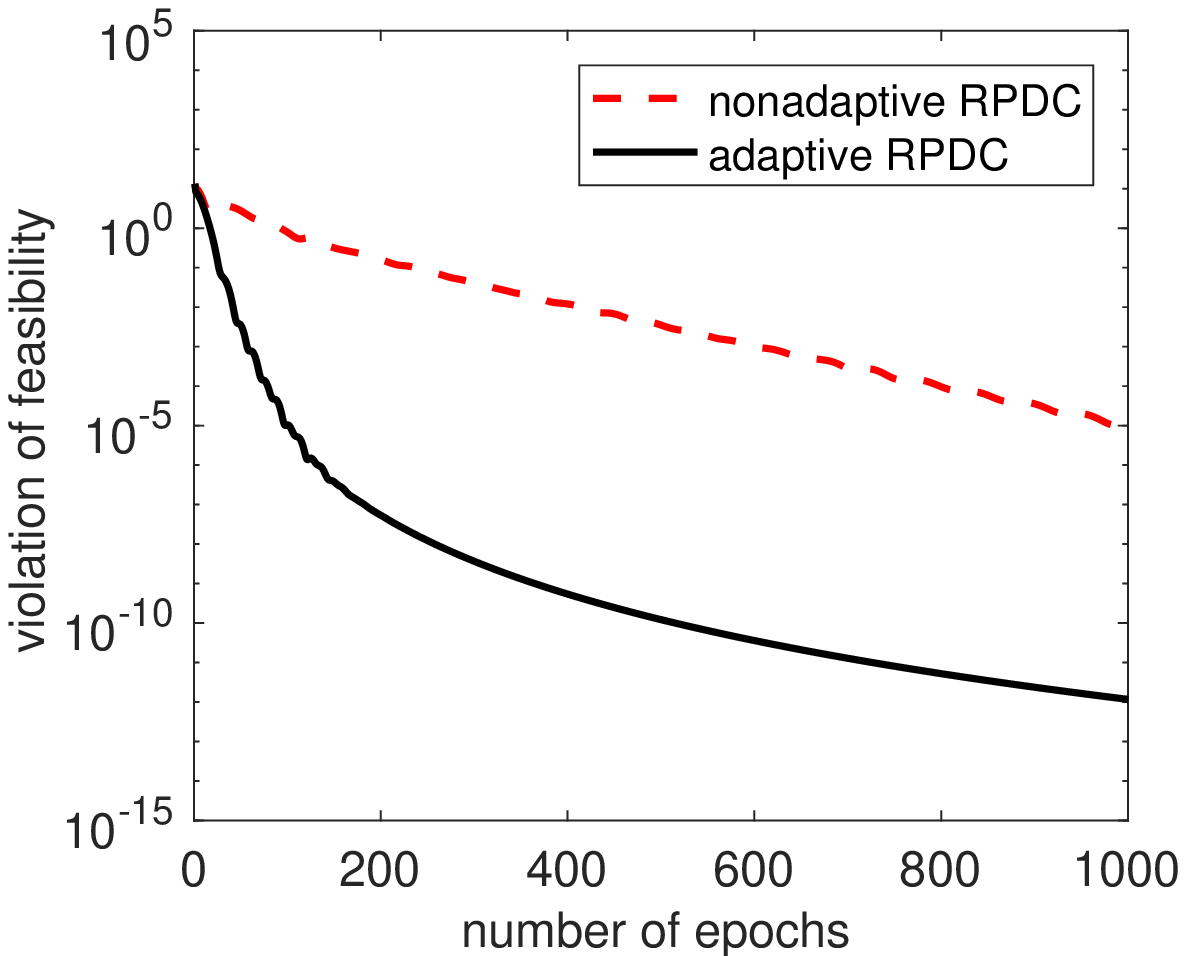}&
\includegraphics[width=0.2\textwidth]{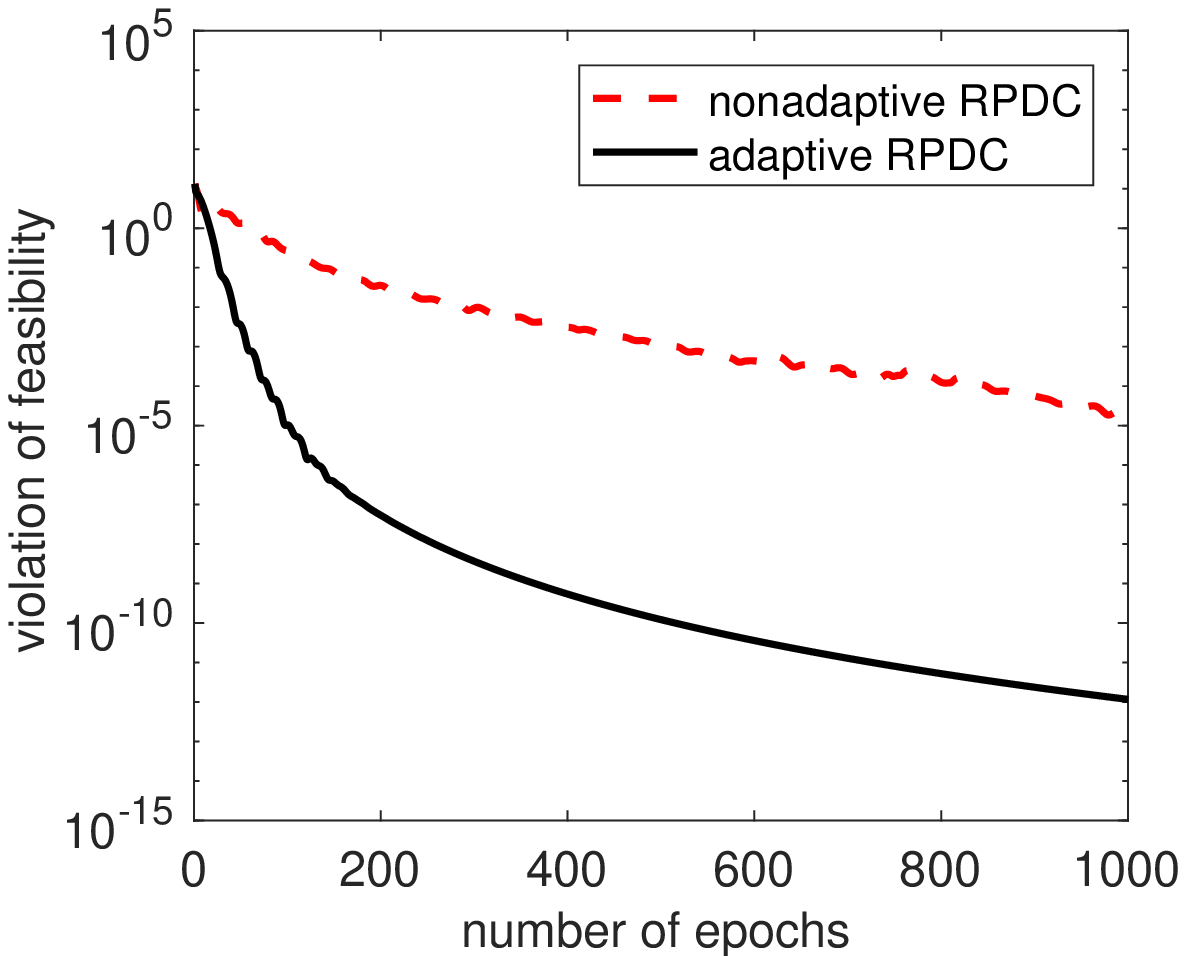}&
\includegraphics[width=0.2\textwidth]{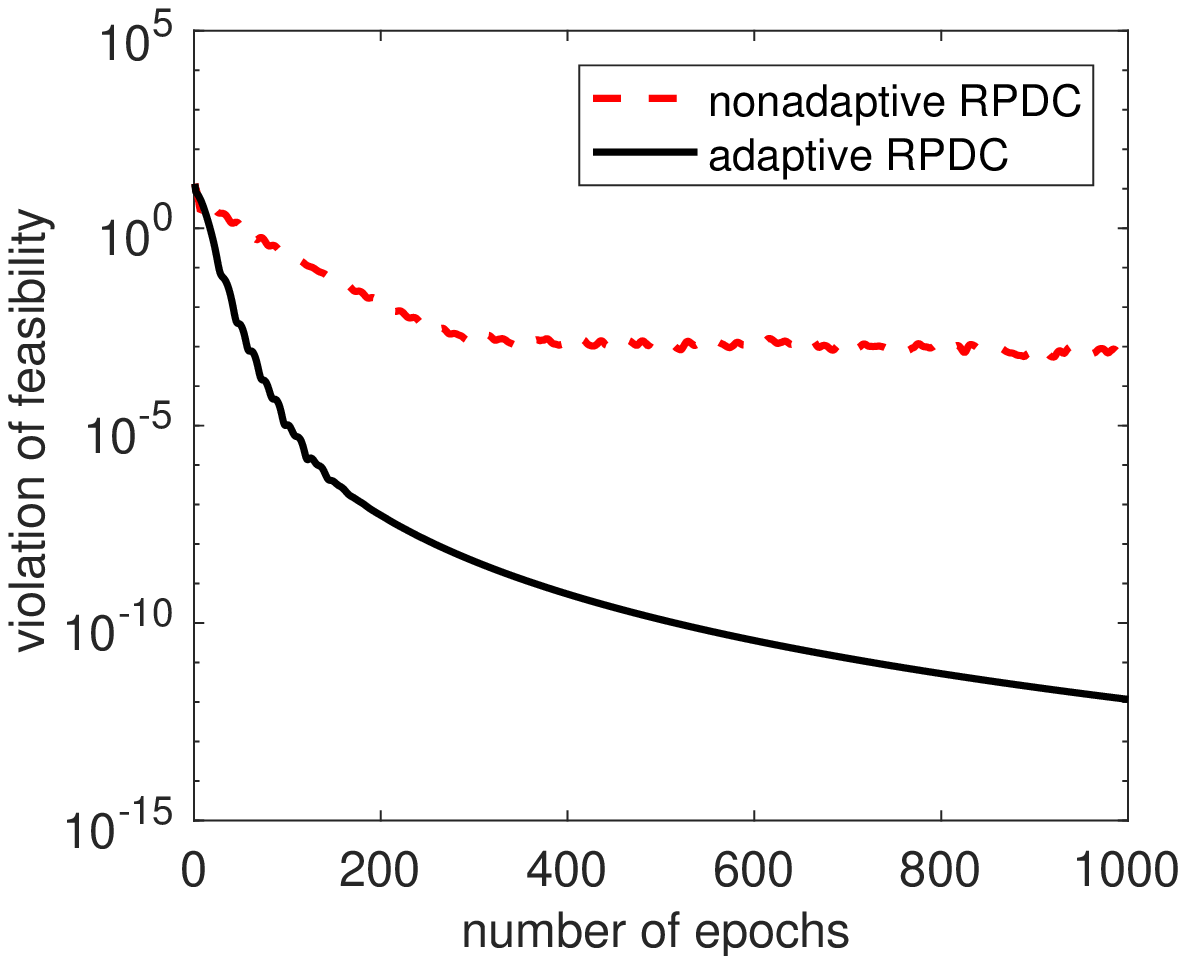}
\end{tabular}
\end{center}
\caption{Results by Algorithm \ref{alg:arpdc} with adaptive parameters and nonadaptive parameters for solving \eqref{eq:qp} with problem size $n=2000, p=200$ and condition number 10. The latter uses different penalty parameter $\beta$. Top row: difference of objective value to the optimal value $|F(x^k)-F(x^*)|$; bottom row: violation of feasibility $\|Ax^k-b\|$.}\label{fig:qp-p200-c10}
\end{figure}

\begin{figure}
\begin{center}
\begin{tabular}{cccc}
$\beta=1$ & $\beta=10$ & $\beta=100$ & $\beta=1000$ \\
\includegraphics[width=0.2\textwidth]{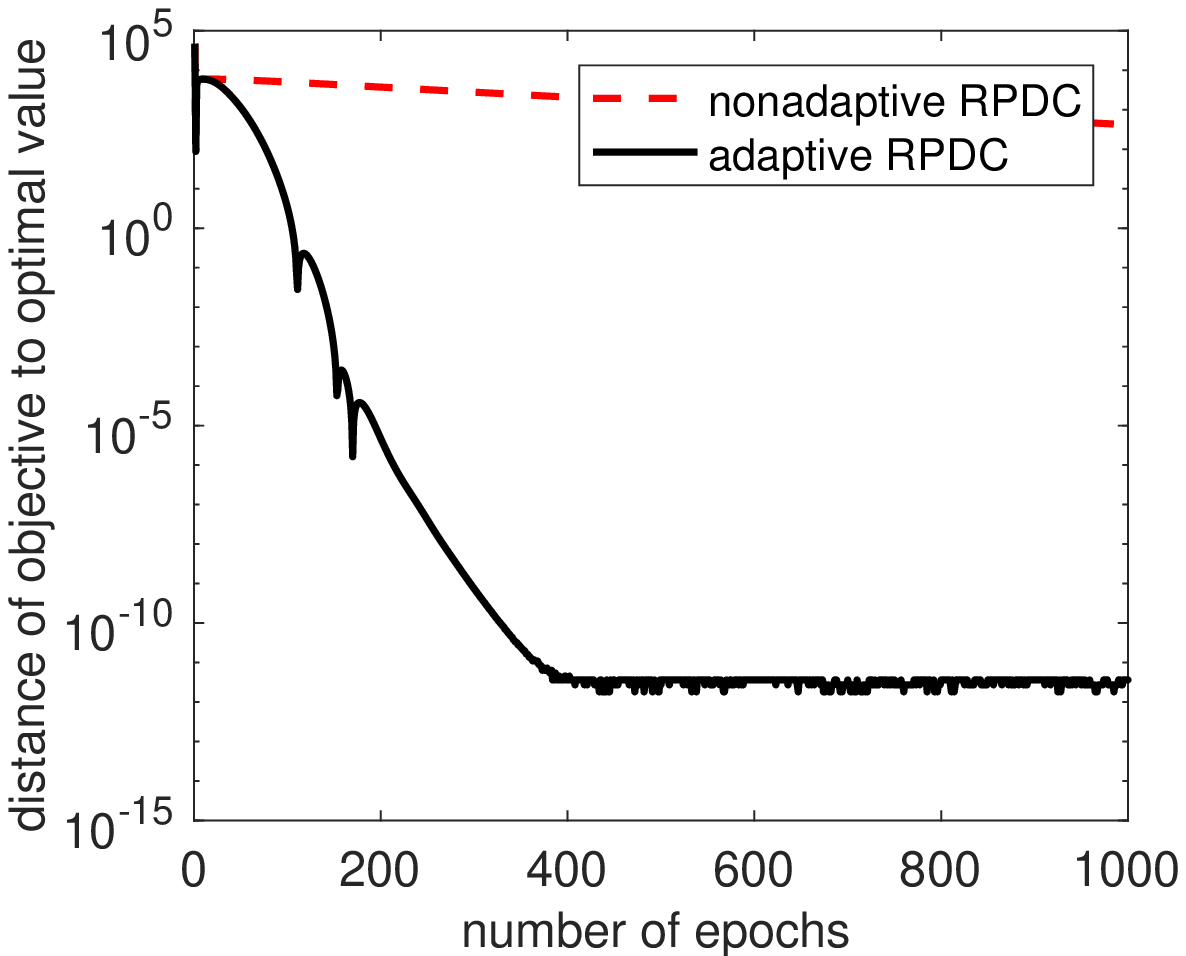}&
\includegraphics[width=0.2\textwidth]{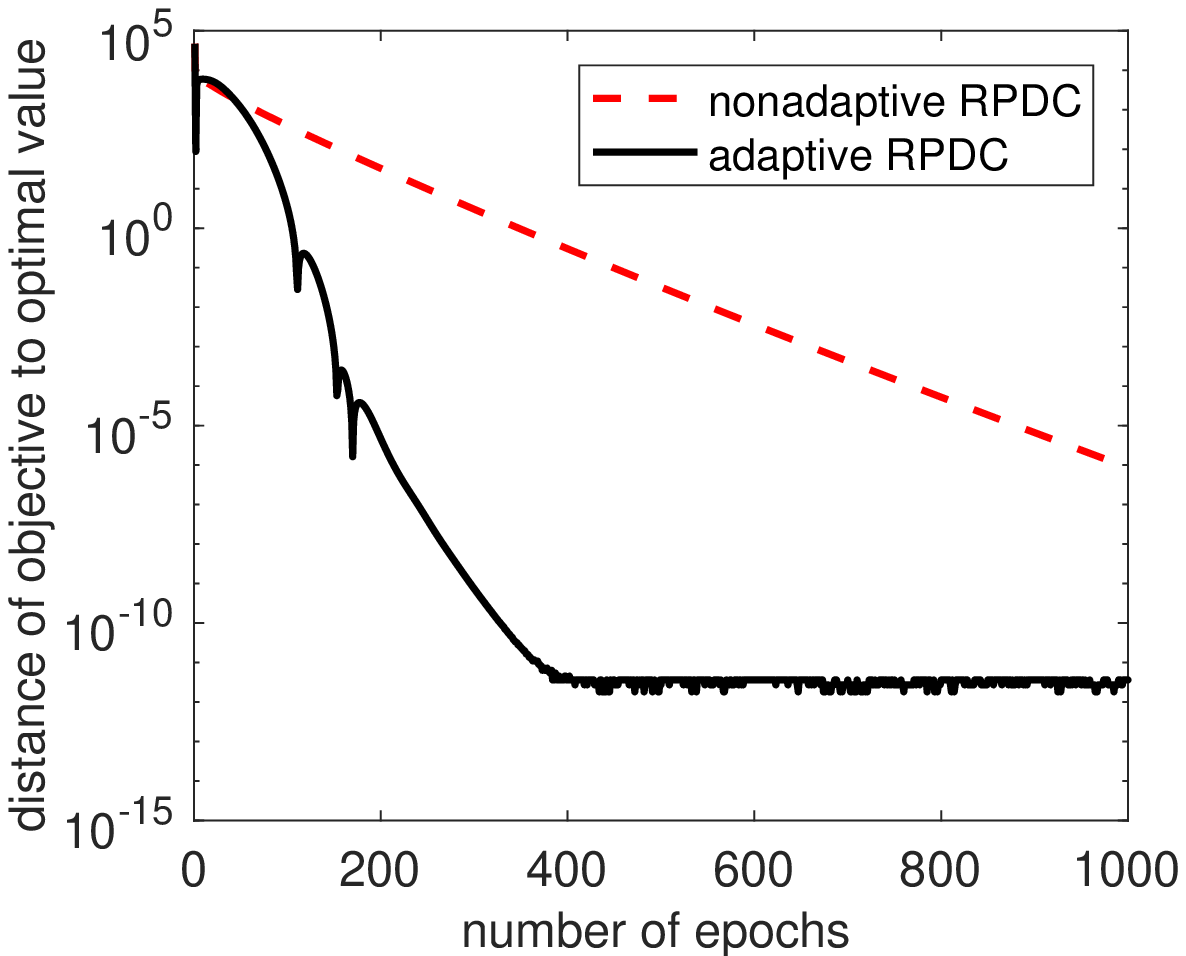}&
\includegraphics[width=0.2\textwidth]{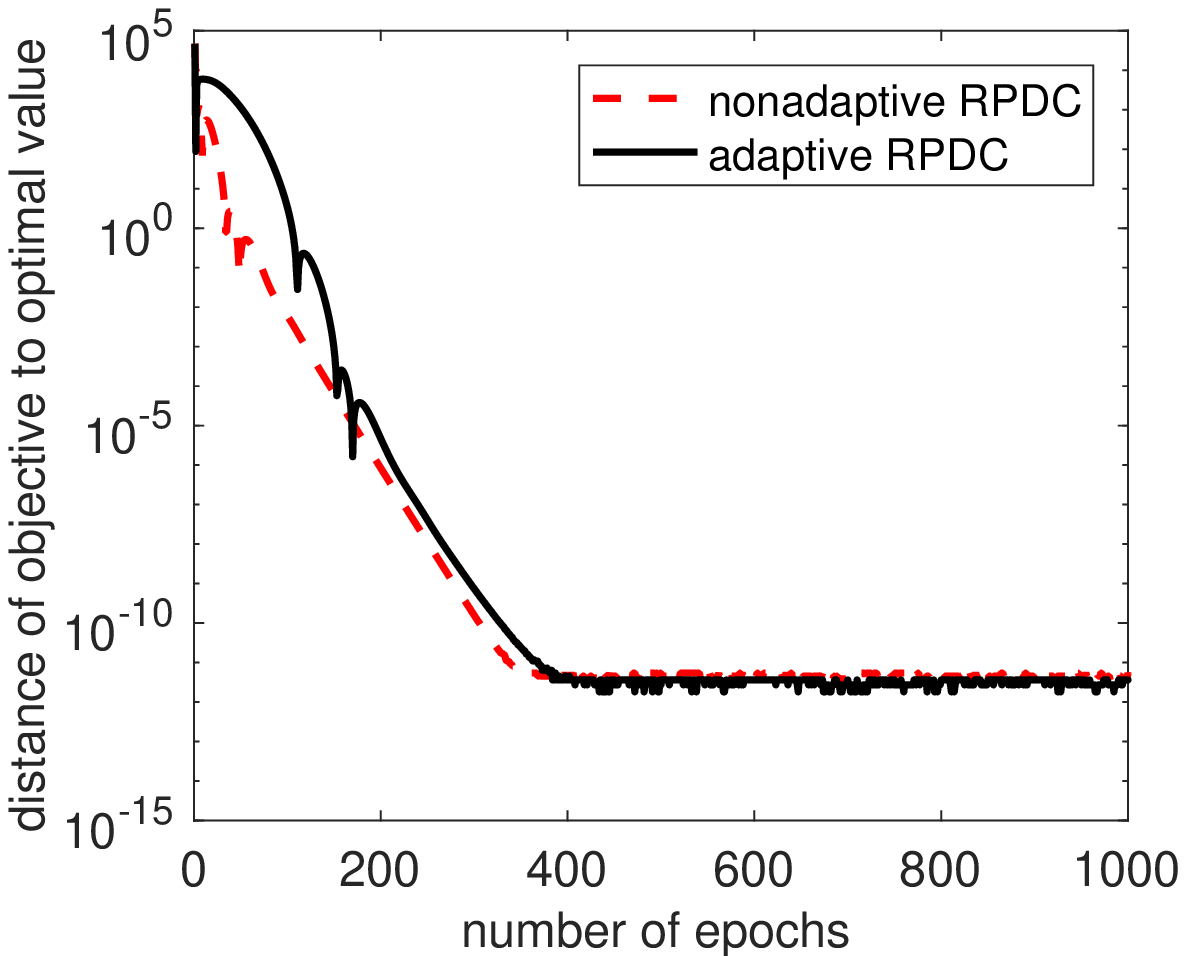}&
\includegraphics[width=0.2\textwidth]{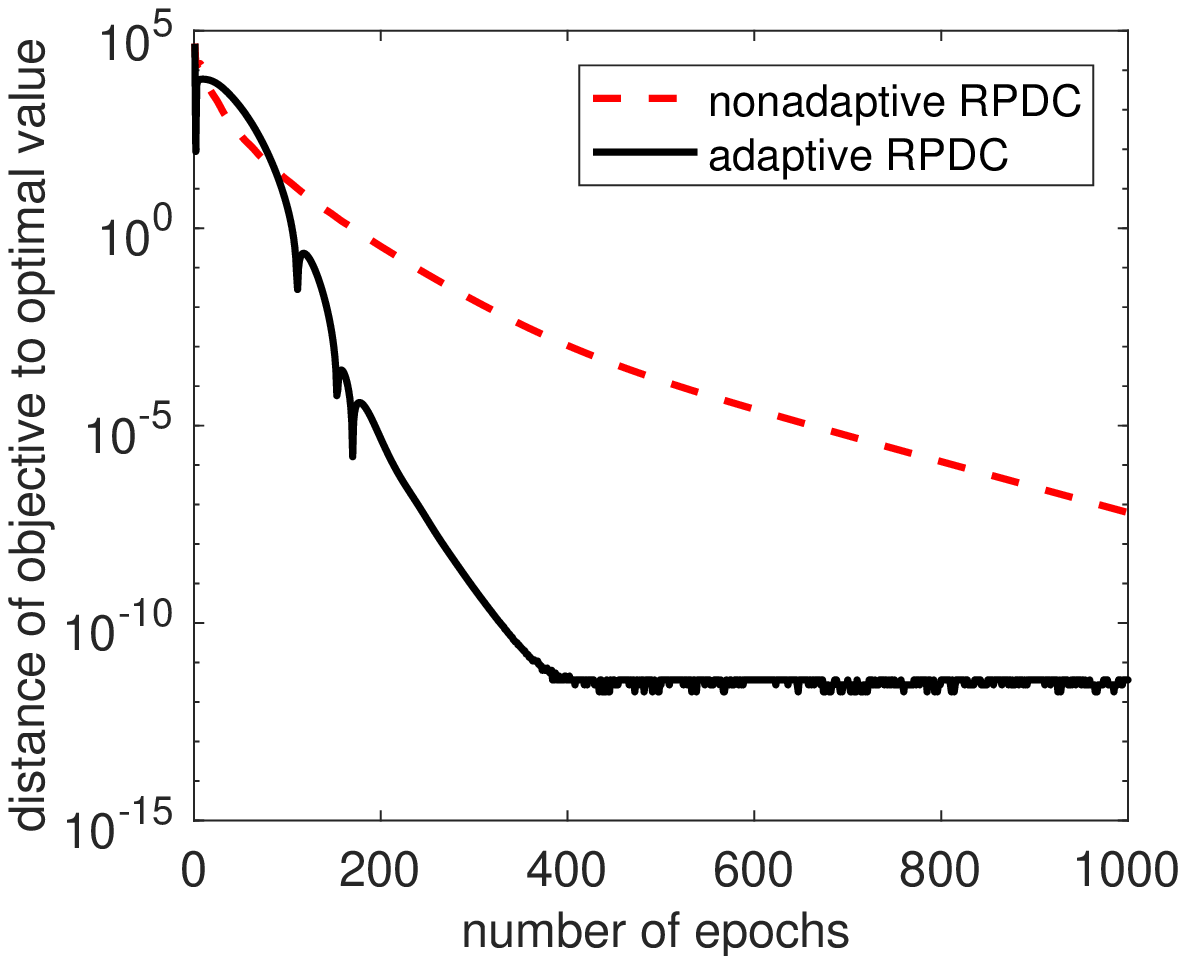}\\
\includegraphics[width=0.2\textwidth]{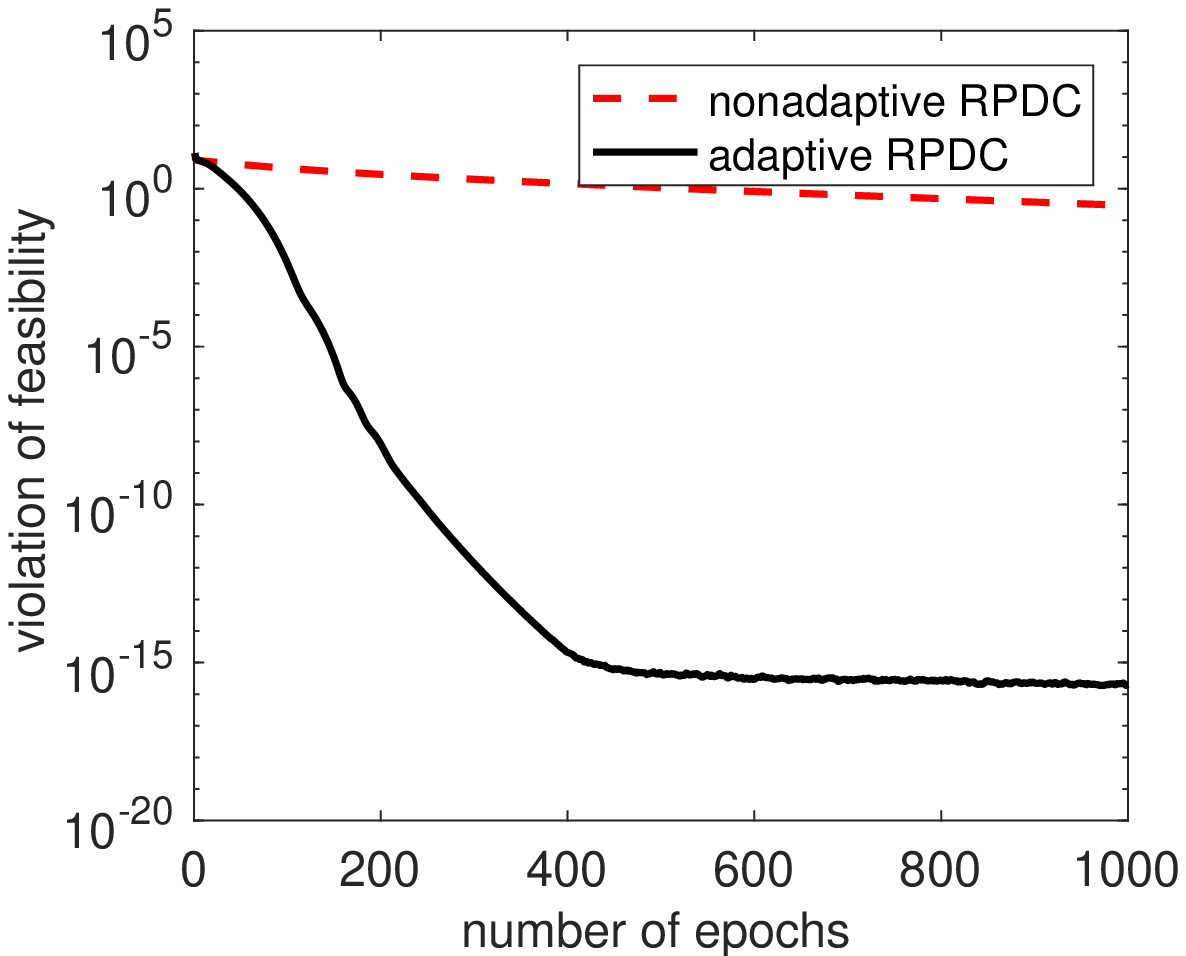}&
\includegraphics[width=0.2\textwidth]{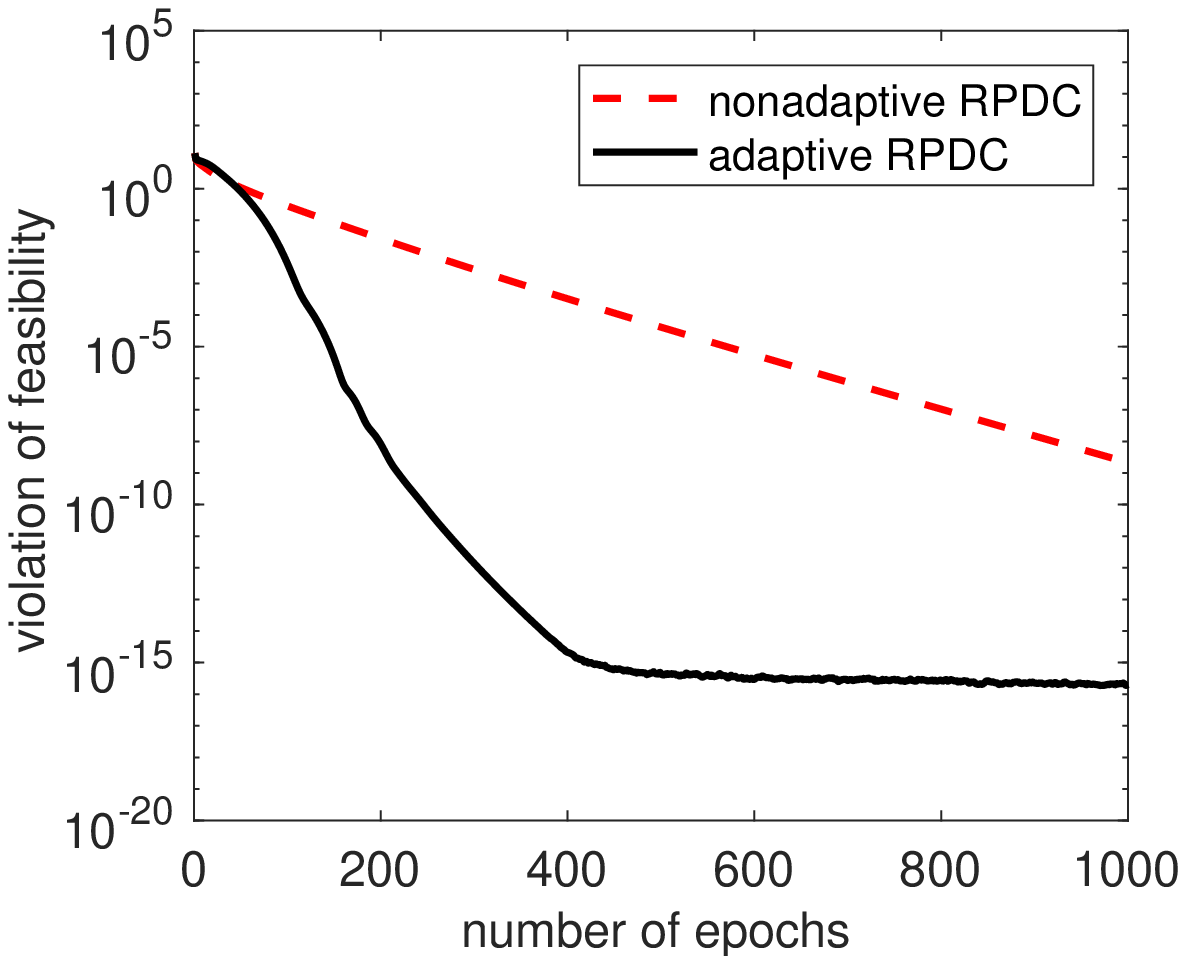}&
\includegraphics[width=0.2\textwidth]{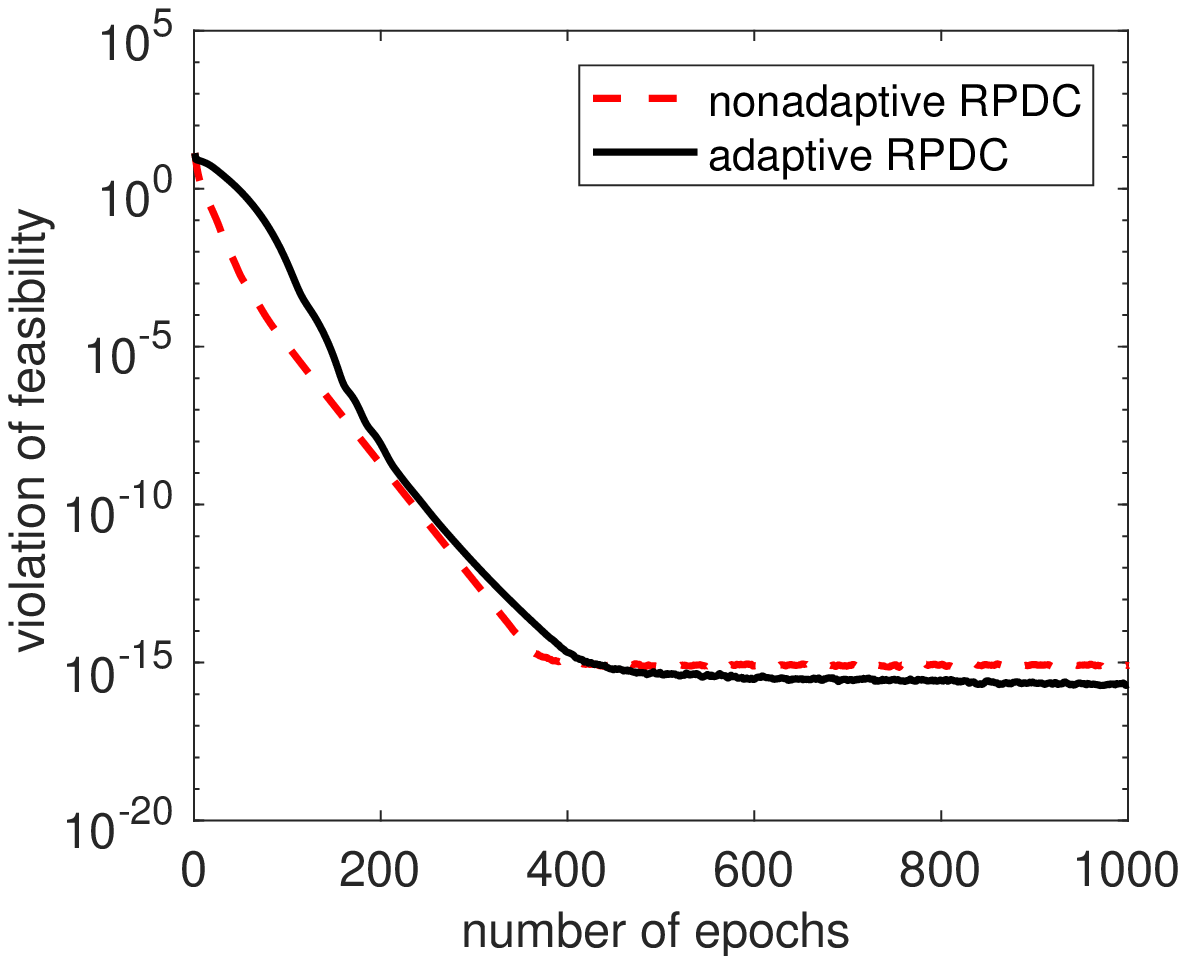}&
\includegraphics[width=0.2\textwidth]{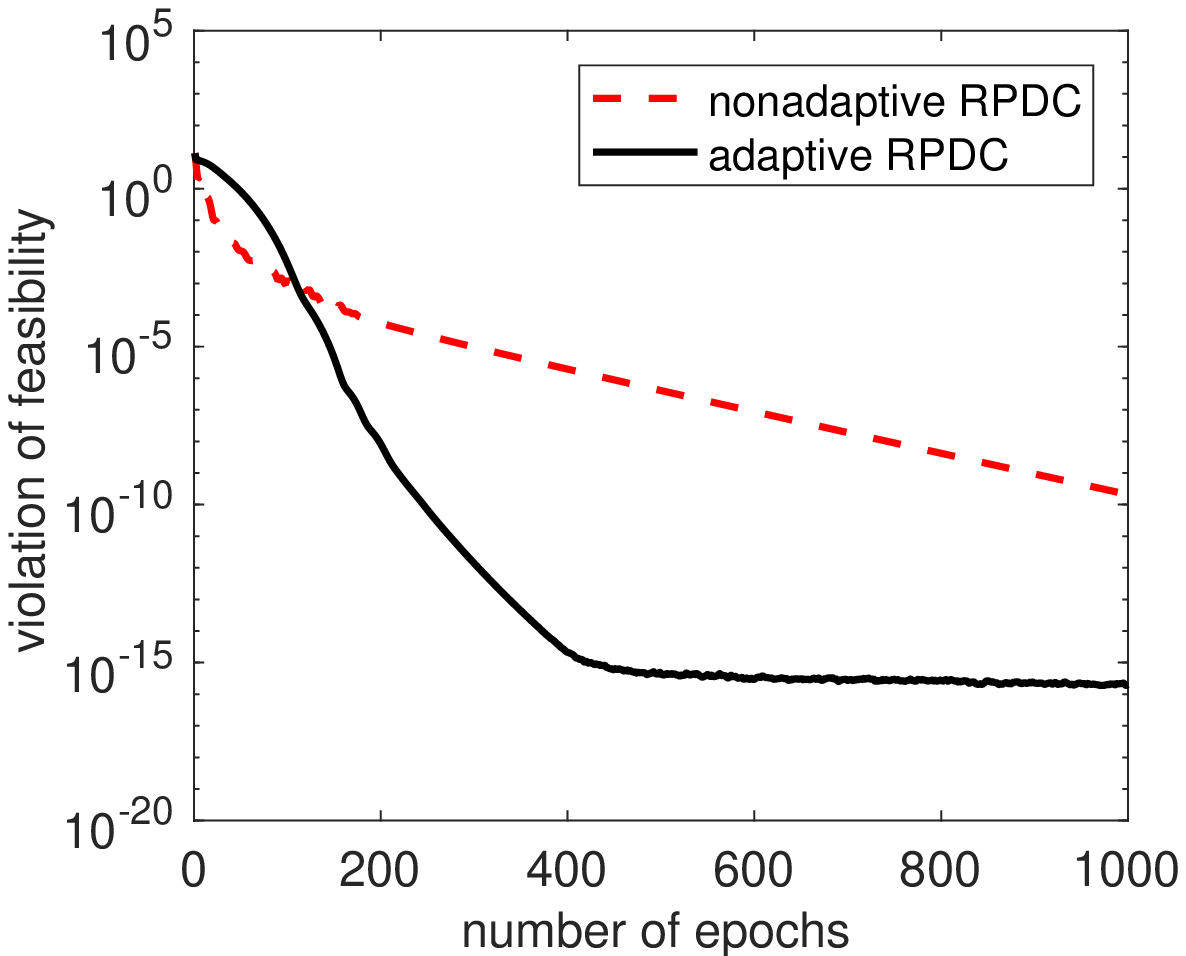}
\end{tabular}
\end{center}
\caption{Results by Algorithm \ref{alg:arpdc} with adaptive parameters and nonadaptive parameters for solving \eqref{eq:qp} with problem size $n=2000, p=200$ and condition number 100. The latter uses different penalty parameter $\beta$. Top row: difference of objective value to the optimal value $|F(x^k)-F(x^*)|$; bottom row: violation of feasibility $\|Ax^k-b\|$.}\label{fig:qp-p200-c100}
\end{figure}

\begin{figure}
\begin{center}
\begin{tabular}{cccc}
$\beta=1$ & $\beta=10$ & $\beta=100$ & $\beta=1000$ \\
\includegraphics[width=0.2\textwidth]{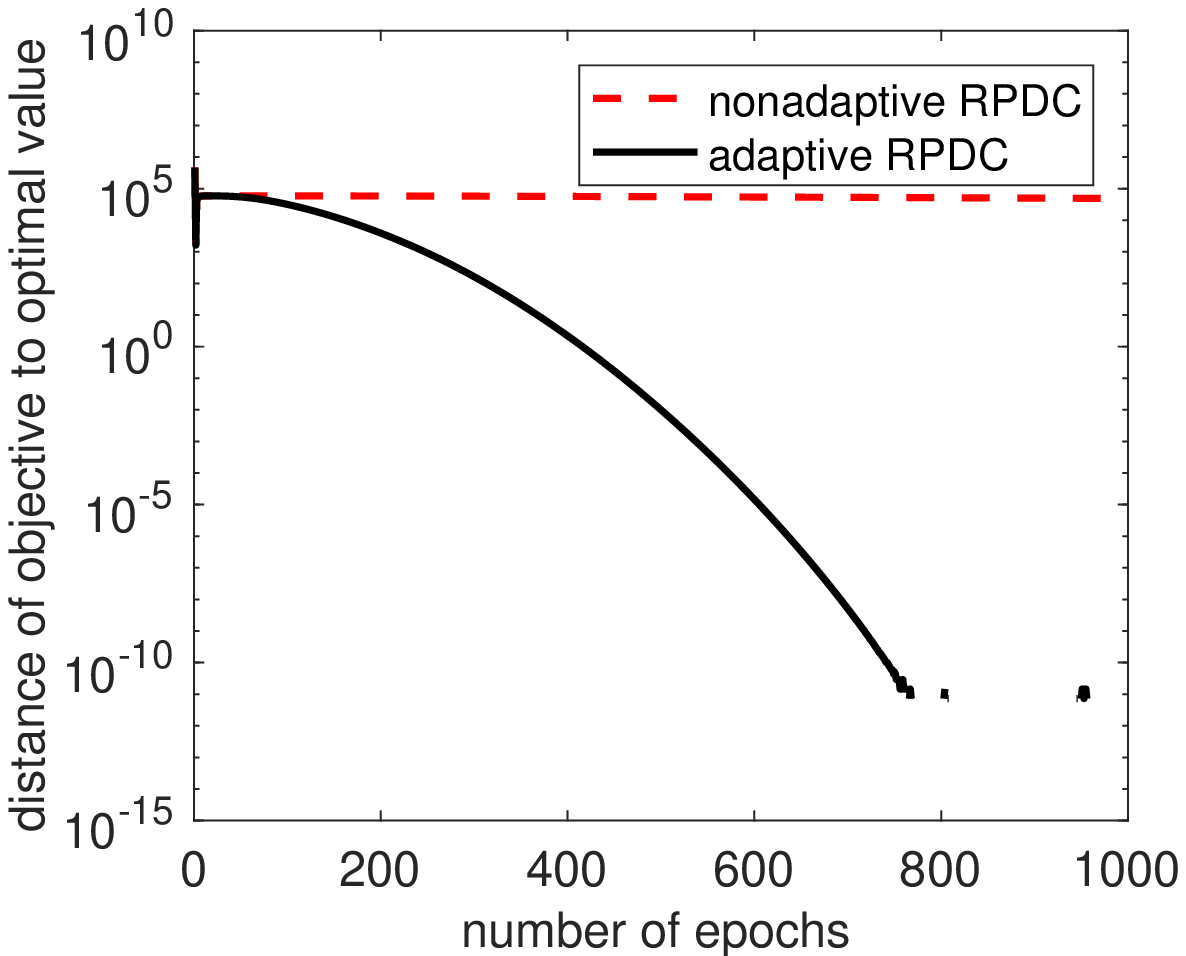}&
\includegraphics[width=0.2\textwidth]{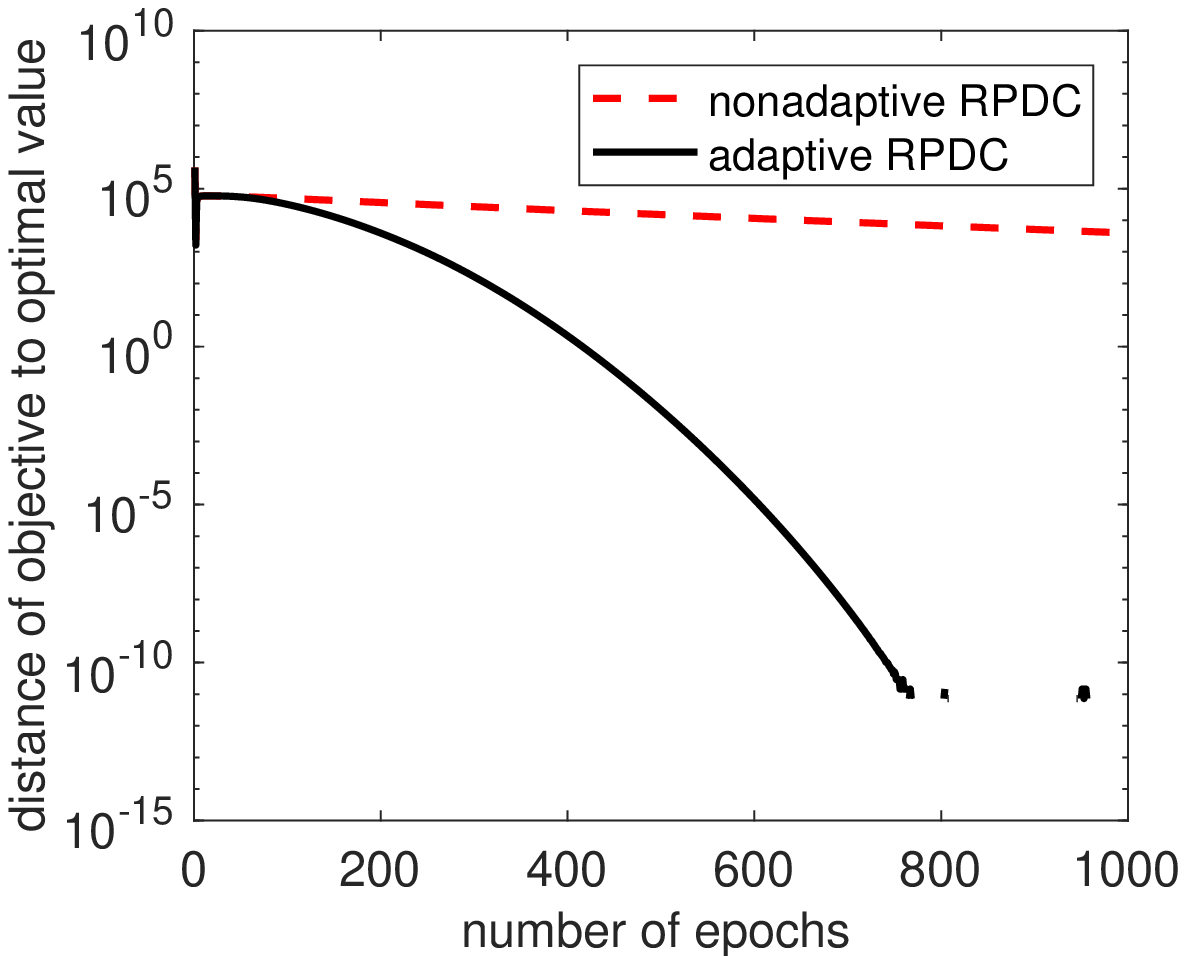}&
\includegraphics[width=0.2\textwidth]{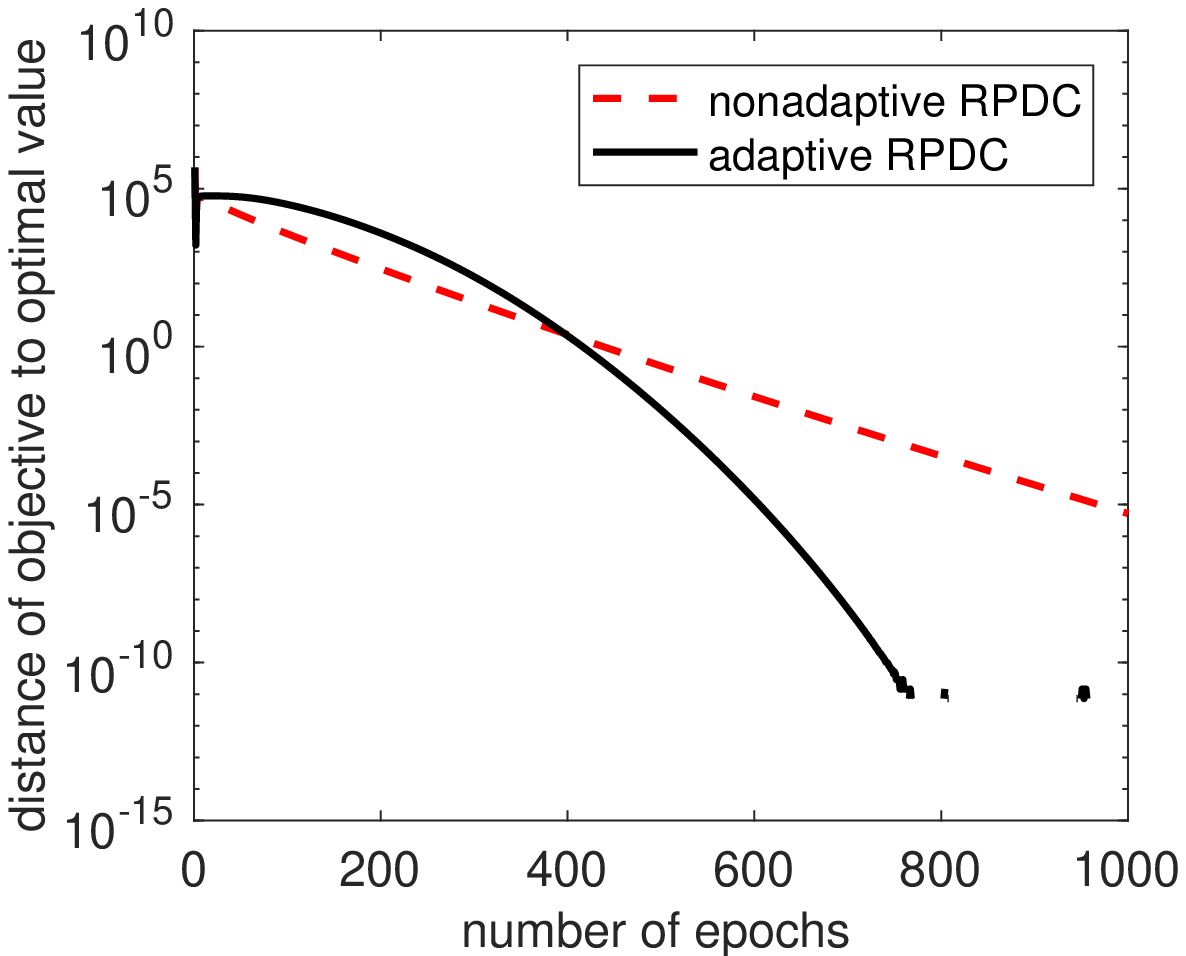}&
\includegraphics[width=0.2\textwidth]{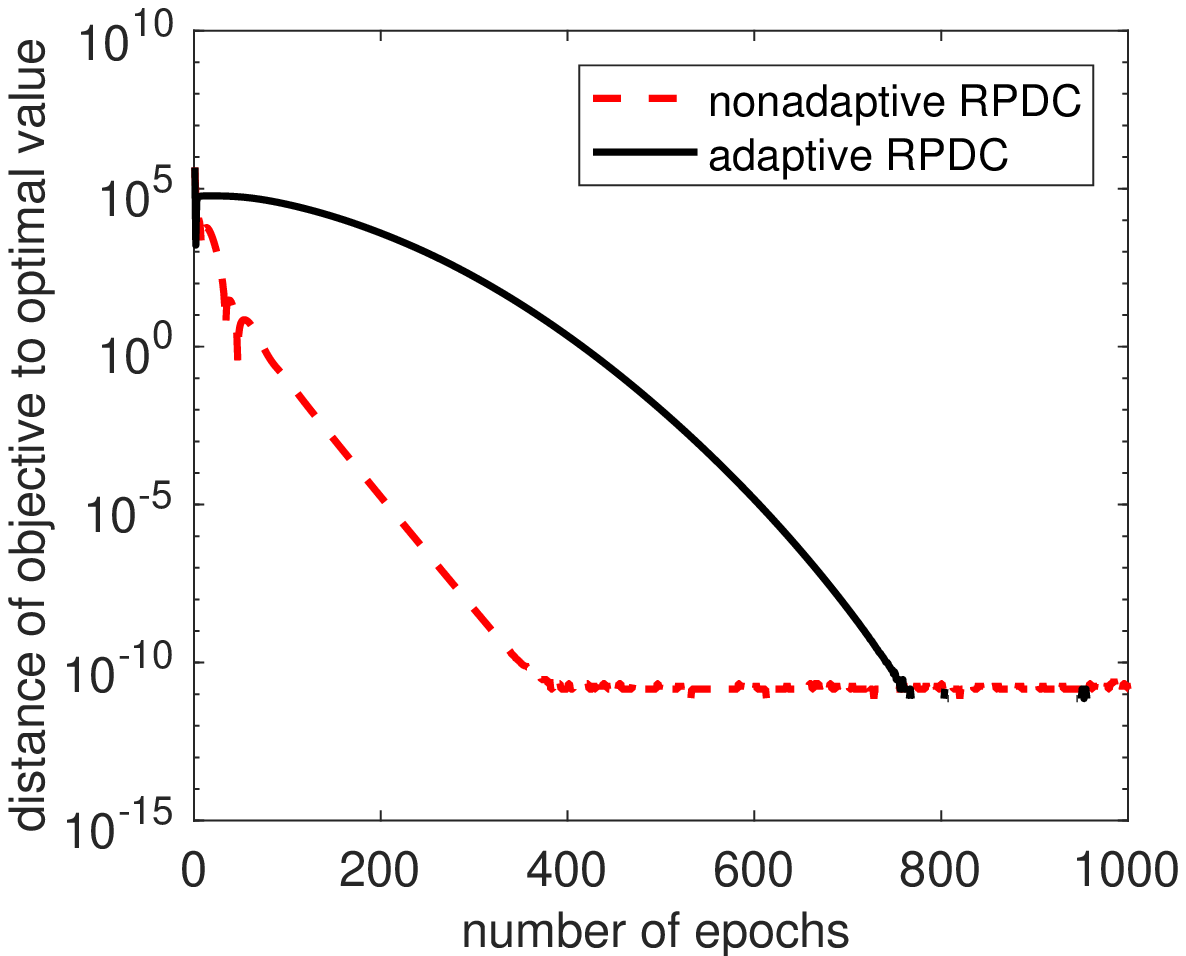}\\
\includegraphics[width=0.2\textwidth]{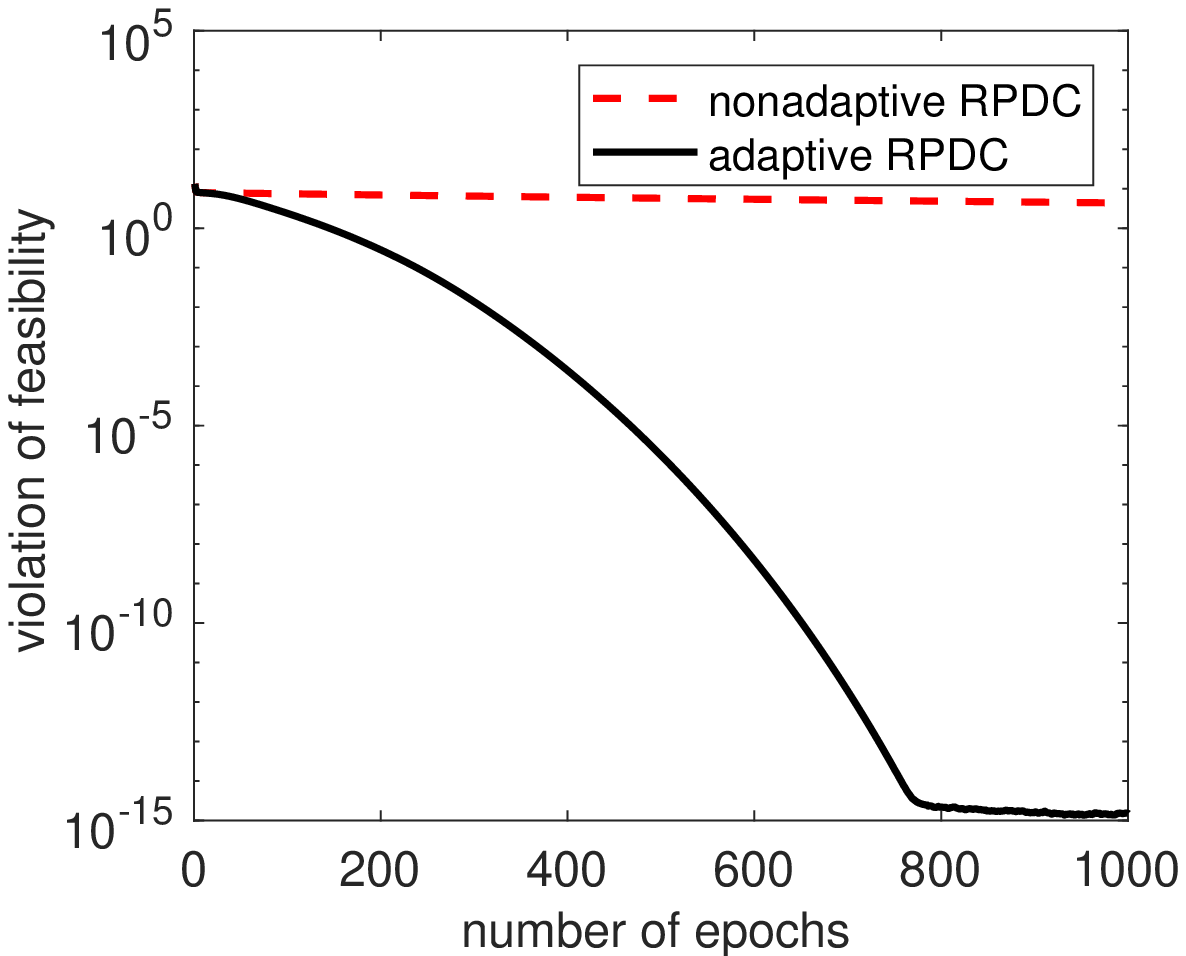}&
\includegraphics[width=0.2\textwidth]{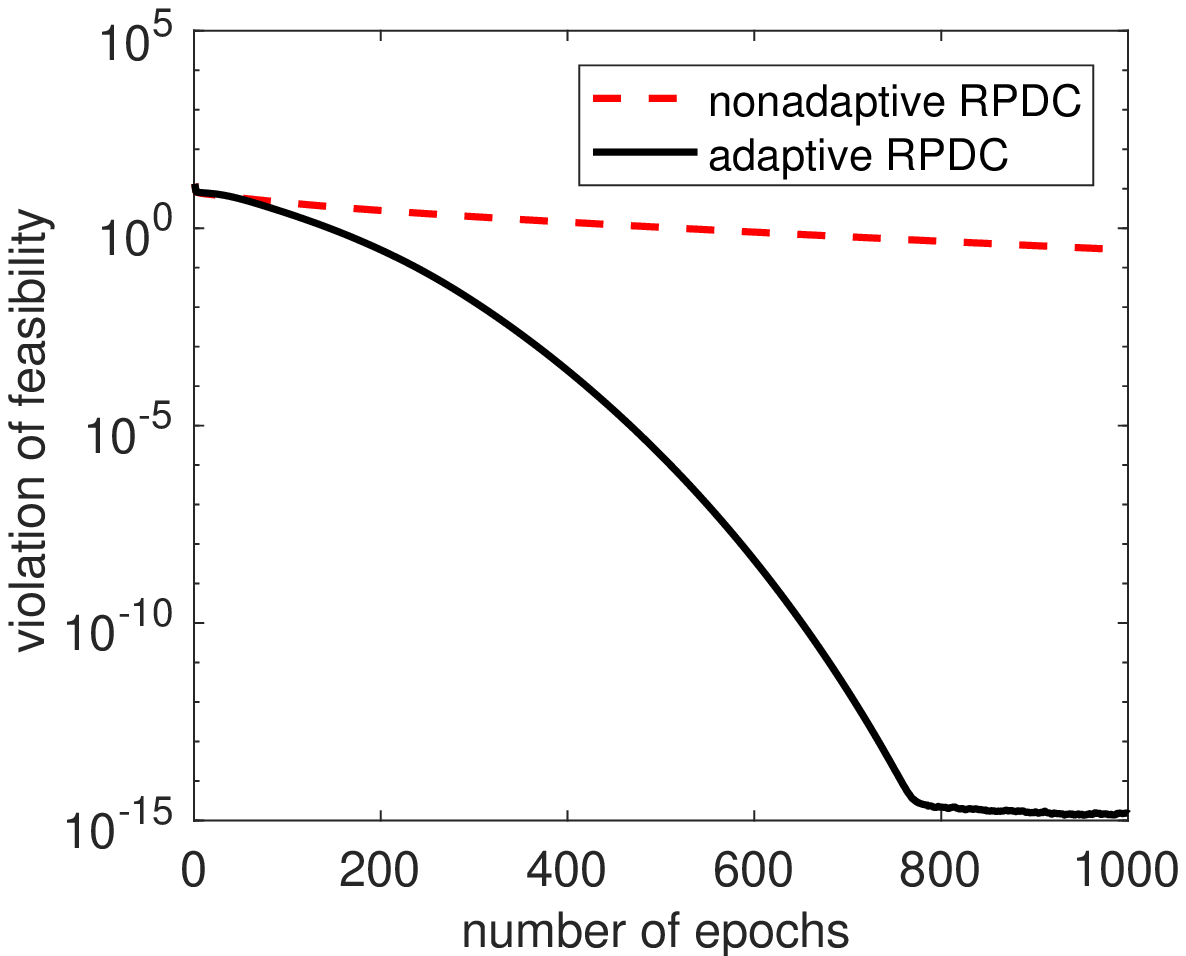}&
\includegraphics[width=0.2\textwidth]{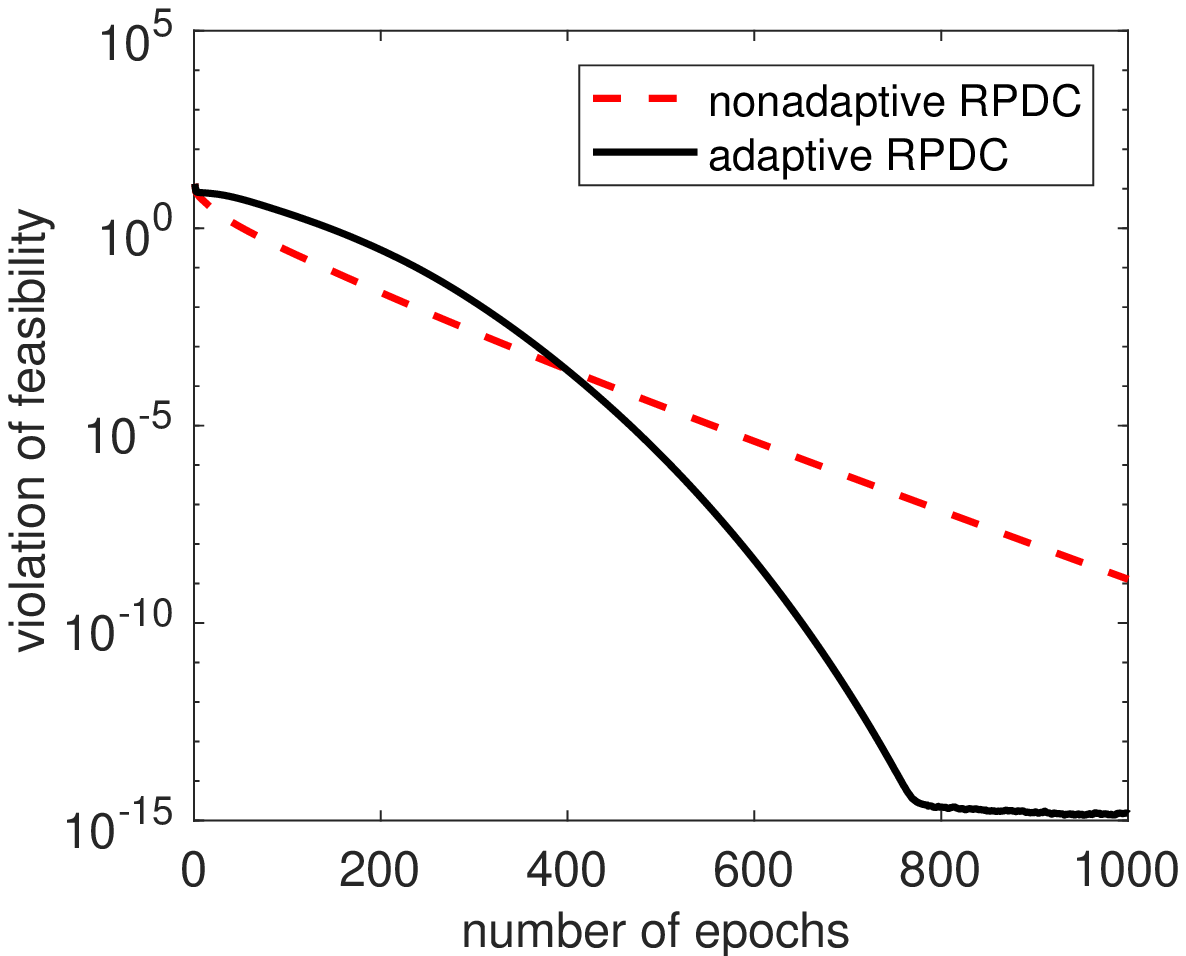}&
\includegraphics[width=0.2\textwidth]{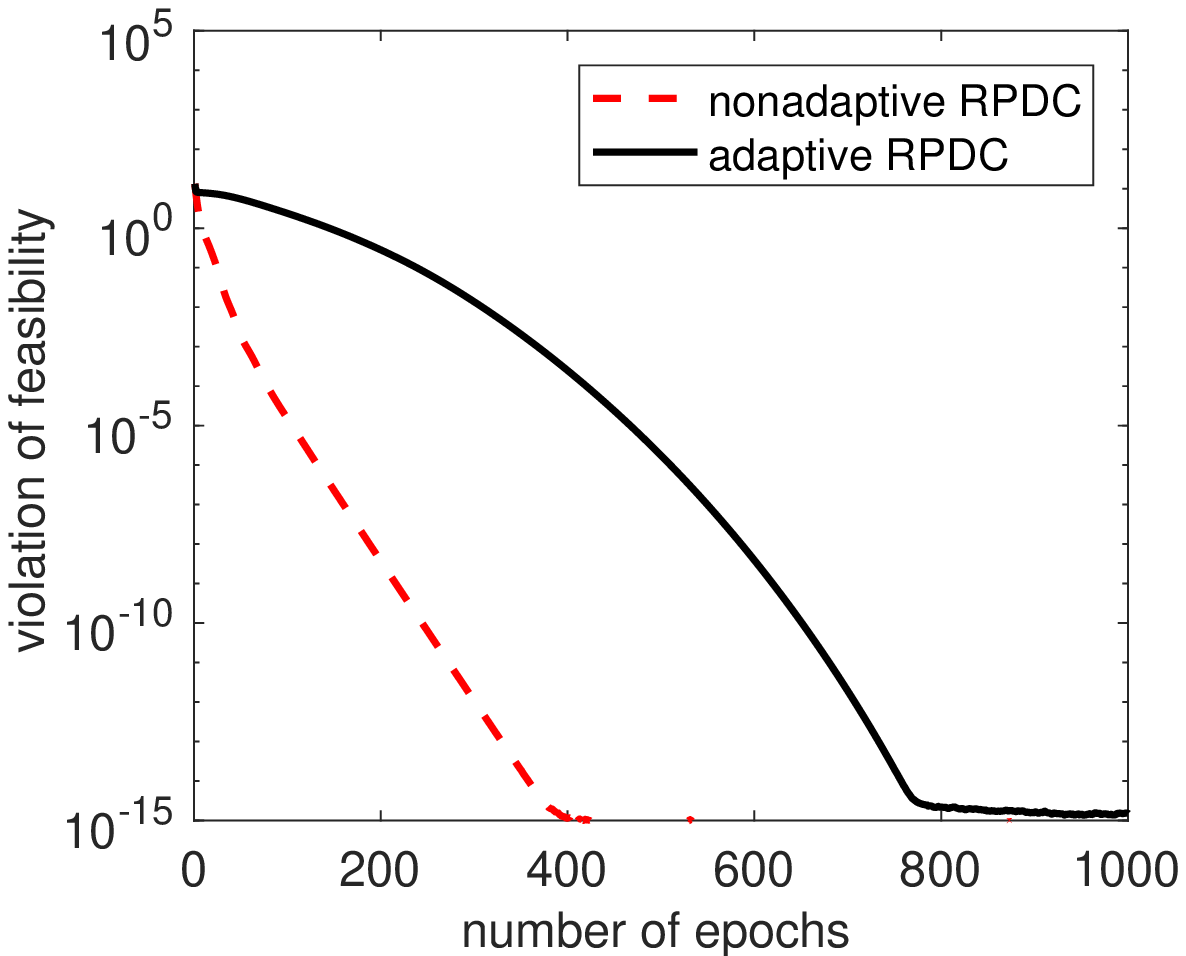}
\end{tabular}
\end{center}
\caption{Results by Algorithm \ref{alg:arpdc} with adaptive parameters and nonadaptive parameters for solving \eqref{eq:qp} with problem size $n=2000, p=200$ and condition number 1000. The latter uses different penalty parameter $\beta$. Top row: difference of objective value to the optimal value $|F(x^k)-F(x^*)|$; bottom row: violation of feasibility $\|Ax^k-b\|$.}\label{fig:qp-p200-c1000}
\end{figure}

{
\textbf{Linear programming.} In this test, we apply Algorithm \ref{alg:rpdc-lin} to the problem \eqref{eq:log-b-lp}, where we let $f(x)=c^\top x, g(x)=-e^\top \log x$ and $h(y)= -e^\top \log y$. The purpose of this experiment is to demonstrate the linear convergence of Algorithm \ref{alg:rpdc-lin}. 

We generated $A\in\RR^{200\times 2000}$ and $c$ according to the standard Gaussian distribution and $b$ by the uniform distribution on $[\frac{1}{2},\frac{3}{2}]$. The upper bound was set to $u_i=10,\forall i$. We treated $x$ as a single block and set the algorithm parameters to $\beta=0.1$, $\eta_x=\beta\|A\|_2^2$, and $\eta_y=\beta\big(1+\frac{2.001\beta}{3\mu}\big)$. This setting satisfies the conditions required in Theorem \ref{thm-linear} if $\alpha$ is sufficiently close to 1.
Note that $g$ and $h$ do not have uniform strong convexity constants but they are both strongly convex on a bounded set.  Figure \ref{fig:lp} shows the convergence behavior of Algorithm \ref{alg:rpdc-lin}. From the figure, we can clearly see that the algorithm linearly converges to an optimal solution.
}

\begin{figure}[h]
\begin{center}
\includegraphics[width=0.32\textwidth]{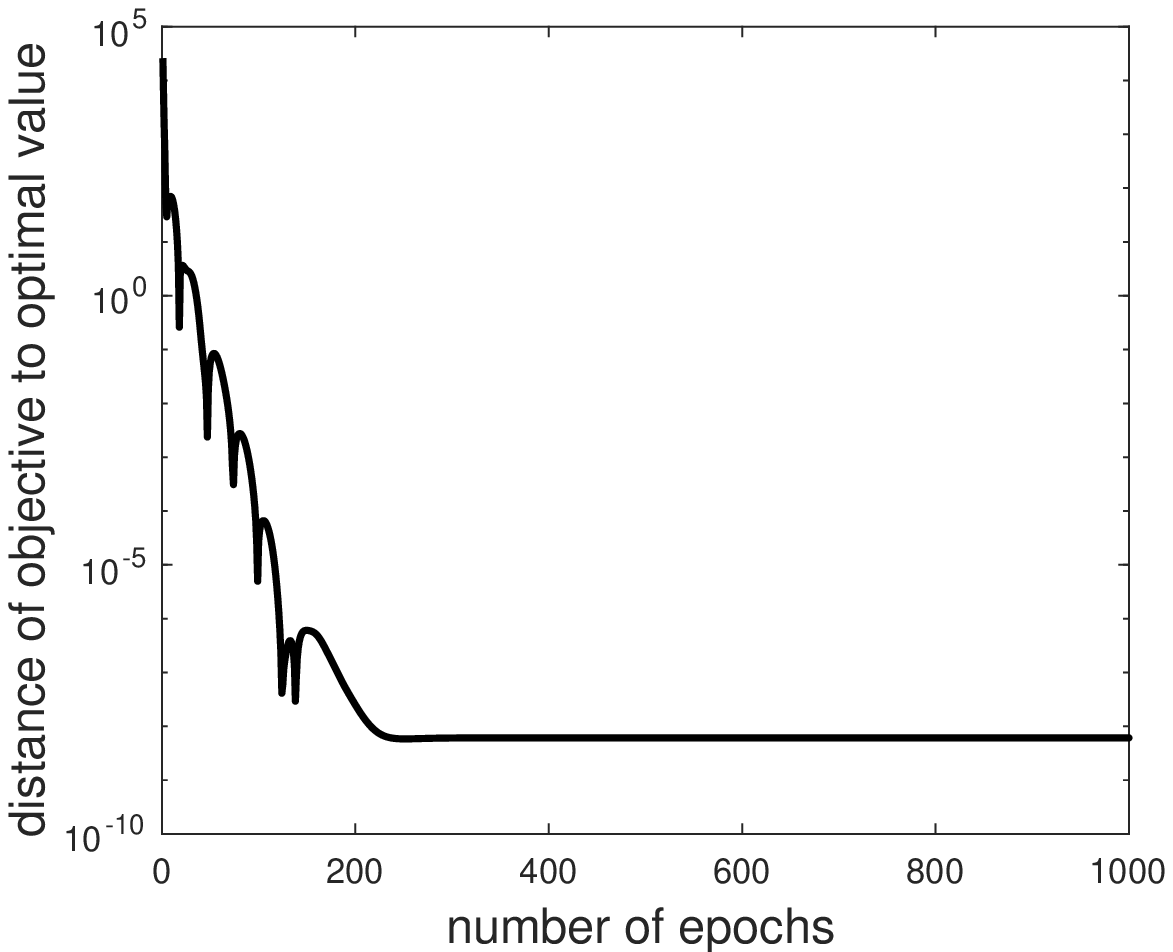}
\includegraphics[width=0.32\textwidth]{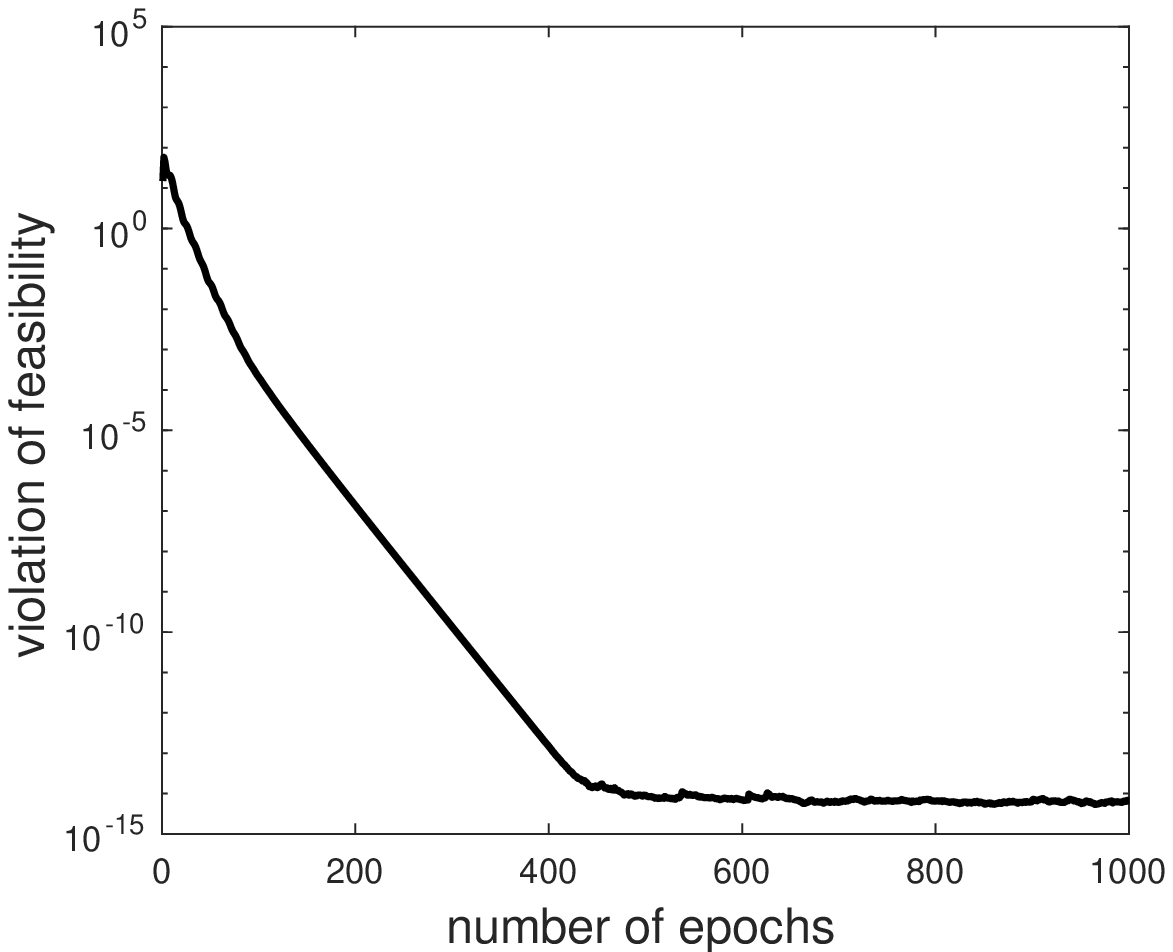}
\end{center}
\caption{Results by Algorithm \ref{alg:rpdc-lin} on the problem \eqref{eq:log-b-lp} with $A\in \RR^{200\times 2000}$. Left: difference of objective value to the optimal value $|F(x^k)+h(y^k)-F(x^*)-h(y^*)|$; Right: violation of feasibility $\|Ax^k+By^k-b\|$}\label{fig:lp}
\end{figure}

\section{Conclusions}\label{sec:conclusion}
In this paper we propose an accelerated proximal Jacobian ADMM method and generalize it to an accelerated randomized primal-dual coordinate updating method for solving linearly constrained multi-block structured convex programs. We show that if the objective is strongly convex then the methods achieve $O(1/t^2)$ convergence rate where $t$ is the total number of iterations. In addition, if one block variable is independent of others in the objective and its part of the objective function is smooth, we have modified the primal-dual coordinate updating method to achieve linear convergence. Numerical experiments on quadratic programming and log-barrier approximation of linear programming have shown the efficacy of the newly proposed methods.

\bibliographystyle{abbrv}
\bibliography{alm,bcd,mblk-adm}

%\newpage

\appendix

\section{Technical proofs: Section \ref{sec:ajadmm}}
In this section, we give the detailed proofs of the lemmas and theorems in section \ref{sec:ajadmm}. The following lemma will be used a few times. Note that when $S=[M]$, the result is deterministic.
\begin{lemma}\label{lem:S-basic-ineq}
Let $S$ be a uniformly selected subset of $[M]$ with cardinality $m$ and $x^o$ be a vector independent of $S$. Suppose $x^+$ is a random vector dependent on $S$ and its coordinates out of $S$ are the same as $x^o$. Let $\beta\in\RR$, $\lambda^o$ and $r^o$ be vectors independent of $S$, and $W$ a positive semidefinite $M\times M$ block diagonal matrix. If 
$$\nabla_{S} f(x^o)+\tilde{\nabla}g_S(x_S^+)-A_S^\top(\lambda^o- \beta  r^o)+W_S(x_S^+-x_S^o)=0,$$
 then for any $x$, it holds that
\begin{equation}
\begin{aligned}\label{eq:S-basic-ineq}
&~ \EE_S\left[F(x^+)-F(x)+\frac{\mu}{2}\|x^+-x\|^2-\left\langle A(x^+-x), \lambda^o-\beta r^o\right\rangle\right] \\
\le &~ (1- \theta)\left[F(x^o)-F(x)+\frac{\mu}{2}\|x^o-x\|^2-\big\langle A(x^o-x), \lambda^o-\beta r^o\big\rangle\right]\\
&~-\frac{1}{2}\EE_S\left[\|x^+-x\|_W^2-\|x^o-x\|_W^2+\|x^+-x^o\|_{W-L_m I}^2\right],
\end{aligned}
\end{equation}
where $\theta=\frac{m}{M}$, $L_m$ is given in Assumption \ref{assump:lip-F}, and the expectation is taken on $S$.
\end{lemma}

\begin{proof}
For any $x$, we have
$$\left\langle x_S^+-x_S, \nabla_{S} f(x^o)+\tilde{\nabla}g_S(x_S^+)-A_S^\top(\lambda^o- \beta  r^o)+W_S(x_S^+-x_S^o)\right\rangle=0.$$
We split the left hand side of the above equation into four terms and bound each of them as below. First, we have 
\begin{align}\label{eq:S-f-term}
&~\EE_S\left\langle x^+_S-x_S, \nabla_{S} f(x^o)\right\rangle\cr
=&~\EE_S\left\langle x^+-x^o, \nabla f(x^o)\right\rangle + \EE_S\left\langle x^o_S-x_S, \nabla_S f(x^o)\right\rangle\cr
\ge & ~\EE_S \left[f(x^+)-f(x^o)-\frac{L_m}{2}\|x^+-x^o\|^2\right] + \theta[f(x^o)-f(x)]\cr
=&~ \EE_S\left[f(x^+)-f(x)-\frac{L_m}{2}\|x^+-x^o\|^2\right] - (1- \theta)[f(x^o)-f(x)],
\end{align} 
where the first equality uses the fact $x_i^+=x_i^o,\,\forall i\not\in S$, and the inequality follows from the uniform distribution of $S$, the convexity of $f$, and also the inequality \eqref{eq:S-lip-ineq}.

Secondly, it follows from the strong convexity of $g$ that
\begin{equation}\label{eq:S-g-term-1}
\left\langle x^+_S-x_S, \tilde{\nabla} g_S(x_S^+)\right\rangle\ge g_S(x_S^+) - g_S(x_S) + \sum_{i\in S}\frac{\mu}{2}\|x_i^+-x_i\|^2.
\end{equation}
Since $g_S(x_S^+) - g_S(x_S)=g(x^+) - g(x^o) + g_S(x_S^o)- g_S(x_S)$ and
$\EE_S[g_S(x_S^o)- g_S(x_S)]=\theta [g(x^o)-g(x)]$, we have
\begin{align}\label{eq:S-g-term-2}
\EE_S[g_S(x_S^+) - g_S(x_S)]=&~\EE_S[g(x^+) - g(x^o)]+\theta [g(x^o)-g(x)] \cr
=&~\EE_S[g(x^+)-g(x)]-(1-\theta) [g(x^o)-g(x)].
\end{align}
Similarly, it holds
$\EE_S\sum_{i\in S}\frac{\mu}{2}\|x_i^+-x_i\|^2=\frac{\mu}{2}\left(\EE_S\|x^+-x\|^2-(1-\theta)\|x^o-x\|^2\right).$
Hence, taking expectation on both sides of \eqref{eq:S-g-term-1} yields
\begin{align}\label{eq:S-g-term}
&~\EE_S\left\langle x^+_S-x_S, \tilde{\nabla} g_S(x_S^+)\right\rangle\cr
\ge &~\EE_S \left[ g(x^+) - g(x)+\frac{\mu}{2}\|x^+-x\|^2\right]-(1-\theta)\left[ g(x^o) - g(x)+\frac{\mu}{2}\|x^o-x\|^2\right]. 
\end{align}

Thirdly, by essentially the same arguments on showing \eqref{eq:S-g-term-2}, we have
\begin{equation}\label{eq:S-lam-term}
\EE_S \left\langle x^+_S-x_S, -A_S^\top (\lambda^o-\beta r^o)\right\rangle = -\EE_S \left\langle A(x^+-x), \lambda^o-\beta r^o\right\rangle + (1-\theta) \big\langle A(x^o-x), \lambda^o-\beta r^o\big\rangle.
\end{equation}
Fourth, note $\left\langle x^+_S-x_S, W_S(x_S^+-x_S^o)\right\rangle=\left\langle x^+-x, W(x^+-x^o)\right\rangle$, and thus by \eqref{uv-cross},
\begin{equation}\label{eq:S-P-term}
\EE_S\left\langle x^+_S-x_S, W_S(x_S^+-x_S^o)\right\rangle=\frac{1}{2}\EE_S\left[\|x^+-x\|_W^2-\|x^o-x\|_W^2+\|x^+-x^o\|_W^2\right].
\end{equation}
The desired result is obtained by adding \eqref{eq:S-f-term}, \eqref{eq:S-g-term}, \eqref{eq:S-lam-term}, and \eqref{eq:S-P-term}, and recalling $F=f+g$.
\end{proof}

\subsection{Proof of Lemma \ref{lem:1iter-ajadmm}}
From \eqref{eq:ajadmm-x}, we have the optimality condition
%$$0\in\nabla_i f(x^k)-A_i^\top(\lambda^k-\beta_kr^k)+\partial g_i(x_i^{k+1})+P_i^k(x_i^{k+1}-x_i^k),\, i=1,\ldots,M,$$
%or equivalently there is $\tilde{\nabla}g(x^{k+1})\in\partial g(x^{k+1})$,
$$\nabla f(x^k)-A^\top(\lambda^k-\beta_kr^k)+\tilde{\nabla} g(x^{k+1})+P^k(x^{k+1}-x^k)=0.$$
Hence, for any $x$ such that $Ax=b$, it follows from the definition of $\Phi$ in \eqref{eq:nota-Phi} and Lemma \ref{lem:S-basic-ineq} with $S=[M]$, $x^o=x^k$, $\lambda^o=\lambda^k$, $\beta=\beta_k$, $x^+=x^{k+1}$, and $W=P^k$ that
%\begin{align}
%0 \ge&~F(x^{k+1})-F(x)-\frac{L_f}{2}\|x^{k+1}-x^k\|^2+\frac{\mu}{2}\|x^{k+1}-x\|^2 \nonumber \\
%&~-\left\langle Ax^{k+1}-b,\lambda^k-\beta_kr^k\right\rangle +\left\langle x^{k+1}-x,P^k(x^{k+1}-x^k)\right\rangle.\label{eq:opt-ineq}
%\end{align}
%Rearranging \eqref{eq:opt-ineq} and recalling the definition of $\Phi$ in \eqref{eq:nota-Phi} give
\begin{align}\label{eq:opt-ineq2}
\Phi(x^{k+1},x,\lambda)
\le &~ \left\langle Ax^{k+1}-b,\lambda^k-\beta_kr^k\right\rangle - \left\langle Ax^{k+1}-b, \lambda\right\rangle\cr
&~-\frac{1}{2}\EE_S\left[\|x^{k+1}-x\|_{P^k+\mu I}^2-\|x^k-x\|_{P^k}^2+\|x^{k+1}-x^k\|_{P^k-L_f I}^2\right].
\end{align}
Using the fact $\lambda^{k+1}=\lambda^k-\rho_k(Ax^{k+1}-b)$, we have
\begin{align}\label{eq:opt-ineq2-1}
\left\langle Ax^{k+1}-b,\lambda^k-\lambda\right\rangle=&~\frac{1}{\rho_k}\left\langle\lambda^k-\lambda^{k+1},\lambda^k-\lambda\right\rangle\cr
\overset{\eqref{uv-cross}}=&~\frac{1}{2\rho_k}\left[\|\lambda-\lambda^k\|^2-\|\lambda-\lambda^{k+1}\|^2+\|\lambda^k-\lambda^{k+1}\|^2\right].
\end{align}
In addition, we write $r^k=r^k-r^{k+1}+r^{k+1}=r^{k+1}-A(x^{k+1}-x^k)$ and have
\begin{align}\label{eq:opt-ineq2-2}
&~\left\langle Ax^{k+1}-b,-\beta_k r^k\right\rangle\cr
%=&~-\beta_k\left\langle Ax^{k+1}-b, r^k-r^{k+1}+r^{k+1}\right\rangle\cr
=&~-\beta_k\|r^{k+1}\|^2+\beta_k\left\langle A(x^{k+1}-x), A(x^{k+1}-x^k)\right\rangle\cr
\overset{\eqref{uv-cross}}=&~-\beta_k\|r^{k+1}\|^2+\frac{\beta_k}{2}\left[\|A(x^{k+1}-x)\|^2-\|A(x^k-x)\|^2+\|A(x^{k+1}-x^k)\|^2\right]
\end{align}
Substituting \eqref{eq:opt-ineq2-1} and \eqref{eq:opt-ineq2-2} into \eqref{eq:opt-ineq2} gives the inequality in \eqref{eq:1iter-ajadmm}.

\subsection{Proof of Theorem \ref{thm:ajadmm-g}}
%First note that
%\begin{equation}\label{eq:ajadmm-g-r}
%\|\lambda^k-\lambda^{k+1}\|^2=\rho_k^2\|r^{k+1}\|^2.
%\end{equation}
%Since $P^k\succeq \beta_k A^\top A+L_f I$, it holds that
%\begin{equation}\label{eq:ajadmm-g-xk}
%\|x^{k+1}-x^k\|_{P^k-\beta_k A^\top A - L_f I}^2 \ge 0.
%\end{equation}
First, we have %from the assumption in \eqref{eq:ajadmm-para-2}, we have
\begin{eqnarray}
& & \sum_{k=1}^t\frac{k+k_0+1}{2\rho_k}\left[\|\lambda-\lambda^k\|^2-\|\lambda-\lambda^{k+1}\|^2\right]\nonumber \\
&=& \frac{k_0+2}{2\rho_1}\|\lambda-\lambda^1\|^2-\frac{t+k_0+1}{2\rho_t}\|\lambda-\lambda^{t+1}\|^2+\sum_{k=2}^t
\left(\frac{k+k_0+1}{2\rho_k}-\frac{k+k_0}{2\rho_{k-1}}\right)\|\lambda-\lambda^k\|^2 \nonumber \\
&\overset{\eqref{eq:ajadmm-para-2}}\le& \frac{k_0+2}{2\rho_1}\|\lambda-\lambda^1\|^2 . \label{eq:ajadmm-g-lam}
\end{eqnarray}
In addition, 
\begin{eqnarray}
& &-\sum_{k=1}^t\frac{k+k_0+1}{2}\left(\|x^{k+1}-x\|_{P^k-\beta_k A^\top A+\mu I}^2-\|x^k-x\|_{P^k-\beta_k A^\top A}^2\right)\nonumber \\
&=& \frac{k_0+2}{2}\|x^1-x\|_{P^1-\beta_1 A^\top A}^2-\frac{t+k_0+1}{2}\|x^{t+1}-x\|_{P^t-\beta_t A^\top A+\mu I}^2\nonumber \\
& & +\frac{1}{2}\sum_{k=2}^t\left((k+k_0+1)\|x^k-x\|_{P^k-\beta_k A^\top A}^2-(k+k_0)\|x^k-x\|_{P^{k-1}-\beta_{k-1}A^\top A+\mu I}^2\right)\nonumber \\
&\overset{\eqref{eq:ajadmm-para-3}}\le & ~\frac{k_0+2}{2}\|x^1-x\|_{P^1-\beta_1 A^\top A}^2-\frac{t+k_0+1}{2}\|x^{t+1}-x\|_{P^t-\beta_t A^\top A+\mu I}^2. \label{eq:ajadmm-g-x}
\end{eqnarray}
Now multiplying $k+k_0+1$ to both sides of \eqref{eq:1iter-ajadmm} and adding it over $k$, we obtain \eqref{eq:ajadmm-rate-g} by using \eqref{eq:ajadmm-g-lam} and \eqref{eq:ajadmm-g-x}, and noting $\|\lambda^k-\lambda^{k+1}\|^2=\rho_k^2\|r^{k+1}\|^2$ and $\|x^{k+1}-x^k\|_{P^k-\beta_k A^\top A - L_f I}^2 \ge 0$.

\subsection{Proof of Theorem \ref{thm:ajadmm-spc}}
From the choice of $k_0$ and the condition $P-\beta A^\top A \preceq \frac{\mu}{2} I$, it is not difficult to verify %it follows that
%$$(2k+k_0+1)\frac{\mu}{2}I+ L_f I=(k+k_0)\mu I,\,\forall k.$$
%Since $P-\beta A^\top A \preceq \frac{\mu}{2} I,$ the above equation implies
%$$(2k+k_0+1)(P-\beta A^\top A)+L_f I\preceq (k+k_0)\mu I,$$
%which is equivalent to
$$(k+k_0+1)\left[kP-k\beta A^\top A+L_f I\right]\preceq (k+k_0)\left[(k-1)P-(k-1)\beta A^\top A+(L_f+\mu)I\right],\,\forall k\ge 1.$$
Hence, the condition in \eqref{eq:ajadmm-para-3} holds. In addition, it is easy to see that all conditions in \eqref{eq:ajadmm-para-1} and \eqref{eq:ajadmm-para-2} also hold. Therefore, we have \eqref{eq:ajadmm-rate-g}, which, by taking parameters in \eqref{eq:ajadmm-para-spc} and $x=x^*$, reduces to
\begin{eqnarray}
\sum_{k=1}^t(k+k_0+1)\Phi(x^{k+1},x^*,\lambda) +\sum_{k=1}^t\frac{k(k+k_0+1)}{2}\beta\|r^{k+1}\|^2& &\nonumber \\
 +\frac{t+k_0+1}{2}\|x^{t+1}-x^*\|^2_{t(P-\beta A^\top A)+(L_f+\mu) I}
&\le & \phi_1(x^*,\lambda), \label{eq:ajadmm-rate-spc-1}
\end{eqnarray}
where we have used the fact $\lambda^1=0$.

Letting $\lambda=\lambda^*$, we have from \eqref{eq:1stopt-cond} and \eqref{eq:ajadmm-rate-spc-1} that (by dropping nonnegative $\Phi(x^{k+1},x^*,\lambda^*)$'s):
%\begin{eqnarray*}%\label{eq:ajadmm-rate-spc-2}
$$\frac{t(t+k_0+1)}{2}\beta\|r^{t+1}\|^2+\frac{t+k_0+1}{2}\|x^{t+1}-x^*\|^2_{t(P-\beta A^\top A)+(L_f+\mu) I}
\le \phi_1(x^*,\lambda^*),$$
%\end{eqnarray*}
which indicates \eqref{eq:ajadmm-rate-spc}. In addition, from the convexity of $F$ and \eqref{eq:ajadmm-rate-spc-1}, we have that for any $\lambda$, it holds
$\frac{t(t+2k_0+3)}{2}\Phi(\bar{x}^{t+1},x^*,\lambda)\le\phi_1(x^*,\lambda),$
which together with Lemmas \ref{lem:xy-rate} and \ref{lem:equiv-rate} implies \eqref{eq:ajadmm-erg-rate-spc}.

\section{Technical proofs: Section \ref{sec:arpdc}}
In this section, we give the proofs of the lemmas and theorems in section \ref{sec:arpdc}.
\subsection{Proof of Lemma \ref{lem:1iter}}
From the update in \eqref{eq:fw-arpdc-x}, we have the optimality condition: %can be written more compactly as
%\begin{equation}\label{eq:upd-x2}
%x_{S_k}^{k+1}=
%\argmin\limits_{x_{S_k}}\langle \nabla_{S_k} f(x^k)-A_{S_k}^\top(\lambda^k-\beta_k r^k), x_{S_k}\rangle+g_{S_k}(x_{S_k})+\frac{\eta_k}{2}\|x_{S_k}-x_{S_k}^k\|^2.
%\end{equation}
%We have the optimality condition
%$$0= \nabla_{S_k} f(x^k)-A_{S_k}^\top(\lambda^k-\beta_k r^k)+\tilde{\nabla} g_{S_k}(x_{S_k}^{k+1})+\eta_k (x_{S_k}^{k+1}-x_{S_k}^k).$$
%Hence, for any $x$, it holds that
\begin{equation}\label{eq:opt-cond}
 \nabla_{S_k} f(x^k)-A_{S_k}^\top(\lambda^k-\beta_k r^k)+\tilde{\nabla} g_{S_k}(x_{S_k}^{k+1})+\eta_k (x_{S_k}^{k+1}-x_{S_k}^k) = 0.
\end{equation}
It follows from the update rule of $\lambda$ that
$$-\langle Ax^{k+1}-b, \lambda^k\rangle = - \langle Ax^{k+1}-b, \lambda^{k+1}\rangle - \rho_k\|r^{k+1}\|^2.$$
Plugging \eqref{eq:opt-ineq2-2} and the above equation into \eqref{eq:S-basic-ineq} with $S=S_k, \lambda^o=\lambda^k, \beta=\beta_k, x^o=x^k$, $x^+=x^{k+1}$, $W=\eta_k I$, and $x$ satisfying $Ax=b$, we have the desired result by taking expectation and recalling the definition of $\Delta$ in \eqref{eq:def-Delta} and $\Phi$ in \eqref{eq:nota-Phi}. 

\subsection{Proof of Theorem \ref{thm:naccl}}
Let $\beta_k=\beta, \rho_k=\rho$ and $\eta_k=\eta$ in \eqref{eq:sum-bd1}, and also note $\mu=0$ and $\eta \ge L_m+\beta\|A\|^2$. We have
\begin{align*}
& ~\EE\left[\Phi(x^{k+1},x,\lambda^{k+1})+(\beta-\rho)\|r^{k+1}\|^2\right]\\
\le & ~(1-\theta)\EE\left[\Phi(x^k,x,\lambda^k)+\beta\| r^k\|^2\right] - \frac{1}{2}\EE\left[\|x^{k+1}-x\|^2_{\eta I-\beta A^\top A}-\|x^{k}-x\|^2_{\eta I-\beta A^\top A}\right]. \nonumber
\end{align*}
Summing the above inequality over $k=1$ through $t$ and noting $\rho\le \theta \beta$ give
\begin{align}
& ~\EE\left[\Phi(x^{t+1},x,\lambda^{t+1})+(\beta-\rho)\|r^{t+1}\|^2\right] + \theta \sum_{k=1}^{t-1}\EE\Phi(x^{k+1},x,\lambda^{k+1})\label{eq:sum-bd1-prf} \\
\le & ~(1-\theta)\EE\left[\Phi(x^1,x,\lambda^1)+\beta\| r^1\|^2\right] + \frac{1}{2}\|x^1-x\|^2_{\eta I-\beta A^\top A}. \nonumber
\end{align}
By the update of $\lambda$, it follows that
\begin{align}\label{eq:sum-bd1-prf2}
\theta\sum_{k=1}^{t-1} \Phi(x^{k+1},x,\lambda^{k+1})=&~\theta\sum_{k=1}^{t-1} \left[\Phi(x^{k+1},x,\lambda)+\frac{1}{\rho} \langle\lambda^{k+1}-\lambda, \lambda^{k+1}-\lambda^k\rangle\right]\cr
=&~\theta\sum_{k=1}^{t-1} \Phi(x^{k+1},x,\lambda)+\frac{\theta}{2\rho}\sum_{k=1}^{t-1} \left[\|\lambda^{k+1}-\lambda\|^2-\|\lambda^{k}-\lambda\|^2+\|\lambda^{k+1}-\lambda^k\|^2\right]\cr
=&~\theta\sum_{k=1}^{t-1} \Phi(x^{k+1},x,\lambda)+\frac{\theta}{2\rho}\left[\|\lambda^{t}-\lambda\|^2-\lambda^1-\lambda\|^2+\sum_{k=1}^{t-1}\|\lambda^{k+1}-\lambda^k\|^2\right]
\end{align}
and
\begin{align}\label{eq:sum-bd1-prf3}
\Phi(x^{t+1},x,\lambda^{t+1}) = &~\Phi(x^{t+1},x,\lambda) - \langle \lambda^t - \lambda - \rho r^{t+1}, r^{t+1}\rangle\cr
=&~\Phi(x^{t+1},x,\lambda)- \langle \lambda^t - \lambda, r^{t+1}\rangle+\rho \|r^{t+1}\|^2.
\end{align}
Since $\rho \le \theta\beta$, by Young's inequality, it holds $$\beta\|r^{t+1}\|^2 - \langle \lambda^t - \lambda, r^{t+1}\rangle + \frac{\theta}{2\rho}\|\lambda^{t}-\lambda\|^2 \ge0.$$
Then plugging \eqref{eq:sum-bd1-prf2} and \eqref{eq:sum-bd1-prf3} into \eqref{eq:sum-bd1-prf}, we have
\begin{align}\label{eq:sum-bd1-prf4}
&~\EE\Phi(x^{t+1},x,\lambda) +\theta\sum_{k=1}^{t-1} \EE \Phi(x^{k+1},x,\lambda)\cr
\le &~(1-\theta)\EE\left[\Phi(x^1,x,\lambda^1)+\beta\| r^1\|^2\right] + \frac{1}{2}\|x^1-x\|^2_{\eta I-\beta A^\top A} + \frac{\theta}{2\rho}\EE\|\lambda^1-\lambda\|^2\cr
\le &~ \EE\phi_2(x,\lambda),
\end{align} 
where in the last inequality we have used $\lambda^1=0$, $\theta>0$ and $\| r^1\|^2=\|x^1-x\|^2_{\beta A^\top A}$.

Therefore, from the convexity of $F$, it follows that $\EE \Phi(\bar{x}^{t},x^*,\lambda) \le \frac{1}{1+\theta (t-1)}\EE\phi_2(x^*,\lambda),\,\forall \lambda$, and we obtain the desired result from Lemmas \ref{lem:xy-rate} and \ref{lem:equiv-rate}.

\subsection{Proof of Theorem \ref{thm:rate0}}
We first establish a few inequalities below. %They can be shown through simple arithmetic operations, and thus we omit the proofs.
%\begin{proposition}\label{prop1} If \eqref{eq:cond5} holds, then
%\begin{equation}\label{eq:sum-Axterm}
%\sum_{k=1}^t\frac{\beta_k(k+k_0+1)}{2}\EE\left(\|A(x^{k+1}-x)\|^2-\|A(x^k-x)\|^2\right)
%\le  \frac{\beta_t(t+k_0+1)}{2}\EE\|A(x^{t+1}-x)\|^2.
%\end{equation}
%\end{proposition}

\begin{proposition}\label{prop2} If \eqref{eq:cond5}, \eqref{eq:cond6} and \eqref{eq:cond7} hold, then
\begin{eqnarray}
%& &\sum_{k=1}^t\frac{\beta_k(k+k_0+1)}{2}\EE\|A(x^{k+1}-x^k)\|^2-\sum_{k=1}^t\frac{k+k_0+1}{2}(\eta_k-L_m)\EE\|x^{k+1}-x^k\|^2\cr
%& &-\sum_{k=1}^t\frac{k+k_0+1}{2}\EE\left(\eta_k\|x^{k+1}-x\|^2-\eta_k\|x^k-x\|^2\right)\cr
& &-\sum_{k=1}^t(k+k_0+1) \EE\left[\Delta_{\eta_k I-\beta_k A^\top A}(x^{k+1},x^k,x)-\frac{L_m}{2}\|x^{k+1}-x^k\|^2\right]\cr
& &-\frac{\mu(t+k_0+1)}{2}\EE\|x^{t+1}-x\|^2-\sum_{k=2}^t\frac{\mu\big(\theta(k+k_0+1)-1\big)}{2}\EE\|x^k-x\|^2\cr
%\overset{\eqref{eq:cond6}}\le &-\sum_{k=1}^t\frac{k+k_0+1}{2}\EE\left(\eta_k\big[\|x^{k+1}-x\|^2-\|x^k-x\|^2\big]+\theta\mu_f\|x^k-x\|^2\right)\cr
%=&\frac{(\eta_1-\theta\mu_f)(k_0+2)}{2}\EE\|x^1-x\|^2-\frac{(\mu_g+\eta_t)(t+k_0+1)}{2}\EE\|x^{t+1}-x\|^2\cr
%&-\sum_{k=2}^t\left(\frac{\eta_{k-1}(k+k_0)}{2}+\frac{\mu_g\big(\theta(k+k_0+1)-1\big)}{2}-\frac{(\eta_k-\theta\mu_f)(k+k_0+1)}{2}\right)\EE\|x^k-x\|^2\cr
&\le &\frac{\eta_1(k_0+2)}{2}\EE\|x^1-x\|^2-\frac{(t+k_0+1)}{2}\EE\|x^{t+1}-x\|^2_{(\mu+\eta_t)I-\beta_t A^\top A}. \label{eq:sum-xterm}
\end{eqnarray}
\end{proposition}

\begin{proof}
This inequality can be easily shown by noting that for any $1\le k\le t$, the weight matrix of $\frac{1}{2}\|x^{k+1}-x^k\|^2$ is $ \beta_k(k+k_0+1)A^\top A-(k+k_0+1)(\eta_k-L_m)I$, which is negative semidefinite, and for any $2\le k\le t$, the weight matrix of $\frac{1}{2}\|x^{k}-x\|^2$ is 
$$
\big[\beta_{k-1}(k+k_0)-\beta_k(k+k_0+1)\big]A^\top A+\left[ (k+k_0+1)\eta_k-(k+k_0)\eta_{k-1}-\mu\big(\theta(k+k_0+1)-1\big)\right]I,$$
which is also negative semidefinite.
\end{proof}

\begin{proposition}\label{prop3}
If \eqref{eq:cond1}, \eqref{eq:cond3} and \eqref{eq:cond4} hold, then
\begin{eqnarray}
& &-\frac{t+k_0+1}{\rho_t}\EE\Delta(\lambda^{t+1},\lambda^t,\lambda)-\sum_{k=2}^t\frac{\theta(k+k_0+1)-1}{\rho_{k-1}}\EE\Delta(\lambda^{k},\lambda^{k-1},\lambda)\cr
&\le& %-\frac{t+k_0+1}{2\rho_t}\EE\|\lambda^{t+1}-\lambda\|^2+
\frac{\theta(k_0+3)-1}{2\rho_1}\EE\|\lambda^1-\lambda\|^2. \label{eq:sum-lamterm}
\end{eqnarray}
\end{proposition}

\begin{proof}
On the left hand side of \eqref{eq:sum-lamterm}, the coefficient of each $\frac{1}{2}\|\lambda^{k+1}-\lambda^k\|^2$ is negative. For $2\le k\le t-1$, the coefficient of $\frac{1}{2}\|\lambda^k-\lambda\|^2$ is  $\frac{\theta(k+k_0+2)-1}{\rho_k}-\frac{\theta(k+k_0+1)-1}{\rho_{k-1}}$, which is nonpositive; the coefficient of $\frac{1}{2}\|\lambda^t-\lambda\|^2$ is
$\frac{t+k_0+1}{\rho_t}-\frac{\theta(t+k_0+1)-1}{\rho_{t-1}}$, which is nonpositive; the coefficient of $\frac{1}{2}\|\lambda^{t+1}-\lambda\|^2$ is also nonpositive. Hence, dropping these nonpositive terms, we have the desired result.
\end{proof}

Now we are ready to prove Theorem \ref{thm:rate0}.
\begin{proof} [of Theorem \ref{thm:rate0}]

Multiplying $k+k_0+1$ to both sides of \eqref{eq:sum-bd1}, summing it up from $k=1$ through $t$, {and moving the terms about $\Phi(x^k,x,\lambda^k)+\frac{\mu}{2}\|x^k-x\|^2$ and $\|r^k\|^2$ to the left hand side for $2\le k\le t$} give
\begin{eqnarray}
& &(t+k_0+1) \EE\left[\Phi(x^{t+1},x,\lambda^{t+1})+(\beta_t-\rho_t)\|r^{t+1}\|^2+ \frac{\mu}{2}\|x^{t+1}-x\|^2\right]\cr
& &+\sum_{k=2}^t\big(\theta(k+k_0+1)-1\big)\EE\left[\Phi(x^k,x,\lambda^k)+\frac{\mu}{2}\|x^k-x\|^2\right]\cr
& &+\sum_{k=2}^t\big((\beta_{k-1}-\rho_{k-1})(k+k_0)-(1-\theta)(k+k_0+1)\beta_k\big)\EE\| r^k\|^2\cr
&\le & (1-\theta)(k_0+2)\EE\left[\Phi(x^1,x,\lambda^1)+\beta_1\| r^1\|^2+\frac{\mu}{2}\|x^1-x\|^2\right] \label{eq:sum-bd2} \\
& &-\sum_{k=1}^t(k+k_0+1) \EE\left[\Delta_{\eta_k I-\beta_k A^\top A}(x^{k+1},x^k,x)-\frac{L_m}{2}\|x^{k+1}-x^k\|^2\right]. \nonumber
%\cr
%& &-\sum_{k=1}^t\frac{k+k_0+1}{2}\EE\left(\eta_k\|x^{k+1}-x\|^2-\eta_k\|x^k-x\|^2+ (\eta_k-L_m)\|x^{k+1}-x^k\|^2\right)
\end{eqnarray}
Hence, from \eqref{eq:cond2} and \eqref{eq:sum-xterm}, it follows that
\begin{equation}\label{eq:use-Ax-xterm}
\begin{aligned}
&~(t+k_0+1) \EE\Phi(x^{t+1},x,\lambda^{t+1})+\sum_{k=2}^t\big(\theta(k+k_0+1)-1\big)\EE\Phi(x^k,x,\lambda^k)\\ 
\le &~ (1-\theta)(k_0+2)\EE\left[\Phi(x^1,x,\lambda^1)+\beta_1\| r^1\|^2+\frac{\mu}{2}\|x^1-x\|^2\right]\\
&~+\frac{\eta_1(k_0+2)}{2}\EE\|x^1-x\|^2-\frac{t+k_0+1}{2}\EE\|x^{t+1}-x\|^2_{(\mu+\eta_t)I-\beta_t A^\top A}.
\end{aligned}
\end{equation}
In addition, from the update of $\lambda$ in \eqref{eq:fw-arpdc-lam}, we have
\begin{equation}\label{eq:lam-eq}
\langle\lambda^{k+1}-\lambda, Ax^{k+1}-b\rangle=-\frac{1}{\rho_k}\langle\lambda^{k+1}-\lambda,\lambda^{k+1}-\lambda^k\rangle=-\frac{1}{\rho_k}\Delta(\lambda^{k+1},\lambda^k,\lambda),
\end{equation}
and thus
\begin{eqnarray*}
& &(t+k_0+1)\EE\langle\lambda^{t+1}-\lambda, Ax^{t+1}-b\rangle+\sum_{k=2}^t\big(\theta(k+k_0+1)-1\big)\EE\langle\lambda^k-\lambda, Ax^k-b\rangle\cr
&=&-\frac{t+k_0+1}{\rho_t}\EE\Delta(\lambda^{t+1},\lambda^t,\lambda)-\sum_{k=2}^t\frac{\theta(k+k_0+1)-1}{\rho_{k-1}}\EE\Delta(\lambda^{k},\lambda^{k-1},\lambda)\cr
&\overset{\eqref{eq:sum-lamterm}}\le &\frac{\theta(k_0+3)-1}{2\rho_1}\EE\|\lambda^1-\lambda\|^2. %\label{eq:sum-lam}
%&=&-\frac{t+k_0+1}{\rho_t}\langle\lambda^{t+1}-\lambda,\lambda^{t+1}-\lambda^t\rangle
%-\sum_{k=2}^t\frac{\theta(k+k_0+1)-1}{\rho_{k-1}}\langle\lambda^k-\lambda,\lambda^k-\lambda^{k-1}\rangle\cr
%&=&-\frac{t+k_0+1}{2\rho_t}\left(\|\lambda^{t+1}-\lambda\|^2-\|\lambda^t-\lambda\|^2+\|\lambda^{t+1}-\lambda^t\|^2\right)\cr
%& &-\sum_{k=2}^t\frac{\theta(k+k_0+1)-1}{2\rho_{k-1}}\left(\|\lambda^k-\lambda\|^2-\|\lambda^{k-1}-\lambda\|^2+\|\lambda^k-\lambda^{k-1}\|^2\right)
\end{eqnarray*}
%Plugging \eqref{eq:sum-lamterm} into the above inequality yields
%$$(t+k_0+1)\langle\lambda^{t+1}-\lambda, Ax^{t+1}-b\rangle+\sum_{k=2}^t\big(\theta(k+k_0+1)-1\big)\langle\lambda^k-\lambda, Ax^k-b\rangle\le \frac{\theta(k_0+3)-1}{2\rho_1}\EE\|\lambda^1-\lambda\|^2.$$
Since $\Phi(x^k,x,\lambda)=\Phi(x^k,x,\lambda^k)+\langle \lambda^k-\lambda, Ax^k-b\rangle,$ we obtain the desired result by adding the above inequality to \eqref{eq:use-Ax-xterm}.% and the above inequality.
%Adding \eqref{eq:sum-lam} to \eqref{eq:sum-bd2} and rearranging terms yields
%\begin{eqnarray}
%& &(t+k_0+1) \EE\Phi(x^{t+1},x,\lambda)+\sum_{k=2}^t\big(\theta(k+k_0+1)-1\big)\EE\Phi(x^k,x,\lambda)
%\cr
%& &+(t+k_0+1)(\beta_t-\rho_t)\EE\|r^{t+1}\|^2+\sum_{k=2}^t\big((\beta_{k-1}-\rho_{k-1})(k+k_0)-(1-\theta)(k+k_0+1)\beta_k\big)\EE\| r^k\|^2\cr
%&\le & (1-\theta)(k_0+2)\EE\left[\Phi(x^1,x,\lambda^1)+\beta_1\| r^1\|^2+\frac{\mu_g}{2}\|x^1-x\|^2\right] \label{eq:sum-bd3} \\
%& &+\sum_{k=1}^t\frac{\beta_k(k+k_0+1)}{2}\EE\left(\|A(x^{k+1}-x)\|^2-\|A(x^k-x)\|^2+\|A(x^{k+1}-x^k)\|^2\right)\cr
%& &-\sum_{k=1}^t\frac{k+k_0+1}{2}\EE\left(\eta_k\|x^{k+1}-x\|^2-(\eta_k-\theta\mu_f)\|x^k-x\|^2+ (\eta_k-L_m)\|x^{k+1}-x^k\|^2\right)\cr
%& &-\frac{\mu_g(t+k_0+1)}{2}\EE\|x^{t+1}-x\|^2-\sum_{k=2}^t\frac{\mu_g\big(\theta(k+k_0+1)-1\big)}{2}\EE\|x^k-x\|^2\cr
%& &-\frac{t+k_0+1}{2\rho_t}\EE\left(\|\lambda^{t+1}-\lambda\|^2-\|\lambda^t-\lambda\|^2+\|\lambda^{t+1}-\lambda^t\|^2\right)\cr
%& & -\sum_{k=2}^t\frac{\theta(k+k_0+1)-1}{2\rho_{k-1}}\EE\left(\|\lambda^k-\lambda\|^2-\|\lambda^{k-1}-\lambda\|^2+\|\lambda^k-\lambda^{k-1}\|^2\right). \nonumber
%\end{eqnarray}
%Therefore, we obtain \eqref{eq:sum-bd4} by substituting \eqref{eq:sum-Axterm} through \eqref{eq:sum-lamterm} into \eqref{eq:sum-bd3} and also noting the summation in the second line of \eqref{eq:sum-bd3} is nonnegative from the condition in \eqref{eq:cond2}.
\end{proof}

\subsection{Proof of Proposition \ref{prop:param-cond}}
Note that \eqref{eq:paras-k0} implies $k_0\ge\frac{4}{\theta}$, and thus \eqref{eq:cond1} must hold. Also, it is easy to see that \eqref{eq:cond4} holds with equality from the second equation of \eqref{eq:paras-rho}. Since $I\succeq \frac{A^\top A}{\|A\|_2^2}$, we can easily have \eqref{eq:cond6} by plugging in $\beta_k$ and $\eta_k$ defined in \eqref{eq:paras-beta} and \eqref{eq:paras-eta} respectively.

To verify \eqref{eq:cond3}, we plug in $\rho_k$ defined in the first equation of \eqref{eq:paras-rho}, and it is equivalent to requiring that for any $2\le k\le t-1$ 
$$\frac{\theta(k+k_0+1)-1}{\theta (k-1)+2+\theta}\ge\frac{\theta(k+k_0+2)-1}{\theta k+2+\theta}\Longleftrightarrow 1+\frac{\theta(k_0+1)-3}{\theta k+2}\ge1+\frac{\theta(k_0+1)-3}{\theta k+2+\theta}.$$
The inequality on the right hand side obviously holds, and thus we have \eqref{eq:cond3}.

Plugging in the formula of $\beta_k$, \eqref{eq:cond5} is equivalent to
$$(\theta k+2+\theta)(k+k_0+1)\ge (\theta k+2)(k+k_0),$$
which holds trivially, and thus \eqref{eq:cond5} follows.

With the given $\beta_k$ and $\rho_k$, \eqref{eq:cond2} becomes $\frac{6}{6-5\theta}(\theta k+2)(k+k_0)\ge (k+k_0+1)(\theta k+2+\theta),\,\forall 2\le k\le t,$ which is equivalent to $\frac{6}{6-5\theta}\ge \frac{(k_0+3)(3\theta+2)}{(k_0+2)(2\theta+2)}$. Note that $\frac{k_0+3}{k_0+2}$ is decreasing with respect to $k_0\ge0$ and also  $\frac{6}{6-5\theta}\ge \frac{(\frac{3}{\theta}+3)(3\theta+2)}{(\frac{3}{\theta}+2)(2\theta+2)}$. Hence, \eqref{eq:cond2} is satisfied from the fact $k_0\ge \frac{4}{\theta}$.
%\begin{eqnarray}
%&&\frac{6-6\theta}{6-5\theta}\beta_{k-1}(k+k_0)\ge (1-\theta)(k+k_0+1)\beta_k,\,\forall k\ge 2\nonumber\\
%&\Longleftrightarrow & \frac{6}{6-5\theta}(\theta k+2)(k+k_0)\ge (k+k_0+1)(\theta k+2+\theta),\,\forall k\ge 2 \nonumber\\
%&\Longleftrightarrow & \frac{6}{6-5\theta}\ge\frac{(k+k_0+1)(\theta k+2+\theta)}{(k+k_0)(\theta k+2)},\,\forall k\ge 2\nonumber\\
%&\Longleftrightarrow & \frac{6}{6-5\theta}\ge \frac{(k_0+3)(3\theta+2)}{(k_0+2)(2\theta+2)}\nonumber\\
%&\Longleftarrow &\frac{6}{6-5\theta}\ge \frac{(\frac{3}{\theta}+3)(3\theta+2)}{(\frac{3}{\theta}+2)(2\theta+2)}\\
%&\Longleftrightarrow & 12(3+2\theta)\ge 3(6-5\theta)(3\theta+2) \label{eq:app-suff}\nonumber\\
%& \Longleftrightarrow & 36+24\theta\ge 36+24\theta-45\theta^2,\nonumber
%\end{eqnarray}
%where the sufficient condition in \eqref{eq:app-suff} uses the fact $k_0\ge \frac{3}{\theta}$, and the last inequality apparently holds. Hence, \eqref{eq:cond2} is satisfied.

Finally, we show \eqref{eq:cond7}. Plugging in $\eta_k$, we have that \eqref{eq:cond7} becomes
$$
(k+k_0)\left(\frac{\mu}{2}\left(\theta k+2\right)+L_m\right)+\mu\big(\theta(k+k_0+1)-1\big)\ge  (k+k_0+1)\left(\frac{\mu}{2}\left(\theta k+2+\theta\right)+L_m\right),\,\forall k\ge 2,
$$
which is equivalent to $k_0+1\ge \frac{4}{\theta}+\frac{2L_m}{\theta\mu}$. Hence, for $k_0$ given in \eqref{eq:paras-k0}, \eqref{eq:cond7} must hold. Therefore, we have verified all conditions in \eqref{eq:para-conds}.
%\begin{eqnarray*}
%&&(k+k_0)\left(\frac{\mu}{2}\left(\theta k+2\right)+L_m\right)+\mu\big(\theta(k+k_0+1)-1\big)\\
%&&\ge  (k+k_0+1)\left(\frac{\mu}{2}\left(\theta k+2+\theta\right)+L_m\right),\,\forall k\ge 2\\
%&\Longleftrightarrow & \frac{\mu\theta}{2}(k+k_0+1)\ge \frac{\mu}{2}(\theta k+2)+L_m+\mu,\,\forall k\ge 2\\
%&\Longleftrightarrow &\frac{\mu\theta}{2}(k_0+1)\ge \mu+\mu+L_m\\
%&\Longleftrightarrow & k_0+1\ge \frac{4}{\theta}+\frac{2L_m}{\theta\mu},
%\end{eqnarray*}
%where the last inequality is implied by \eqref{eq:paras-k0}. Therefore, \eqref{eq:cond7} holds, and we have verified all conditions in \eqref{eq:para-conds}.

\subsection{Proof of Theorem \ref{thm:rate1}}
From Proposition \ref{prop:param-cond}, we have the inequality in \eqref{eq:sum-bd4} that, as $\lambda^1=0$, reduces to
\begin{eqnarray}
& &(t+k_0+1) \EE\Phi(x^{t+1},x,\lambda)+\sum_{k=2}^t\big(\theta(k+k_0+1)-1\big)\EE\Phi(x^k,x,\lambda)\nonumber\\
&\le & \phi_3(x,\lambda)-\frac{t+k_0+1}{2}\EE\|x^{t+1}-x\|_{(\mu+\eta_t) I-\beta_t A^\top A}^2.\label{eq:sum-bd4-1}
\end{eqnarray}

For $\rho\ge1$, we have
\begin{equation}\label{eq:rate1-mu-ineq}(\mu+\eta_t) I-\beta_t A^\top A\succeq \left(\frac{(\rho-1)\mu}{2\rho}(\theta t + \theta + 2) + \mu+L_m\right)I.
\end{equation}
Letting $x=x^*$ and using the convexity of $F$, we have from \eqref{eq:sum-bd4-1} and the above inequality that
\begin{align}\label{eq:sum-bd6}
\EE\left[F(\bar{x}^{t+1})-F(x^*)-\big\langle\lambda,A\bar{x}^{t+1}-b\big\rangle \right]\le \frac{1}{T}\EE\phi_3(x^*,\lambda),\,\forall \lambda,
\end{align}
which together with Lemmas \ref{lem:xy-rate} and \ref{lem:equiv-rate} with $\gamma=\max(2\|\lambda^*\|, 1+\|\lambda^*\|)$ indicates \eqref{eq:rate-obj-feas}.

In addition, note
$$\Phi(x^{t+1},x^*,\lambda^*)\ge\frac{\mu}{2}\|x^{t+1}-x^*\|^2.$$
Hence, letting $(x,\lambda)=(x^*,\lambda^*)$ in \eqref{eq:sum-bd4-1} and using \eqref{eq:1stopt-cond}, we have from \eqref{eq:rate1-mu-ineq} that
\begin{align}\label{eq:rate-pt}
\frac{t+k_0+1}{2}\left(\frac{(\rho-1)\mu}{2\rho}(\theta t + \theta + 2)+2\mu+L_m\right)\EE\|x^{t+1}-x^*\|^2
\le \phi_3(x^*,\lambda^*),
\end{align}
and the proof is completed.

\section{Technical proofs: Section \ref{sec:linear}}
In this section, we provide the proofs of the lemmas and theorems in section \ref{sec:linear}.

\subsection{Proof of Lemma \ref{lem:linear-1step}}
Note $r^{k+1}-r^k=A(x^{k+1}-x^k)+B(y^{k+1}-y^k)$. Hence by \eqref{uv-cross}, we have 
\begin{equation}
\begin{aligned}\label{eq:lin1-r-term}
\left\langle A(x^{k+1}-x),-\beta r^k\right\rangle
=&~-\beta\left\langle A(x^{k+1}-x), r^{k+1}\right\rangle+\beta\left\langle A(x^{k+1}-x), B(y^{k+1}-y^k)\right\rangle\\
&~+\frac{\beta}{2}\left[\|A(x^{k+1}-x)\|^2-\|A(x^k-x)\|^2+\|A(x^{k+1}-x^k)\|^2\right].
\end{aligned}
\end{equation}
In addition, $\langle A(x^{k+1}-x), \lambda^k\rangle = \langle A(x^{k+1}-x), \lambda^{k+1}+\rho r^{k+1}\rangle$. Plugging this equation and \eqref{eq:lin1-r-term} into \eqref{eq:S-basic-ineq} with $x^o=x^k, \lambda^o=\lambda^k, x^+=x^{k+1}, W=\eta_x I$ and taking expectation yield
\begin{align}\label{eq:lin-convg-ineq1}
 &~ \EE\left[F(x^{k+1})-F(x)+\frac{\mu}{2}\|x^{k+1}-x\|^2-\big\langle A(x^{k+1}-x), \lambda^{k+1}\big\rangle+(\beta-\rho)\big\langle A(x^{k+1}-x), r^{k+1}\big\rangle\right] \nonumber\\
&~+\frac{1}{2}\EE\left[\|x^{k+1}-x\|_P^2-\|x^k-x\|_P^2+\|x^{k+1}-x^k\|^2_{P-L_mI}\right]\nonumber\\
\le &~  (1- \theta)\EE\left[F(x^k)-F(x)+\frac{\mu}{2}\|x^k-x\|^2-\big\langle A(x^k-x), \lambda^k-\beta r^k\big\rangle\right]\\
&~+\beta\EE\left\langle A(x^{k+1}-x), B(y^{k+1}-y^k)\right\rangle,\nonumber
\end{align}
where $P=\eta_x I-\beta A^\top A$.

%Similar to \eqref{eq:term1}, we have
%\begin{eqnarray}
%& &\EE\left\langle x^{k+1}_{S_k}-x_{S_k}, \nabla_{S_k} f(x^k)-A_{S_k}^\top(\lambda^k- \beta_k  r^k)\right\rangle\cr
%&\ge&\EE\left[ f(x^{k+1})-f(x)-\big\langle A(x^{k+1}-x),\lambda^{k+1}\big\rangle+(\beta-\rho)\big\langle A(x^{k+1}-x),r^{k+1}\big\rangle\right.\cr
%& & \left.-\beta\big\langle A(x^{k+1}-x),A(x^{k+1}-x^k)+B(y^{k+1}-y^k)\big\rangle \right]- \frac{L_m}{2}\EE\|x^{k+1}-x^k\|^2\cr
%& &-(1-\theta)\EE\left[ f(x^k)-f(x)-\big\langle A(x^k-x),\lambda^k\big\rangle+\beta\big\langle A(x^k-x), r^k\big\rangle\right]+\frac{\theta\mu_f}{2}\EE\|x^k-x\|^2. \label{gradf-term2}
%\end{eqnarray}

From \eqref{class-y-update}, the optimality condition for $ \tilde{y}^{k+1}$ is
\begin{equation}\label{eq:opt-y}
\nabla h( \tilde{y}^{k+1})- B^\top\vlam^k+ \beta  B^\top\vr^{k+\frac{1}{2}} +\eta_y (\tilde{y}^{k+1}-y^k)=\vzero.
\end{equation}
Since
$\Prob(y^{k+1}=\tilde{y}^{k+1})=\theta,\, \Prob(y^{k+1}=y^k)=1-\theta,$
we have
\begin{eqnarray*}
& &\EE\left\langle y^{k+1}- y, \nabla h( y^{k+1})- B^\top\vlam^{k}
+ \beta  B^\top r^{k+\frac{1}{2}}+\eta_y(y^{k+1}-y^k)\right\rangle\cr
&=&(1-\theta)\EE\left\langle y^k- y, \nabla h( y^k)- B^\top\vlam^k
+ \beta  B^\top r^{k+\frac{1}{2}}\right\rangle,
\end{eqnarray*}
or equivalently,
\begin{eqnarray}
& &\EE\left\langle y^{k+1}- y, \nabla h( y^{k+1})- B^\top\vlam^{k+1}
+( \beta -\rho) B^\top r^{k+1}- \beta B^\top B(y^{k+1}-y^k)+\eta_y(y^{k+1}-y^k)\right\rangle\cr
&=&(1-\theta)\EE\left\langle y^k- y, \nabla h( y^k)- B^\top\vlam^k
+ \beta  B^\top r^k\right\rangle+\beta(1-\theta) \EE\left\langle B(y^k- y),A(x^{k+1}-x^k)\right\rangle. \label{opt-1Y}
\end{eqnarray}
Recall $Q=\eta_y I -\beta B^\top B$. We have
$$\left\langle y^{k+1}- y, - \beta B^\top B(y^{k+1}-y^k)+\eta_y(y^{k+1}-y^k)\right\rangle = \frac{1}{2}\left[\|y^{k+1}- y\|_Q^2-\|y^k- y\|_Q^2+\|y^{k+1}- y^k\|_Q^2\right].$$
Therefore adding \eqref{opt-1Y} to \eqref{eq:lin-convg-ineq1}, noting $Ax+By=b$, and plugging \eqref{eq:lam-eq} with $\rho_k=\rho$, we have the desired result.

\subsection{Proof of Theorem \ref{thm-pre}}
Before proving Theorem \ref{thm-pre}, we establish a few inequalities. First, using Young's inequality, we have the following results.
\begin{lemma}
%For any $\alpha>0$, it holds
%\begin{equation}\label{ineq-grad-F}
%-\frac{\alpha\mu}{4}\|\vx^{k+1}-\vx^*\|^2-\frac{L_m^2}{\alpha\mu}\|\vx^{k+1}-\vx^k\|^2\le -(\vx^{k+1}-\vx^*)^\top(\nabla f(\vx^{k+1})-\nabla f(\vx^k)),
%\end{equation}
For any $\tau_1,\tau_2>0$, it holds that
\begin{eqnarray}
& &\langle A(x^{k+1}- x^*),  B( y^{k+1}- y^k)\rangle\le\frac{1}{2\tau_1}\|A(x^{k+1}-x^*)\|^2+\frac{\tau_1}{2}\|B(y^{k+1}-y^k)\|^2,\label{cross-xy-term1}\\
& &\langle B(y^k- y^*),A(x^{k+1}-x^k)\rangle\le\frac{1}{2\tau_2}\|B(y^k- y^*)\|^2+\frac{\tau_2}{2}\|A(x^{k+1}-x^k)\|^2.\label{cross-xy-term2}
\end{eqnarray}
\end{lemma}
In addition, we are able to bound the $\lambda$-term by $y$-term and the residual $r$. The proofs are given in Appendix \ref{sec:pf-lem-c1} and \ref{sec:pf-lem-c2}.
\begin{lemma}\label{lem:bd-lam-by-y}
For any $\delta>0$, we have
\begin{eqnarray}\label{bd-lam-term-n}
& &\EE\|B^\top(\lambda^{k+1}-\lambda^*)\|^2-(1-\theta)(1+\delta)\EE\|B^\top(\lambda^k-\lambda^*)\|^2\cr
&\le & 4\EE\big[L_h^2\|y^{k+1}-y^*\|^2+\|Q(y^{k+1}-y^k)\|^2\big]+2(\beta-\rho)^2\EE\|B^\top r^{k+1}\|^2\\
& &+2\rho^2(1-\theta)(1+\frac{1}{\delta})\EE\big[\|B^\top r^{k+1}\|^2+\|B^\top B(y^{k+1}-y^k)\|^2\big].\nonumber
\end{eqnarray}
\end{lemma}

\begin{lemma}\label{lem:ineq-lam-term}
Assume \eqref{eq:cond-kappa-del}. Then
\begin{eqnarray}\label{eq:ineq-lam-term}
& &\frac{\sigma_{\min}(BB^\top)}{2}\big[\|\lambda^{k+1}-\lambda^*\|^2-(1-\theta)\|\lambda^{k}-\lambda^*\|^2+\frac{1}{\theta}\|\lambda^{k+1}-\lambda^k\|^2\big]\cr
&\le &\|B^\top(\lambda^{k+1}-\lambda^*)\|^2-(1-\theta)(1+\delta)\|B^\top(\lambda^k-\lambda^*)\|^2+\kappa\|B^\top(\lambda^{k+1}-\lambda^k)\|^2,
\end{eqnarray}
where $\sigma_{\min}(BB^\top)$ denotes the smallest singular value of $BB^\top$.
\end{lemma}

\begin{lemma}\label{lem:three-lemma-together}
Let $c,\delta,\tau_1,\tau_2$ and $\kappa$ be constants satisfying the conditions in Theorem \ref{thm-pre}. Then
\begin{align}\label{eq:three-lemma-together}
&~\beta \EE\big\langle A(x^{k+1}- x^*),  B( y^{k+1}- y^k)\big\rangle+\beta(1 -\theta)\EE\big\langle B(y^k- y^*),A(x^{k+1}-x^k)\big\rangle\cr
&~+\frac{c}{2}\sigma_{\min}(BB^\top)\EE\big[\|\lambda^{k+1}-\lambda^*\|^2-(1-\theta)\|\lambda^{k}-\lambda^*\|^2+\frac{1}{\theta}\|\lambda^{k+1}-\lambda^k\|^2\big]\cr
\le &~ \frac{1}{2}\EE\|x^{k+1}- x^k\|_{P-L_mI}^2+\frac{\beta}{2\tau_1}\EE\|A(x^{k+1}-x^*)\|^2\\
&~+\frac{1}{2}\EE\|y^{k+1}- y^k\|_Q^2+\frac{\beta(1-\theta)}{2\tau_2}\EE\|B(y^k-y^*)\|^2+4cL_h^2\EE\|y^{k+1}-y^*\|^2\cr
&~+\left[c\rho^2\left(\kappa+2(1-\theta)\big(1+\frac{1}{\delta}\big)\right)+2c(\beta-\rho)^2\right]\EE\|B^\top r^{k+1}\|^2.\nonumber
\end{align}
\end{lemma}

Now we are ready to show Theorem \ref{thm-pre}.
\begin{proof} [of Theorem \ref{thm-pre}]

Letting $(x,y,\lambda)=(x^*,y^*,\lambda^*)$ in \eqref{lin-ineq1-1Y}, plugging \eqref{kkt} into it, and noting $Ax^*+By^*=b$, we have
\begin{eqnarray}
& &\EE\Psi(z^{k+1},z^*)+( \beta -\rho)\EE \|r^{k+1}\|^2+\EE \left[\Delta_P(x^{k+1},x^k, x^*)-\frac{L_m}{2}\|x^{k+1}- x^k\|^2\right]\nonumber\\
& &+\EE\Delta_Q(y^{k+1},y^k, y^*)+\frac{\mu}{2}\EE\|x^{k+1}-x^*\|^2+\frac{1}{\rho}\EE\Delta(\lambda^{k+1},\lambda^{k},\lambda^*)\nonumber\\
%&-\frac{L_m}{2}\EE\|x^{k+1}-x^k\|^2+\frac{\theta\mu_f}{2}\EE\|x^k-x\|^2+\frac{\mu_g}{2}\EE\|x^{k+1}-x\|^2\nonumber\\
&\le &
(1-\theta)\EE\Psi(z^k,z^*)+\beta(1 -\theta)\EE \|r^k\|^2+\frac{1-\theta}{\rho}\EE\Delta(\lambda^{k},\lambda^{k-1},\lambda^*)+\frac{\mu(1-\theta)}{2}\EE\|x^k-x^*\|^2  \nonumber\\
& &+ \beta \EE\big\langle A(x^{k+1}- x^*),  B( y^{k+1}- y^k)\big\rangle+\beta(1 -\theta)\EE\big\langle B(y^k- y^*),A(x^{k+1}-x^k)\big\rangle, \label{lin-ineq2-1Y}
\end{eqnarray}
where $\Psi$ is defined in \eqref{eq:def-Psi}. 
Note 
\begin{align*}
&~\frac{1}{\rho}\Delta(\lambda^{k+1},\lambda^{k},\lambda^*)\\
=&~\frac{1}{2\rho}\big[\|\lambda^{k+1}-\lambda^*\|^2-(1-\theta)\|\lambda^{k}-\lambda^*\|^2
+\frac{1}{\theta}\|\lambda^{k+1}-\lambda^k\|^2\big]-\frac{\rho}{2}(\frac{1}{\theta}-1)\|r^{k+1}\|^2 - \frac{\theta}{2\rho}\|\lambda^{k}-\lambda^*\|^2,
\end{align*}
and
\begin{align*}
&~\frac{1-\theta}{\rho}\Delta(\lambda^{k},\lambda^{k-1},\lambda^*)\\
=&~\frac{1}{2\rho}\big[\|\lambda^{k}-\lambda^*\|^2-(1-\theta)\|\lambda^{k-1}-\lambda^*\|^2
+\frac{1}{\theta}\|\lambda^{k}-\lambda^{k-1}\|^2\big]-\frac{\rho}{2}(\frac{1}{\theta}-(1-\theta))\|r^k\|^2- \frac{\theta}{2\rho}\|\lambda^{k}-\lambda^*\|^2.
\end{align*}

Adding \eqref{eq:three-lemma-together} to \eqref{lin-ineq2-1Y} and plugging the above two equations yield
\begin{eqnarray}
& &\EE\Psi(z^{k+1},z^*)+( \beta -\rho)\EE \|r^{k+1}\|^2+\EE \left[\Delta_P(x^{k+1},x^k, x^*)-\frac{L_m}{2}\|x^{k+1}- x^k\|^2\right]\nonumber\\
& &+\EE\Delta_Q(y^{k+1},y^k, y^*)+\frac{\mu}{2}\EE\|x^{k+1}-x^*\|^2-\frac{\rho}{2}(\frac{1}{\theta}-1)\EE\|r^{k+1}\|^2 - \frac{\theta}{2\rho}\EE\|\lambda^{k}-\lambda^*\|^2\nonumber\\
& &+\left(\frac{1}{2\rho}+\frac{c}{2}\sigma_{\min}(BB^\top)\right)\EE\big[\|\lambda^{k+1}-\lambda^*\|^2-(1-\theta)\|\lambda^{k}-\lambda^*\|^2+\frac{1}{\theta}\|\lambda^{k+1}-\lambda^k\|^2\big]\nonumber\\
&\le &
(1-\theta)\EE\Psi(z^k,z^*)+\beta(1 -\theta)\EE \|r^k\|^2 -\frac{\rho}{2}(\frac{1}{\theta}-(1-\theta))\EE\|r^k\|^2- \frac{\theta}{2\rho}\EE\|\lambda^{k}-\lambda^*\|^2 \nonumber\\
&&+\frac{1}{2\rho}\EE\big[\|\lambda^{k}-\lambda^*\|^2-(1-\theta)\|\lambda^{k-1}-\lambda^*\|^2
+\frac{1}{\theta}\|\lambda^{k}-\lambda^{k-1}\|^2\big]\nonumber\\
& &+\frac{\mu(1-\theta)}{2}\EE\|x^k-x^*\|^2+\frac{1}{2}\EE\|x^{k+1}- x^k\|_{P-L_mI}^2+\frac{\beta}{2\tau_1}\EE\|A(x^{k+1}-x^*)\|^2\nonumber\\
& &+\frac{1}{2}\EE\|y^{k+1}- y^k\|_Q^2+\frac{\beta(1-\theta)}{2\tau_2}\EE\|B(y^k-y^*)\|^2+4cL_h^2\EE\|y^{k+1}-y^*\|^2\cr
& &+\left[c\rho^2\left(\kappa+2(1-\theta)\big(1+\frac{1}{\delta}\big)\right)+2c(\beta-\rho)^2\right]\EE\|B^\top r^{k+1}\|^2. \nonumber
\end{eqnarray}

Using the definition in \eqref{eq:def-Delta} to expand $\Delta_P(x^{k+1},x^k, x^*)$ and $\Delta_Q(y^{k+1},y^k, y^*)$ in the above inequality, and then rearranging terms, we have 
\begin{eqnarray}
& &\EE\Psi(z^{k+1},z^*)+\left(( \beta -\rho)-\frac{\rho}{2}(\frac{1}{\theta}-1)\right)\EE \|r^{k+1}\|^2\cr
& & -\left[c\rho^2\left(\kappa+2(1-\theta)\big(1+\frac{1}{\delta}\big)\right)+2c(\beta-\rho)^2\right]\EE\|B^\top r^{k+1}\|^2\cr
& &+\EE \left[\frac{1}{2}\|x^{k+1}-x^*\|_P^2+\frac{\mu}{2}\|x^{k+1}-x^*\|^2-\frac{\beta}{2\tau_1}\|A(x^{k+1}-x^*)\|^2\right]\nonumber\\
& &+\EE\left[\frac{1}{2}\|y^{k+1}-y^*\|_Q^2-4cL_h^2\|y^{k+1}-y^*\|^2\right] \nonumber\\
& &+\left(\frac{1}{2\rho}+\frac{c}{2}\sigma_{\min}(BB^\top)\right)\EE\big[\|\lambda^{k+1}-\lambda^*\|^2-(1-\theta)\|\lambda^{k}-\lambda^*\|^2+\frac{1}{\theta}\|\lambda^{k+1}-\lambda^k\|^2\big]\nonumber\\
&\le &
(1-\theta)\EE\Psi(z^k,z^*)+\beta(1 -\theta)\EE \|r^k\|^2 -\frac{\rho}{2}(\frac{1}{\theta}-(1-\theta))\EE\|r^k\|^2 + \frac{1}{2}\EE\|x^{k}-x^*\|_P^2 \nonumber\\
& &+\frac{\mu(1-\theta)}{2}\EE\|x^k-x^*\|^2+\frac{1}{2}\EE\|y^{k}- y^*\|_Q^2+\frac{\beta(1-\theta)}{2\tau_2}\EE\|B(y^k-y^*)\|^2\nonumber\\
&&+\frac{1}{2\rho}\EE\big[\|\lambda^{k}-\lambda^*\|^2-(1-\theta)\|\lambda^{k-1}-\lambda^*\|^2
+\frac{1}{\theta}\|\lambda^{k}-\lambda^{k-1}\|^2\big]. \label{lin-ineq2-1Y-21}
\end{eqnarray}
Since $\rho = \theta\beta$, it holds 
$$( \beta -\rho)-\frac{\rho}{2}(\frac{1}{\theta}-1) = \frac{\beta-\rho}{2}, \quad \beta(1 -\theta)-\frac{\rho}{2}(\frac{1}{\theta}-(1-\theta))\le \frac{\beta(1-\theta)}{2},$$
and thus the inequality \eqref{lin-ineq2-1Y-21} implies
\begin{eqnarray}
& &\EE\Psi(z^{k+1},z^*)+\frac{ \beta -\rho}{2}\EE \|r^{k+1}\|^2 -\left[c\rho^2\left(\kappa+2(1-\theta)\big(1+\frac{1}{\delta}\big)\right)+2c(\beta-\rho)^2\right]\EE\|B^\top r^{k+1}\|^2\cr
& &+\EE \left[\frac{1}{2}\|x^{k+1}-x^*\|_P^2+\frac{\mu}{2}\|x^{k+1}-x^*\|^2-\frac{\beta}{2\tau_1}\|A(x^{k+1}-x^*)\|^2\right]\nonumber\\
& &+\EE\left[\frac{1}{2}\|y^{k+1}-y^*\|_Q^2-4cL_h^2\|y^{k+1}-y^*\|^2\right] \nonumber\\
& &+\left(\frac{1}{2\rho}+\frac{c}{2}\sigma_{\min}(BB^\top)\right)\EE\big[\|\lambda^{k+1}-\lambda^*\|^2-(1-\theta)\|\lambda^{k}-\lambda^*\|^2+\frac{1}{\theta}\|\lambda^{k+1}-\lambda^k\|^2\big]\nonumber\\
&\le & \psi(z^k,z^*; P,Q,\beta,\rho,c,\tau_2),\label{lin-ineq2-1Y-22}
\end{eqnarray}
where $\psi$ is defined in \eqref{eq:def-psi}.

From \eqref{scvx-Fh}, it follows that
\begin{equation}\label{eq:ineq-Psi-term} (1-\alpha)\Psi(z^{k+1},z^*)+\frac{\alpha\mu}{2}\|x^{k+1}-x^*\|^2+\alpha\nu\|y^{k+1}-y^*\|^2\le \Psi(z^{k+1},z^*) .
\end{equation}
In addition, note that
\begin{eqnarray*}
\|r^{k+1}\|^2 &=&\|Ax^{k+1}+By^{k+1}-(Ax^*+By^*)\|^2\cr
&\le & 2\|A\|_2^2\|x^{k+1}-x^*\|^2+2\|B\|_2^2\|y^{k+1}-y^*\|^2\cr
&\le & \gamma\left(\frac{\alpha\mu}{4}\|x^{k+1}-x^*\|^2+\frac{\alpha\nu}{4}\|y^{k+1}-y^*\|^2\right),
\end{eqnarray*}
and thus
\begin{equation}\label{eq:ineq-r-term}
\frac{1}{\gamma}\|r^{k+1}\|^2\le \frac{\alpha\mu}{4}\|x^{k+1}-x^*\|^2+\frac{\alpha\nu}{4}\|y^{k+1}-y^*\|^2.
\end{equation}
Adding \eqref{eq:ineq-Psi-term} and \eqref{eq:ineq-r-term} to \eqref{lin-ineq2-1Y-22} gives the desired result. %Then we obtain the desired result by noting $\rho=\theta\beta$, 
\end{proof}

\subsection{Proof of Theorem \ref{thm-linear}}
From $0<\alpha<\theta$, the full row-rankness of $B$, and the conditions in \eqref{eq:paras-lin2}, it is easy to see that $\eta>1$. Next we find lower bounds of the terms on the left hand of \eqref{lin-ineq5-1Y}. Since $\eta\le \frac{1-\alpha}{1-\theta}$, we have 
\begin{equation}\label{eq:lin-convg-Phi-term} \eta(1-\theta)\Psi(z^{k+1},z^*)\le (1-\alpha)\Psi(z^{k+1},z^*).
\end{equation} 
Note $\|A\|_2\le 1$ and $$\left(\frac{\alpha\mu}{2}+\mu-\frac{\beta}{\tau_1}\right)I\succeq \frac{\frac{\alpha\mu}{2}+\theta\mu-\frac{\beta}{\tau_1}}{\eta_x+\mu(1-\theta)}(\eta_x I-\beta A^\top A)+ \frac{\frac{\alpha\mu}{2}+\theta\mu-\frac{\beta}{\tau_1}}{\eta_x+\mu(1-\theta)}\mu(1-\theta)I +\mu(1-\theta)I.$$  Hence, from $\eta\le 1+\frac{\frac{\alpha\mu}{2}+\theta\mu-\frac{\beta}{\tau_1}}{\eta_x+\mu(1-\theta)}$ and $P=\eta_x I-\beta A^\top A$, it follows that
\begin{equation}\label{eq:lin-convg-x-term}
\eta\|x^{k+1}- x^*\|^2_{P+\mu(1-\theta)I}
\le \|x^{k+1}- x^*\|^2_{P+(\frac{\alpha\mu}{2}+\mu)I-\frac{\beta}{\tau_1}A^\top A}.
\end{equation}
Similarly, since  
$$\left(\frac{3\alpha\nu}{2}-8c L_h^2\right)I\succeq \frac{\frac{3\alpha\nu}{2}-8c L_h^2-\frac{\beta(1-\theta)}{\tau_2}}{\eta_y+\frac{\beta(1-\theta)}{\tau_2}}(\eta_y I-\beta B^\top B) +\frac{\frac{3\alpha\nu}{2}-8c L_h^2-\frac{\beta(1-\theta)}{\tau_2}}{\eta_y+\frac{\beta(1-\theta)}{\tau_2}}\frac{\beta(1-\theta)}{\tau_2}I+\frac{\beta(1-\theta)}{\tau_2}I,$$
$Q=\eta_y I-\beta B^\top B$, and $B^\top B \preceq I$, we have
\begin{equation}\label{eq:lin-convg-y-term}
\eta\|y^{k+1}- y^*\|^2_{Q+\frac{\beta(1 -\theta)}{\tau_2}B^\top B}
\le \|y^{k+1}- y^*\|^2_{Q+(\frac{3\alpha\nu}{2}-8cL_h^2)I}.
\end{equation}
For the $r$-term, we note from the definition of $\eta$ that 
$$\eta\frac{ \beta(1-\theta)}{2}\le \big(\frac{ \beta(1-\theta)}{2}+\frac{1}{\gamma}\big)-\left(c\rho^2\big(\kappa+2(1-\theta)(1+\frac{1}{\delta})\big)+2c(\beta-\rho)^2\right).$$
In addition, since $\|B\|_2\le 1$, it holds $\|B^\top r^{k+1}\|\le \|r^{k+1}\|$, and thus
\begin{equation}\label{eq:lin-convg-r-term}\eta\frac{ \beta(1-\theta)}{2}\|r^{k+1}\|^2\le \big(\frac{ \beta(1-\theta)}{2}+\frac{1}{\gamma}\big)\|r^{k+1}\|^2-\left(c\rho^2\big(\kappa+2(1-\theta)(1+\frac{1}{\delta})\big)+2c(\beta-\rho)^2\right) \|B^\top r^{k+1}\|^2.
\end{equation}
Finally, it is obvious to have
\begin{equation}\label{eq:lin-convg-lam-term}
\begin{aligned}
&~ \frac{\eta}{2\rho}\left[\|\lambda^{k+1}-\lambda^*\|^2-(1-\theta)\|\lambda^{k}-\lambda^*\|^2+\frac{1}{\theta}\|\lambda^{k+1}-\lambda^k\|^2\right]\\
\le &~\left(\frac{1}{2\rho}+\frac{c}{2}\sigma_{\min}(BB^\top)\right)\left[\|\lambda^{k+1}-\lambda^*\|^2-(1-\theta)\|\lambda^{k}-\lambda^*\|^2+\frac{1}{\theta}\|\lambda^{k+1}-\lambda^k\|^2\right].
\end{aligned}
\end{equation}

Therefore, we obtain \eqref{eq:lin-cvg-ineq} by the definition of $\psi$ and adding \eqref{eq:lin-convg-Phi-term} through \eqref{eq:lin-convg-lam-term}.

\subsection{Proof of Lemma \ref{lem:bd-lam-by-y}}\label{sec:pf-lem-c1}
Let
$
\tilde{\lambda}^{k+1}=\lambda^k-\rho(Ax^{k+1}+B\tilde{y}^{k+1}-b).
$
Then from the update of $y$, we have
\begin{equation}\label{eq:lemmac2-lam-ineq}
\begin{aligned}
&~\EE\|B^\top(\lambda^{k+1}-\lambda^*)\|^2\\
=&~\theta\EE\|B^\top(\tilde{\lambda}^{k+1}-\lambda^*)\|^2+(1-\theta)\EE\|B^\top(\lambda^k-\lambda^*-\rho(Ax^{k+1}+By^k-b))\|^2.
\end{aligned}
\end{equation}

Below we bound the two terms on the right hand side of \eqref{eq:lemmac2-lam-ineq}. First, the definition of $\tilde{\lambda}^{k+1}$ together with \eqref{eq:opt-y} implies
\begin{equation}\label{eq:tlam}
B^\top\tilde{\lambda}^{k+1}=\nabla h(\tilde{y}^{k+1})+Q(\tilde{y}^{k+1}-y^k)+(\beta-\rho)B^\top (Ax^{k+1}+B\tilde{y}^{k+1}-b).
\end{equation}
Hence, by the Young's inequality and the condition in \eqref{kkt2}, we have
\begin{equation}\label{eq:lemmac2-kkt2-ineq}
\begin{aligned}
&~\theta\EE\|B^\top(\tilde{\lambda}^{k+1}-\lambda^*)\|^2\cr
\le&~2\theta\EE\|\nabla h(\tilde{y}^{k+1})-\nabla h(y^*)+Q(\tilde{y}^{k+1}-y^k)\|^2+2\theta(\beta-\rho)^2\EE\|B^\top (Ax^{k+1}+B\tilde{y}^{k+1}-b)\|^2.
\end{aligned}
\end{equation}
Since $\Prob(y^{k+1}=\tilde{y}^{k+1})=\theta$ and $\Prob(y^{k+1}=y^k)=1-\theta$, it follows that
\begin{equation*}
\begin{aligned}
&~\EE\|\nabla h(y^{k+1})-\nabla h(y^*)+Q(y^{k+1}-y^k)\|^2\\
=&~\theta\EE\|\nabla h(\tilde{y}^{k+1})-\nabla h(y^*)+Q(\tilde{y}^{k+1}-y^k)\|^2+(1-\theta)\EE\|\nabla h(y^k)-\nabla h(y^*)\|^2,
\end{aligned}
\end{equation*}
and thus
$$\theta\EE\|\nabla h(\tilde{y}^{k+1})-\nabla h(y^*)+Q(\tilde{y}^{k+1}-y^k)\|^2\le \EE\|\nabla h(y^{k+1})-\nabla h(y^*)+Q(y^{k+1}-y^k)\|^2.$$
Similarly,
$$\theta(\beta-\rho)^2\EE\|B^\top (Ax^{k+1}+B\tilde{y}^{k+1}-b)\|^2\le (\beta-\rho)^2\EE\|B^\top (Ax^{k+1}+B y^{k+1}-b)\|^2.$$
Plugging the above two equations into \eqref{eq:lemmac2-kkt2-ineq} and applying the Young's inequality and also the Lipschitz continuity of $\nabla h$ give
\begin{equation}\label{eq:lemmac2-kkt2-ineq2}
\theta\EE\|B^\top(\tilde{\lambda}^{k+1}-\lambda^*)\|^2
\le4\EE\big[L_h^2\|y^{k+1}-y^*\|^2+\|Q(y^{k+1}-y^k)\|^2\big]+2(\beta-\rho)^2\EE\|B^\top r^{k+1}\|^2.
\end{equation}

In addition, from the Young's inequality, it follows for any $\delta>0$ that
$$\|B^\top(\lambda^k-\lambda^*-\rho(Ax^{k+1}+By^k-b))\|^2\le (1+\delta)\|B^\top(\lambda^k-\lambda^*)\|^2+\rho^2(1+\frac{1}{\delta})\|B^\top(Ax^{k+1}+By^k-b)\|^2.$$
Note $\|B^\top(Ax^{k+1}+By^k-b)\|^2\le 2\|B^\top r^{k+1}\|^2+2\|B^\top B(y^{k+1}-y^k)\|^2$. Therefore, plugging \eqref{eq:lemmac2-kkt2-ineq2} and the above two inequalites into \eqref{eq:lemmac2-lam-ineq}, we complete the proof. 

\subsection{Proof of Lemma \ref{lem:ineq-lam-term}}\label{sec:pf-lem-c2}
It is straightforward to verify
\begin{align*}
&~\|B^\top(\lambda^{k+1}-\lambda^*)\|^2-(1-\theta)(1+\delta)\|B^\top(\lambda^k-\lambda^*)\|^2+\kappa\|B^\top(\lambda^{k+1}-\lambda^k)\|^2\cr
=&~\left[\begin{array}{c}\lambda^{k+1}-\lambda^*\\\lambda^{k+1}-\lambda^k
\end{array}\right]^\top\left[\begin{array}{cc}(1-(1-\theta)(1+\delta))
&(1-\theta)(1+\delta)\\(1-\theta)(1+\delta)&(\kappa-(1-\theta)(1+\delta))\end{array}\right]\otimes BB^\top\left[\begin{array}{c}(\lambda^{k+1}-\lambda^*)\\(\lambda^{k+1}-\lambda^k)
\end{array}\right],
\end{align*}
and
\begin{align*}
&~\left[\begin{array}{c}\lambda^{k+1}-\lambda^*\\\lambda^{k+1}-\lambda^k
\end{array}\right]^\top\left[\begin{array}{cc}\theta  & (1-\theta)\\ (1-\theta) & (\frac{1}{\theta}-(1-\theta))\end{array}\right]\otimes I\left[\begin{array}{c}\lambda^{k+1}-\lambda^*\\ \lambda^{k+1}-\lambda^k
\end{array}\right]\cr
= &~\left[\|\lambda^{k+1}-\lambda^*\|^2-(1-\theta)\|\lambda^{k}-\lambda^*\|^2+\frac{1}{\theta}\|\lambda^{k+1}-\lambda^k\|^2\right].
\end{align*}
Hence, we have the desired result from \eqref{eq:cond-kappa-del} and the inequality $U\otimes V\succeq \sigma_{\min}(V) U\otimes I$ for any PSD matrices $U$ and $V$.

\subsection{Proof of Lemma \ref{lem:three-lemma-together}}
From \eqref{eq:choice-P-mat} and \eqref{eq:choice-Q-mat}, we have
\begin{equation*}
\beta(1 -\theta)\frac{\tau_2}{2}\|A(x^{k+1}-x^k)\|^2\le \frac{1}{2}\|x^{k+1}- x^k\|_{P-L_mI}^2,%\label{eq:ineq-x-term}
\end{equation*}
and
\begin{eqnarray*}
& & 4c\|Q(y^{k+1}-y^k)\|^2+2c\rho^2(1-\theta)(1+\frac{1}{\delta})\|B^\top B(y^{k+1}-y^k)\|^2
+\frac{\beta\tau_1}{2}\|B(y^{k+1}-y^k)\|^2\nonumber \\
&\le & \frac{1}{2}\|y^{k+1}- y^k\|_Q^2.%\label{eq:ineq-y-term}.
\end{eqnarray*}
The desired result is then obtained by adding the above two inequalities together with $\beta$ times of \eqref{cross-xy-term1}, $\beta(1-\theta)$ times of \eqref{cross-xy-term2}, $c$ times of both \eqref{bd-lam-term-n} and \eqref{eq:ineq-lam-term}, and also noting $\lambda^{k+1}-\lambda^k=-\rho r^{k+1}$.

\end{document}

%% file: macros.tex
\usepackage[normalem]{ulem} % to use \sout

\newcommand{\vr}{r}

\newcommand{\vx}{x}
\newcommand{\vy}{y}

\newcommand{\vlam}{\lambda}

\newcommand{\cX}{{\mathcal{X}}}

\newcommand{\EE}{\mathbb{E}} % expectation
\newcommand{\RR}{\mathbb{R}} % real
\newcommand{\vzero}{0} % 0 vector
\newcommand{\Prob}{{\mathrm{Prob}}} % probability
\newcommand{\st}{\mbox{ s.t. }}

\DeclareMathOperator*{\argmin}{arg\,min} % argmin
\newcommand{\bc}{\begin{center}}
\newcommand{\ec}{\end{center}}

\newcommand{\bdm}{\begin{displaymath}}
\newcommand{\edm}{\end{displaymath}}

\newcommand{\beq}{\begin{equation}}
\newcommand{\eeq}{\end{equation}}

\newcommand{\bfl}{\begin{flushleft}}
\newcommand{\efl}{\end{flushleft}}

\newcommand{\bt}{\begin{tabbing}}
\newcommand{\et}{\end{tabbing}}

\newcommand{\beqn}{\begin{eqnarray}}
\newcommand{\eeqn}{\end{eqnarray}}

\newcommand{\beqs}{\begin{align*}} % no equation numbers
\newcommand{\eeqs}{\end{align*}}  % no equation numbers

\newtheorem{remark}{Remark}[section]
\newtheorem{assumption}{Assumption}